\documentclass{article}
\usepackage[paper=a4paper, top=2.3cm, bottom=2.3cm, left=2.2cm, right=2.2cm]{geometry}
\usepackage{url}
\usepackage{amsfonts}
\usepackage{amsmath}
\usepackage{amsthm}
\usepackage{amssymb}
\usepackage{latexsym}
\usepackage{stmaryrd}
\usepackage{graphicx}
\usepackage{multirow, longtable}
\usepackage{graphicx}
\usepackage{xspace}
\usepackage{verbatim}
\usepackage{ifthen}
\usepackage{changebar}
\usepackage{hyperref}

\usepackage{tikz}
\usepackage{tikzfig}    
\usepackage{tikzcdiag}  

\usetikzlibrary{trees}
\usetikzlibrary{topaths}
\usetikzlibrary{decorations.pathmorphing}
\usetikzlibrary{decorations.markings}
\usetikzlibrary{snakes,matrix,backgrounds,folding}
\usetikzlibrary{chains,scopes,positioning,fit}
\usetikzlibrary{arrows,shadows}
\usetikzlibrary{calc} 
\usetikzlibrary{chains}
\usetikzlibrary{shapes,shapes.geometric,shapes.misc}

\pgfsetlayers{background,edgelayer,nodelayer,main}

\newcommand{\markat}{0.55}
\newcommand{\markwithsym}{>}
\newcommand{\markwith}{{\arrow[black]{\markwithsym}}}
\pgfkeys{/tikz/.cd, markat/.store in=\markat, markwith/.store
in=\markwithsym}
\tikzset{decoration={markings, mark=at position \markat with \markwith}}

\tikzstyle{square box}=[rectangle,fill=white,draw=black,minimum height=6mm,minimum width=6mm]
\tikzstyle{diredge}=[postaction=decorate]
\tikzstyle{tvertex}=[circle,inner sep=0.6pt,fill=black,draw=black]
\tikzstyle{outputedge}=[diredge,-)]
\tikzstyle{inputedge}=[diredge,(-]

\tikzstyle{mcgraph}=[baseline=(current bounding box.center), node distance=1.5em and 1.5em, every loop/.style={->, shorten >= 1pt, min distance=8mm}]
\tikzstyle{mcgraph2}=[mcgraph,baseline=(current bounding box.center), node distance=6mm and 8mm]

\tikzstyle{ivert}=[val, rectangle, fill=green!10!white, inner sep=3pt]
\tikzstyle{bndry}=[line width=0.4pt, draw=black, inner sep=0.5pt, circle, dotted, font=\small]
\tikzstyle{sitev}=[line width=0.4pt, draw=black, fill=white, inner sep=0.5pt, circle, dotted, font=\small]

\tikzstyle{ellipses}=[font=\small]
\tikzstyle{brace}=[font=\large]

\tikzstyle{elabel}=[inner sep=1pt, font=\small]
\tikzstyle{ipoint}=[fill=black, inner sep=1pt, circle,font=\small]
\tikzstyle{ipointlabel}=[node distance=0em and 0em,font=\small]
\tikzstyle{vpoint}=[rectangle,line width=0.4pt, draw=black, fill=green!10!white]

\tikzstyle{onegate}=[node distance=1.5mm and 4mm]
\tikzstyle{big}=[node distance=2mm and 6mm]

\tikzstyle{copyg}=[draw=black, fill=white, line width=2pt, inner sep=2pt, circle]
\tikzstyle{pregate}=[line width=1pt, inner sep=1mm, font=\small]
\tikzstyle{gate}=[pregate, draw=green!50!black, fill=green!10!white]
\tikzstyle{val}=[fill=white, inner sep=1mm, line width=0.2pt, draw=black, font=\small]
\tikzstyle{andg}=[gate, rectangle]
\tikzstyle{notg}=[gate, circle, inner sep=1pt]
\tikzstyle{valg}=[val, circle, inner sep=1pt]
\tikzstyle{subg}=[val, inner sep=3pt, rectangle]
\tikzstyle{esubg}=[subg, line width=1pt, dotted]
\tikzstyle{defg}=[val, line width=1pt, inner sep=3pt, draw=blue!50!black, double, rectangle]
\tikzstyle{ignoreg}=[copyg]
\tikzstyle{halfe}=[outer sep=1mm, inner sep=0mm, font=\small]
\tikzstyle{evar}=[font=\small]
\tikzstyle{line}=[-,rounded corners=1ex,line width=0.5pt,font=\small]
\tikzstyle{dline}=[->,>=stealth,rounded corners=1ex,line width=0.5pt,font=\small]
\tikzstyle{bbox}=[draw=red!50!black,line width=1pt,dotted,rectangle,inner sep=3pt]
\tikzstyle{greenbox}=[rectangle,fill=green!30,draw=green!50!black]
\tikzstyle{bluebox}=[rectangle,fill=blue!30,draw=blue]

\tikzstyle{greybg}=[background rectangle/.style={fill=black!5,draw=black!30,rounded corners=1ex}, show background rectangle]
\tikzstyle{precircuit}=[baseline=(current bounding box.center), node distance=2mm and 4mm]
\tikzstyle{circuit}=[greybg, precircuit, edge/.style={dline}]

\newcommand{\TODO}[1]{%
\typeout{WARNING!!! there is still a TODO left}
\marginpar{\textbf{!TODO: }\emph{#1}}
}

\renewcommand{\TODO}[1]{}
\nochangebars

\newcommand{\catSet}{\ensuremath{\textbf{Set}\xspace}}
\newcommand{\catVect}{\ensuremath{\textbf{Vect}\xspace}}
\newcommand{\catTop}{\ensuremath{\textbf{Top}\xspace}}

\newcommand{\catOGraph}{\ensuremath{\textbf{OGraph}}\xspace}
\newcommand{\catGraph}{\ensuremath{\textbf{Graph}}\xspace}

\newcommand{\slicecat}[2]{#1 / #2}
\newcommand{\catGraphSlice}{\ensuremath{\slicecat{\catGraph}{2_\mathcal{G}}}}
\newcommand{\catTGSlice}{\ensuremath{\slicecat{\catGraph}{T_\mathcal{G}}}}
\newcommand{\catOGraphTG}{\ensuremath{\catOGraph_{\typegraph}}}
\newcommand{\catMonPreCat}{\ensuremath{\textbf{MonPreCat}}\xspace}
\newcommand{\rewriteCat}[1]{\ensuremath{\DCsp(\catOGraphTG)/\!\!/\,#1}}

\newcommand{\typegraph}{\ensuremath{T_{\mathcal G}}}

\newcommand{\mergew}[1]{+_{#1}}
\newcommand{\plugw}[1]{+_{\!#1}^{\!\!*}}

\newlength{\hookrightarrowwidth}
\DeclareMathOperator{\matches}{%
\makebox[\hookrightarrowwidth][c]{$\hookrightarrow$}
\hspace*{-\hookrightarrowwidth}%
\makebox[\hookrightarrowwidth][c]{\raise0.5mm\hbox{$\, ^{\sim}$}}%
}

\DeclareMathOperator{\leftmatches}{%
\makebox[\hookrightarrowwidth][c]{$\hookleftarrow$}
\hspace*{-\hookrightarrowwidth}%
\makebox[\hookrightarrowwidth][c]{\raise0.5mm\hbox{$\, ^{\sim}$}}%
}

\DeclareMathOperator{\notleftmatches}{%
\makebox[\hookrightarrowwidth][c]{$\hookleftarrow$}
\hspace*{-\hookrightarrowwidth}%
\makebox[\hookrightarrowwidth][c]{\raise0.5mm\hbox{$\, ^{\sim}$}}%
\hspace*{-\hookrightarrowwidth}%
\makebox[\hookrightarrowwidth][c]{\raise-0.2mm\hbox{\footnotesize{$/$}}}%
}

\newcommand{\seqcompww}[3]{#1\, ;_{#2} #3}

\newcommand{\exten}[2]{{#1}^{\uparrow #2}}

\newcommand{\injmap}{\hookrightarrow}

\newcommand{\In}{\textrm{In}}
\newcommand{\Out}{\textrm{Out}}
\newcommand{\DCsp}{\textbf{DCsp}}

\newcommand{\cmdrewritesto}{\tikz[baseline=-0.25em] { \draw [-open triangle 45, line width=0.2pt] (0,0) -- (0.5,0); }\,}
\newcommand{\cmdrewriteequiv}{\tikz[baseline=-0.25em] { \draw [open triangle 45-open triangle 45, line width=0.2pt] (0,0) -- node [auto,yshift=-1.2mm] {$*$} (0.7,0); }\,}
\newcommand{\cmdrewritetrans}{\tikz[baseline=-0.25em] { \draw [-open triangle 45, line width=0.2pt] (0,0) -- node [auto,pos=0.3,yshift=-1.2mm] {$*$} (0.5,0); }\,}

\newcommand{\lengthymultimapdot}[1]{%
\begin{tikzpicture}[baseline=(current bounding box.south), node distance=0.5em and 1em,text height=0em, text depth=0ex,inner sep=0em]
  \node[fill=white, line width=0.2pt, draw=black, inner sep=1.5pt, circle] (b) {};
  \path[-] ($(b) - (#1,0)$) edge (b);
\end{tikzpicture}%
}

\DeclareMathOperator{\multimapdot}{\lengthymultimapdot{1em}}
\DeclareMathOperator{\longmultimapdot}{\lengthymultimapdot{1.7em}}

\DeclareMathOperator{\rewritesto}{\cmdrewritesto}
\DeclareMathOperator{\rewriteequiv}{\cmdrewriteequiv}
\DeclareMathOperator{\rewritetrans}{\cmdrewritetrans}

\newcommand{\rwarrow}{\multimapdot}
\newcommand{\longrwarrow}{\longmultimapdot}
\newcommand{\rwrulew}[3]{#1 \multimapdot_{#2} #3}
\newcommand{\rwrule}[2]{#1 \multimapdot #2}
\newcommand{\rwsubstw}[3]{#1 \multimapdot_{#2} #3}
\newcommand{\rwsubst}[2]{#1 \multimapdot #2}

\newtheorem{theorem}{Theorem}[section]

\newtheorem{lemma}[theorem]{Lemma}

\newtheorem{conjecture}[theorem]{Conjecture}

\newtheorem{definition}[theorem]{Definition}
\newtheorem{definitions}[theorem]{Definitions}

\newtheorem{example}[theorem]{Example}

\newtheorem{remark}[theorem]{Remark}

\newcommand{\institute}[1]{\\#1}
\newcommand{\email}[1]{\\\tt{#1}}

\title{Open Graphs and Monoidal Theories}

 \author{Lucas Dixon
 \institute{University of Edinburgh}
 \email{ldixon@inf.ed.ac.uk}
 \and
 Aleks Kissinger
 \institute{University of Oxford}
 \email{alexander.kissinger@comlab.ox.ac.uk} }

\date{Draft: \today \thanks{ This research was funded by EPSRC grant
    EPE/005713/1 and by a Clarendon Studentship. We would also like to
    thank Ross Duncan and Matvey Soloviev for their their many helpful
    discussions and remarks on this work.}}

\begin{document}

\maketitle

\begin{abstract}
String diagrams are a powerful tool for reasoning about physical processes, logic circuits, tensor networks, and many other compositional structures. The distinguishing feature of these diagrams is that edges need not be connected to vertices at both ends, and these unconnected ends can be interpreted as the inputs and outputs of a diagram. In this paper, we give a concrete construction for string diagrams using a special kind of typed graph called an \emph{open-graph}. While the category of open-graphs is not itself adhesive, we introduce the notion of a \emph{selective adhesive functor}, and show that such a functor embeds the category of open-graphs into the ambient adhesive category of typed graphs. Using this functor, the category of open-graphs inherits ``enough adhesivity'' from the category of typed graphs to perform double-pushout (DPO) graph rewriting. A salient feature of our theory is that it ensures rewrite systems are ``type-safe'' in the sense that rewriting respects the inputs and outputs. This formalism lets us safely encode the interesting structure of a computational model, such as evaluation dynamics, with succinct, explicit rewrite rules, while the graphical representation absorbs many of the tedious details. Although topological formalisms exist for string diagrams, our construction is discreet, finitary, and enjoys decidable algorithms for composition and rewriting. We also show how open-graphs can be parametrised by graphical signatures, similar to the monoidal signatures of Joyal and Street, which define types for vertices in the diagrammatic language and constraints on how they can be connected. Using typed open-graphs, we can construct free symmetric monoidal categories, PROPs, and more general monoidal theories. Thus open-graphs give us a handle for mechanised reasoning in monoidal categories.
\end{abstract}

\section{Introduction}

Graphs are often used for specification and reasoning, both formally and informally.  They have both an appealing visual nature as well as the ability to naturally abstract structure. In this paper, we will focus on ``string diagrams'', the graphical structures that arise in monoidal theories. Well known examples include proof-nets in linear logic \cite{Girard:proof-nets:96}, Penrose's tensor notation~\cite{Penrose1971Applications-of}, Feynman diagrams, diagrammatic notations for logic circuits, and high level languages for quantum information processing~\cite{Coecke:2008jo}. A common feature of these graphical languages is that they can be understood as describing a computational process, and they support reasoning by manipulating the graphical presentation. However, such manipulation is both tedious and error prone to do by hand. In this paper, we address this difficulty by providing a generic, but also concrete and computable, account of graphical reasoning in monoidal-theories. Our long-term goal is to support automation for graphical reasoning about computational structures.

The main concept we introduce is a formal theory of \emph{open-graphs}. Like graph-based drawings of circuits, the visual presentation of open-graphs consists of vertices connected by edges. Crucially, edges in an open-graph need not be attached to vertices. They may be unconnected at one or both ends, or even connected to themselves to form a ``circle''. In terms of a computational process, the unconnected ends of edges represent the inputs and outputs of a process. A diagram in this graphical language is interpreted as a compound computation with vertices as the atomic operations and wires defining the flow of information. For example, an electronic circuit that defines the compound logical operation of an or-gate, using not-gates around an and-gate, can be drawn as:

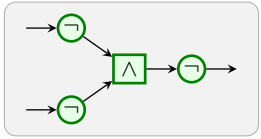
\begin{center}
\begin{tikzpicture}[circuit]
\node (g1) [andg] {$\land$};
\node (x2) [notg, above left=of g1] {$\lnot$};
\node (x1) [notg, below left=of g1] {$\lnot$};
\node (i1) [left=of x1, halfe] {};
\node (i2) [left=of x2, halfe] {};
\node (n3) [notg, right=of g1] {$\lnot$};
\node (o1) [halfe, right=of n3] {};
\path[dline] 
  (i1) edge (x1)
  (i2) edge (x2)
  (x1) edge (g1)
  (x2) edge (g1)
  (g1) edge (n3);
\path[dline] 
  (n3) edge (o1)
;
\end{tikzpicture}
\end{center}

Open-graphs have a rich compositional structure and a convenient algebraic language. We introduce methods for plugging graphs together, merging over common subgraphs, and cutting out pieces of a graph. Using these tools, we develop rewriting for open-graphs. In this regard, our formalism functions analogously to a type-system in a programming language: we ensure that the interface of a process is maintained by rewriting.  In particular, we show that rewriting also has a compositional nature: the decomposition of graphs by cutting their edges enables rewriting to be performed in parallel on the separated components, with a guarantee that the separate rewritten parts can be recomposed appropriately. Moreover, the compositional properties of open-graphs allow rewrite rules themselves to be rewritten using the same machinery.

To formalise the process of rewriting, we use a well-behaved embedding of the category of open-graphs into its ambient category of typed graphs. This embedding is an instance of a more general notion which we introduce as \emph{selective adhesive functors}. In particular, these functors reflect pushouts, so many results about pushouts in an adhesive category are true of so-called \emph{adhesive pushouts}, i.e. the pushouts reflected by a selective adhesive functor.

We also parameterise the category of open-graphs by a \emph{graphical signature}. This defines a collection of vertex and edge types and assigns to each vertex type its input and output types. We construct a type graph from such a signature and form the category of \emph{typed open-graphs} by slicing over this type graph. Combined with a collection of graphical rules, these typed open-graphs provide a formal way to reason with a graphical theory of some algebraic or dynamical system. We demonstrate the generality of our construction by showing that typed open-graphs can be used to construct free symmetric monoidal categories, PROPs, and a wide range of more general monoidal theories. Unlike many other (topological) constructions for diagrammatic accounts of monoidal categories, our construction involves finite data. Thus our construction enables the development of software tools that work with graphical theories. In particular, it provides the basis for employing techniques from automated reasoning, such as completion-based methods~\cite{knuth-bendix},  to mechanise working with string diagrams.

The rest of the paper is structured as follows.  In section \ref{sec:motiv}, we introduce and motivate graphical theories with boolean circuits and tensor networks. We also note key challenges in working with these systems using traditional graph-based methods. After reviewing some of these methods in section \ref{sec:related}, we define selective adhesive functors in section \ref{sec:adhesive-functors}. These give an abstract characterisation for categories that sit inside an ambient adhesive category, and inherit enough properties to support rewriting. We define open-graphs in section \ref{sec:open-graphs} and show that they have a selective adhesive functor into a slice category over $\catGraph$. In section \ref{sec:compose-graphs}, we demonstrate how open-graphs can be composed and decomposed, and use these operations for rewriting open-graphs in section \ref{sec:rewriting}. Section \ref{sec:graph-lang} defines graphical signatures, and shows how they can be used to construct typed open-graphs. Section \ref{sec:monoidal-theories} uses typed open-graphs to construct a monoidal category of cospans, and shows how such categories correspond to the free constructions of monoidal categories over a graphical signature. We also show how PROPs can be defined in this language. Finally, we conclude and discuss future work in section \ref{sec:conclusions}.


\section{Motivating Examples}
\label{sec:motiv}

We introduce two examples here to motivate the use of open-graphs for
computation. The first is the familiar language of boolean
circuits. Boolean circuits are formed by taking basic logic gates and
plugging them together. For instance, we can represent the logical
expression ``$a \land (b \land \lnot c) $'' as the graph:

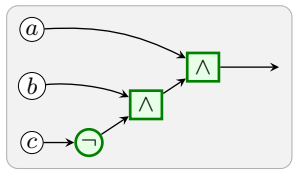
\begin{center}
\beginpgfgraphicnamed{example_circuit}
\begin{tikzpicture}[circuit]
	\begin{pgfonlayer}{nodelayer}
		\node [style=valg] (0) at (-1.5, 1.5) {$a$};
		\node [style=andg] (1) at (0.75, 1) {$\land$};
		\node [style=none] (2) at (1.75, 1) {};
		\node [style=valg] (3) at (-1.5, 0.75) {$b$};
		\node [style=andg] (4) at (0, 0.5) {$\land$};
		\node [style=valg] (5) at (-1.5, 0) {$c$};
		\node [style=notg] (6) at (-0.75, 0) {$\lnot$};
	\end{pgfonlayer}
	\begin{pgfonlayer}{edgelayer}
		\draw[dline] (6) to (4);
		\draw[dline] (4) to (1);
		\draw[dline] (1) to (2.center);
		\draw[dline] (5) to (6);
		\draw[bend left=15, dline] (0) to (1);
		\draw[bend left=15, looseness=0.75, dline] (3) to (4);
	\end{pgfonlayer}
\end{tikzpicture}}
\endpgfgraphicnamed
\end{center}

Notice that the output wire of this graph does not end at a vertex. We
call this a \emph{half-edge}. We can also represent inputs to a
circuit as half-edges. In the above example, this removes the need to
introduce the variables $a$, $b$, and $c$ as inputs to the
circuit. Instead, we represent the inputs as half-edges:

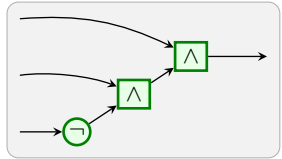
\begin{center}
\beginpgfgraphicnamed{example_circuit_no_binding}
\begin{tikzpicture}[circuit]
	\begin{pgfonlayer}{nodelayer}
		\node [style=none] (0) at (-1.5, 1.5) {};
		\node [style=andg] (1) at (0.75, 1) {$\land$};
		\node [style=none] (2) at (1.75, 1) {};
		\node [style=none] (3) at (-1.5, 0.75) {};
		\node [style=andg] (4) at (0, 0.5) {$\land$};
		\node [style=none] (5) at (-1.5, 0) {};
		\node [style=notg] (6) at (-0.75, 0) {$\lnot$};
	\end{pgfonlayer}
	\begin{pgfonlayer}{edgelayer}
		\draw[dline] (4) to (1);
		\draw[dline] (5.center) to (6);
		\draw[bend left=15, looseness=0.75, dline] (3.center) to (4);
		\draw[dline] (1) to (2.center);
		\draw[dline] (6) to (4);
		\draw[bend left=15, dline] (0.center) to (1);
	\end{pgfonlayer}
\end{tikzpicture}}
\endpgfgraphicnamed
\end{center}

Now, suppose we wanted to introduce an expression like ``$a \land
(\lnot a \land b)$''. We can do this without introducing explicitly named variables by introducing a ``copy'' operation.

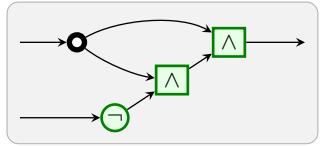
\begin{center}
\beginpgfgraphicnamed{example_circuit_with_copy}
\begin{tikzpicture}[circuit]
	\begin{pgfonlayer}{nodelayer}
		\node [style=none] (0) at (-1.25, 1) {};
		\node [style=copyg] (1) at (-0.5, 1) {};
		\node [style=andg] (2) at (1.5, 1) {$\land$};
		\node [style=none] (3) at (2.5, 1) {};
		\node [style=andg] (4) at (0.75, 0.5) {$\land$};
		\node [style=none] (5) at (-1.25, 0) {};
		\node [style=notg] (6) at (0, 0) {$\lnot$};
	\end{pgfonlayer}
	\begin{pgfonlayer}{edgelayer}
		\draw[dline] (2) to (3.center);
		\draw[dline] (4) to (2);
		\draw[dline] (6) to (4);
		\draw[out=30, looseness=0.75, in=150, dline] (1) to (2);
		\draw[dline] (5.center) to (6);
		\draw[bend right=15, looseness=0.75, dline] (1) to (4);
		\draw[dline] (0.center) to (1);
	\end{pgfonlayer}
\end{tikzpicture}}
\endpgfgraphicnamed
\end{center}

We can also introduce an explicit ``ignore'' operation that takes on input and produces no output. To sum up, our language has the following generators, where $b$ is a boolean value.
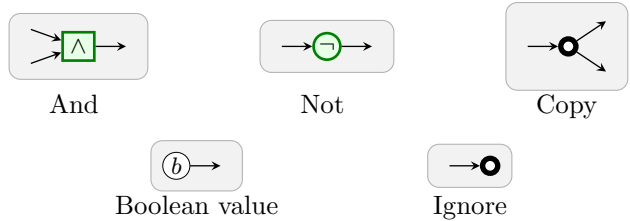
\begin{center}
\begin{tabular}{ccccc}
{
\begin{tikzpicture}[circuit, onegate]
\node (g) [andg] {$\land$};
\node (i1) [halfe, above left=of g.west] {}; 
\node (i2) [halfe, below left=of g.west] {};
\node (o1) [halfe, right=of g] {};
\draw [dline] (i1) -- (g);
\draw [dline] (i2) -- (g);
\draw [dline] (g) -- (o1);
\end{tikzpicture}
}& \hspace{0.5cm} & 
{
\begin{tikzpicture}[circuit]
\node (g) [notg] {$\lnot$};
\node (i1) [halfe, left=of g] {}; 
\node (o1) [halfe, right=of g] {};
\draw [dline] (i1) -- (g);
\draw [dline] (g) -- (o1);
\end{tikzpicture}
}& \hspace{0.5cm} & 
{
\begin{tikzpicture}[circuit]
\node (g) [copyg] {};
\node (i1) [halfe, left=of g] {};
\node (o1) [halfe, above right=of g] {}; 
\node (o2) [halfe, below right=of g] {};
\draw [dline] (i1) -- (g);
\draw [dline] (g) -- (o1);
\draw [dline] (g) -- (o2);
\end{tikzpicture}
} \\
And & & Not & & Copy
\end{tabular}

\medskip

\begin{tabular}{ccc}
{
\begin{tikzpicture}[circuit]
\node (g) [valg] {$b$};
\node (o1) [halfe, right=of g] {};
\draw [dline] (g) -- (o1);
\end{tikzpicture}
}
&
\hspace{1cm}
& 
{
\begin{tikzpicture}[circuit]
\node (g) [ignoreg] {};
\node (i1) [halfe, left=of g] {};
\draw [dline] (i1) -- (g);
\end{tikzpicture}
}
\\
Boolean value & & Ignore
\end{tabular}
\end{center}

Copies of these components can then be connected together by joining outputs to inputs to form compound circuits. While this is a simple language, it includes satisfiability questions, which are formed by asking whether a given graph can be rewritten to the single boolean value $T$. To answer such questions, and more generally to describe the dynamics of boolean circuits, some axioms need to be introduced. For copying and ignoring values, these are:

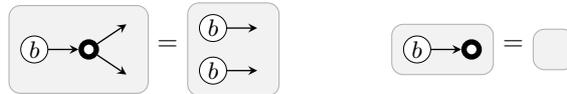
\begin{center}
\mbox{
\begin{tikzpicture}[circuit]
\node (g) [copyg] {};
\node (i1) [valg, left=of g] {$b$};
\node (o1) [halfe, above right=of g] {}; 
\node (o2) [halfe, below right=of g] {};
\draw [dline] (i1) -- (g);
\draw [dline] (g) -- (o1);
\draw [dline] (g) -- (o2);
\end{tikzpicture}
 =
\begin{tikzpicture}[circuit]
\node (i1) [valg] {$b$};
\node (i2) [valg, below=of i1] {$b$};
\node (o1) [halfe, right=of i1] {}; 
\node (o2) [halfe, right=of i2] {};
\draw [dline] (i1) -- (o1);
\draw [dline] (i2) -- (o2);
\end{tikzpicture}
}
\hspace{1cm}
\mbox{
\begin{tikzpicture}[circuit]
\node (i1) [valg] {$b$};
\node (o1) [ignoreg, right=of i1] {};
\draw [dline] (i1) -- (o1);
\end{tikzpicture}
 =
\begin{tikzpicture}[circuit]
\node {\ }; 
\end{tikzpicture}
}
\end{center}

\noindent The axioms for conjunction (and-gates: $\land$) and negation
(not-gates: $\lnot$) are:

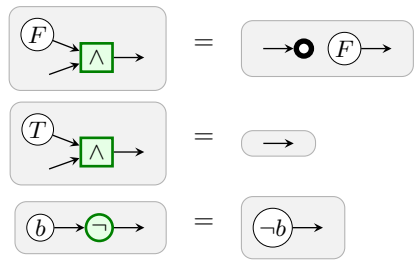
\begin{eqnarray*}
\begin{tikzpicture}[circuit, onegate]
\node (g) [andg] {$\land$};
\node (i1) [valg, above left=of g.west] {$F$}; 
\node (i2) [halfe, below left=of g.west] {};
\node (o1) [halfe, right=of g] {};
\draw [dline] (i1) -- (g);
\draw [dline] (i2) -- (g);
\draw [dline] (g) -- (o1);
\end{tikzpicture}
& = &
\begin{tikzpicture}[circuit, onegate]
\node (v1) [valg] {$F$}; 
\node (o1) [halfe, right=of v1] {};
\node (o2) [ignoreg, left=of v1.center] {};
\node (i2) [halfe, left=of o2] {};
\draw [dline] (v1) -- (o1);
\draw [dline] (i2) -- (o2);
\end{tikzpicture}
\\
\begin{tikzpicture}[circuit,onegate]
\node (g) [andg] {$\land$};
\node (i1) [valg, above left=of g.west] {$T$}; 
\node (i2) [halfe, below left=of g.west] {};
\node (o1) [halfe, right=of g] {};
\draw [dline] (i1) -- (g);
\draw [dline] (i2) -- (g);
\draw [dline] (g) -- (o1);
\end{tikzpicture}
& = &
\begin{tikzpicture}[circuit]
\node (v1) [halfe] {};
\node (o1) [halfe, right=of v1] {};
\draw [dline] (v1) -- (o1);
\end{tikzpicture}
\\
\begin{tikzpicture}[circuit]
\node (v1) [valg] {$b$};
\node (g1) [notg, right=of v1] {$\lnot$}; 
\node (o1) [halfe, right=of g1] {};
\draw [dline] (v1) -- (g1);
\draw [dline] (g1) -- (o1);
\end{tikzpicture}
& = & 
\begin{tikzpicture}[circuit]
\node (v1) [valg] {$\lnot b$};
\node (o1) [halfe, right=of v1] {};
\draw [dline] (v1) -- (o1);
\end{tikzpicture}
\end{eqnarray*}

These rules characterise the computational aspects of boolean circuits. Applying the axioms from left to right can be used to evaluate the output of a circuit. The equations can also be used to simplify circuits.


Although the above rules are sufficient for evaluation (when a circuit
has all inputs given), they cannot prove all true equations about
boolean circuits. To get a complete set of equations, some additional
graphical rules are needed. For instance, the following rule, for double negation elimination, is not
directly derivable from those presented earlier:
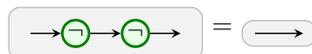
\begin{center}
\mbox{
\begin{tikzpicture}[circuit]
\node (i1) [halfe] {};
\node (g1) [notg, right=of i1] {$\lnot$}; 
\node (g2) [notg, right=of g1] {$\lnot$}; 
\node (o1) [halfe, right=of g2] {};
\draw [dline] (i1) -- (g1);
\draw [dline] (g1) -- (g2);
\draw [dline] (g2) -- (o1);
\end{tikzpicture}
=
\begin{tikzpicture}[circuit]
\node (i1) [halfe] {};
\node (o1) [halfe, right=of i1] {};
\draw [dline] (i1.center) -- (o1.center);
\end{tikzpicture}
}
\end{center}
However, verification of such circuits can be done by exhaustive analysis directly in the graphical language: we can evaluate every combination of inputs to a graphical equation to see if the left- and right-hand sides always evaluate to the same result. This corresponds to a proof by exhaustive case analysis, much like verification by truth-tables.

Once there are sufficient equations, new rules can also be derived directly, without examining all cases. For example, using the double-negation equation above with the evaluation axioms, allows the following derivation:

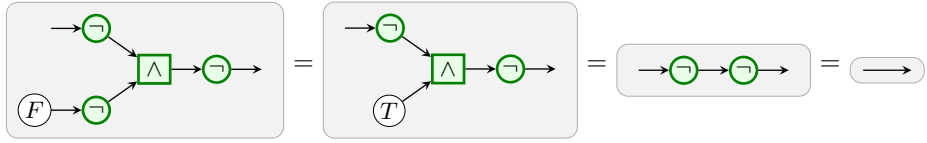
\begin{center}
\begin{tikzpicture}[circuit]
\node (g1) [andg] {$\land$};
\node (x2) [notg, above left=of g1] {$\lnot$};
\node (x1) [notg, below left=of g1] {$\lnot$};
\node (i1) [left=of x1, valg] {$F$};
\node (i2) [left=of x2, halfe] {};
\node (n3) [notg, right=of g1] {$\lnot$};
\node (o1) [halfe, right=of n3] {};
\path[dline] 
  (i1) edge (x1)
  (i2) edge (x2)
  (x1) edge (g1)
  (x2) edge (g1)
  (g1) edge (n3);
\path[dline] 
  (n3) edge (o1)
;
\end{tikzpicture}
$=$
\begin{tikzpicture}[circuit]
\node (g1) [andg] {$\land$};
\node (x2) [notg, above left=of g1] {$\lnot$};
\node (x1) [valg, below left=of g1] {$T$};
\node (i2) [left=of x2, halfe] {};
\node (n3) [notg, right=of g1] {$\lnot$};
\node (o1) [halfe, right=of n3] {};
\path[dline] 
  (i2) edge (x2)
  (x1) edge (g1)
  (x2) edge (g1)
  (g1) edge (n3);
\path[dline] 
  (n3) edge (o1)
;
\end{tikzpicture}
$=$
\begin{tikzpicture}[circuit]
\node (i1) [halfe] {};
\node (g1) [notg, right=of i1] {$\lnot$}; 
\node (g2) [notg, right=of g1] {$\lnot$}; 
\node (o1) [halfe, right=of g2] {};
\draw [dline] (i1) -- (g1);
\draw [dline] (g1) -- (g2);
\draw [dline] (g2) -- (o1);
\end{tikzpicture}
$=$
\begin{tikzpicture}[circuit]
\node (i1) [halfe] {};
\node (o1) [halfe, right=of i1] {};
\path[dline] (i1.center) edge (o1.center);
\end{tikzpicture}
\end{center}

\noindent This proves that giving $F$ to the compound or-gate is the same as the identity on the other input. Such derivations can be exponentially shorter than case-analysis. Moreover, rules in a derivation can simultaneously be applied to separate parts of a graph to parallelise a computation or derivation.


Another salient feature of graph-based representations is that certain aspects of sharing and binding can be described using graphical structure. For example, consider the following rule:

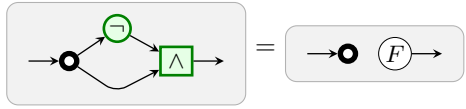
\begin{center}
\mbox{
\begin{tikzpicture}[circuit]
\node (i1) [halfe] {};
\node (g1) [copyg, right=of i1] {};
\node (x2) [halfe, below right=of g1] {};
\node (x1) [notg, above right=of g1] {$\lnot$};
\node (x) [right=of g1] {};
\node (g2) [andg, right=of x] {$\land$};
\node (o1) [halfe, right=of g2] {};
\draw [dline] (i1) -- (g1);
\draw [dline] (g1) -- (x1);
\draw [dline] (x1) -- (g2);
\draw [dline] (g1) -- (x2.center) -- (g2);
\draw [dline] (g2) -- (o1);
\end{tikzpicture}
=
\begin{tikzpicture}[circuit]
\node (i2) [halfe] {};
\node (o2) [ignoreg, right=of i2] {};
\node (v1) [valg, right=of o2.center] {$F$}; 
\node (o1) [halfe, right=of v1] {};
\draw [dline] (v1) -- (o1);
\draw [dline] (i2) -- (o2);
\end{tikzpicture}
}
\end{center}

\noindent With a formula-based notation this could be described by an equation between lambda-terms: ``$\lambda x.\ ((\lnot x) \land x) = \lambda x.\ F$''. Graphical notation can treat certain forms of binding by the structure of edges with function application of formula corresponding to composition along half-edges. For example consider applying the left hand side of the equation to the term $F$, giving the lambda-term ``$\lambda x.\ ((\lnot x) \land x)\ F$''.  In this situation, beta-reduction, which reduces the formula to ``$(\lnot F) \land F$'', corresponds to an application of the copying rule. In the graphical language, the beta-reduction step is:

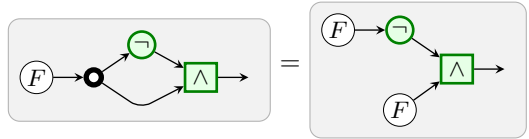
\begin{center}
\mbox{
\begin{tikzpicture}[circuit]
\node (i1) [valg] {$F$}; 
\node (g1) [copyg, right=of i1] {};
\node (x2) [halfe, below right=of g1] {};
\node (x1) [notg, above right=of g1] {$\lnot$};
\node (x) [right=of g1] {};
\node (g2) [andg, right=of x] {$\land$};
\node (o1) [halfe, right=of g2] {};
\draw [dline] (i1) -- (g1);
\draw [dline] (g1) -- (x1);
\draw [dline] (x1) -- (g2);
\draw [dline] (g1) -- (x2.center) -- (g2);
\draw [dline] (g2) -- (o1);
\end{tikzpicture}
=
\begin{tikzpicture}[circuit]
\node (x) {};
\node (x1) [notg, above=of x] {$\lnot$};
\node (i1) [valg, left=of x1] {$F$};
\node (x2) [valg, below=of x] {$F$};
\node (g2) [andg, right=of x] {$\land$};
\node (o1) [halfe, right=of g2] {};
\draw [dline] (i1) -- (x1);
\draw [dline] (x1) -- (g2);
\draw [dline] (x2) -- (g2);
\draw [dline] (g2) -- (o1);
\end{tikzpicture}
}
\end{center}

Notice that the graphical representation controls copying carefully: by explicit application of equational rules. This is an essential feature in graphical representations of quantum information, where copying can only happen in restricted situations.

We move now from the familiar case of logic circuit rewriting to an example from linear algebra. In (multi-)linear algebra, differential geometry, and physics, many computations can be performed using networks of \emph{tensors}. A tensor is a set of real or complex numbers, indexed by one or more integers. For example, the following is an $(n_1 \cdot n_2 \cdot n_3)$-dimensional tensor indexed by 3 integers. 
\[ \{ \chi_{i j}^{k} : i = 1..n_1; \ j = 1..n_2; \ k = 1..n_3 \} \]

Tensors are written with subscript indices, which serve the purpose of inputs, and superscript indices which are outputs. Familiar examples of tensors are vectors, $v^i$ and matrices, $M^i_j$. We can compose tensors by \emph{contraction}, i.e. ``summing together'' a lower index and an upper index of the same dimension:
\[ \xi^i_j = \sum_{kl} \chi_{kl}^i \beta_j^{k} \rho^l \]

In order to simplify such expressions, we can use the Einstein summation convention, where any repeated indexes are assumed to be summed over. However, even with this convention, contraction expressions can get quite complex. Consider this expression, involving six tensors:
\begin{equation}\label{eqn:gross-tensor}
  \alpha_{abc}^{de} \beta_{f}^{bfg} \gamma_{dh}^{i}
  \rho_{i}^{h} \phi_{eg}^{jk} \delta_l^l	
\end{equation}

In order to understand this expression, one has to keep track of $11$ indices, which makes computations time-consuming and error-prone. We can instead represent this expression using a graphical language introduced by Penrose \cite{Penrose1971Applications-of}. Tensors are drawn as boxes, and summations over pairs of indices as wires. The ``identity'' tensor (i.e. the Dirac delta $\delta_i^j$) is also drawn as a wire. The un-summed, or ``free'' indices are left as dangling wires, and sums $\sum \delta_i^i$ are represented as circles. In the graphical notion, expression (\ref{eqn:gross-tensor}) becomes the following diagram:

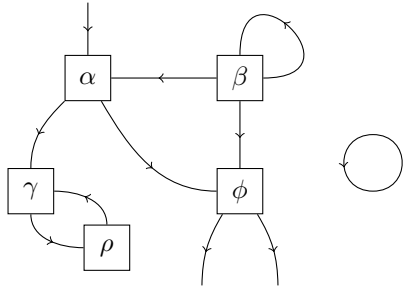
\begin{center}
\beginpgfgraphicnamed{box_diagram}
\begin{tikzpicture}
	\path [use as bounding box] (-2.75,-2) rectangle (3,2.25);
	\begin{pgfonlayer}{nodelayer}
		\node [style=none] (0) at (-1.5, 2) {};
		\node [style=square box] (1) at (-1.5, 1) {$\alpha$};
		\node [style=square box] (2) at (0.5, 1) {$\beta$};
		\node [style=none] (3) at (2.25, 0.25) {};
		\node [style=square box] (4) at (-2.25, -0.5) {$\gamma$};
		\node [style=square box] (5) at (0.5, -0.5) {$\phi$};
		\node [style=none] (6) at (2.25, -0.5) {};
		\node [style=square box] (7) at (-1.25, -1.25) {$\rho$};
		\node [style=none] (8) at (0, -1.75) {};
		\node [style=none] (9) at (1, -1.75) {};
	\end{pgfonlayer}
	\begin{pgfonlayer}{edgelayer}
		\draw[bend left=15, style=diredge] (5) to (9.center);
		\draw[out=-60, style=diredge, in=180] (1) to (5);
		\draw[style=diredge, out=90, in=0] (7) to (4);
		\draw[style=diredge, in=180, out=180, looseness=1.75] (3.center) to (6.center);
		\draw[style=diredge] (0.center) to (1);
		\draw[style=diredge, bend right=15] (5) to (8.center);
		\draw[out=225, style=diredge, in=90] (1) to (4);
		\draw[out=0, looseness=1.75, in=0] (3.center) to (6.center);
		\draw[out=270, style=diredge, in=180] (4) to (7);
		\draw[out=0, style=diredge, in=90, loop] (2) to ();
		\draw[style=diredge] (2) to (5);
		\draw[style=diredge] (2) to (1);
	\end{pgfonlayer}
\end{tikzpicture}}
\endpgfgraphicnamed
\end{center}

These diagrams are called tensor networks. We can then work directly with these graphs, expressing equations of tensor expressions as graph rewrites rules.

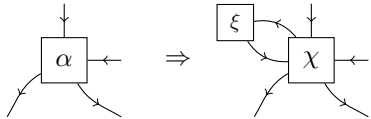
\begin{center}
\beginpgfgraphicnamed{rewrite}
\begin{tikzpicture}
	\path [use as bounding box] (-2.5,-1.25) rectangle (3,1.25);
	\begin{pgfonlayer}{nodelayer}
		\node [style=none] (0) at (-1.5, 0.75) {};
		\node [style=none] (1) at (1.75, 0.75) {};
		\node [style=square box, scale=0.8] (2) at (0.75, 0.5) {$\xi$};
		\node [style=square box] (3) at (-1.5, 0) {$\alpha$};
		\node [style=none] (4) at (-0.75, 0) {};
		\node [style=none] (5) at (0, 0) {$\Rightarrow$};
		\node [style=square box] (6) at (1.75, 0) {$\chi$};
		\node [style=none] (7) at (2.5, 0) {};
		\node [style=none] (8) at (-2.25, -0.75) {};
		\node [style=none] (9) at (-0.75, -0.75) {};
		\node [style=none] (10) at (1, -0.75) {};
		\node [style=none] (11) at (2.5, -0.75) {};
	\end{pgfonlayer}
	\begin{pgfonlayer}{edgelayer}
		\draw[style=diredge] (4.center) to (3);
		\draw[style=diredge, bend right] (2) to (6);
		\draw[style=diredge] (7.center) to (6);
		\draw[out=-60, style=diredge, in=150] (6) to (11.center);
		\draw[style=diredge] (0.center) to (3);
		\draw[out=-60, style=diredge, in=150] (3) to (9.center);
		\draw[style=diredge, bend right] (6) to (2);
		\draw[out=210, style=diredge, in=60] (3) to (8.center);
		\draw[out=210, style=diredge, in=60] (6) to (10.center);
		\draw[style=diredge] (1.center) to (6);
	\end{pgfonlayer}
\end{tikzpicture}}
\endpgfgraphicnamed
\end{center}

More generally, circuit diagrams, tensor networks, and many other graphical formalisms, can be expressed as arrows in some symmetric monoidal category. The diagrams above can then be interpreted as examples of a diagrammatic language common to all symmetric monoidal categories. These kinds of graphical languages introduce a particular challenge to formalising rewriting. For instance, consider a simple graph containing a self loop:  

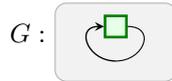
\begin{center}
	$G :$
	\begin{tikzpicture}[circuit, node distance=0.5em and 0.7em]
	\node (x1) [] {};
	\node (g) [andg, right=of x1, inner sep=0.4em] {};
	\node (x2) [right=of g] {};
	\node (x3) [below=of g] {};

	\draw [in=0, out=0, looseness=2.00] (g.east) to (x3.center);
	\draw [line,in=180, out=180, looseness=2.00, shorten <=0 pt, shorten >=-0.1 pt, dline] (x3.center) to (g.west);


	\end{tikzpicture}
\end{center}
and a rewrite rule that rewrites the box to a line: 
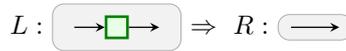
\begin{center}
	\raisebox{-0.25em}{$L:$}
	\begin{tikzpicture}[circuit]
	\node (x1) [halfe] {};
	\node (g) [andg, right=of x1, inner sep=0.4em] {};
	\node (x2) [halfe,right=of g] {};
	\draw [dline] (x1) -- (g);
  \draw [dline] (g) -- (x2);
	\end{tikzpicture}
	\raisebox{-0.25em}{$\Rightarrow\ R:$}
	\begin{tikzpicture}[circuit]
	\node (x1) [halfe] {};
	\node (x2) [halfe,right=of x1] {};
	\draw [dline] (x1.center) -- (x2.center);
	\end{tikzpicture}
\end{center}
Then, the graph resulting from rewriting the box with a self loop should be a circular edge with no vertices:
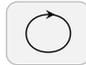
\begin{center}
	\begin{tikzpicture}[circuit]
		\node [style=none] (0) at (0, 0.25) {};
		\node [style=none] (1) at (0, -0.25) {};
		\draw[in=180, out=180, looseness=2.00, shorten >=-2.5 pt, dline] (1.center) to (0.center);
		\draw[in=0, out=0, looseness=2.00, shorten <=0 pt, shorten >=-0.1 pt, line] (0.center) to (1.center);
\end{tikzpicture}}
\end{center}
Graphs of this shape are beyond the normal notion of what one might consider a ``graph'', yet in many contexts, they have a well-behaved interpretation. For instance, in tensor networks, this is the trace of the identity matrix, i.e. the dimension of the underlying vector space.

Suppose we tried na\"ively to formalise this situation, by representing half-edges as edges connected to ``dummy'' points at the boundary.
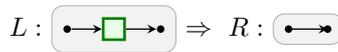
\begin{center}
	\raisebox{-0.25em}{$L:$}
	\begin{tikzpicture}[circuit]
	\node (x1) [ipoint] {};
	\node (g) [andg, right=of x1, inner sep=0.4em] {};
	\node (x2) [ipoint,right=of g] {};
	\draw [dline] (x1) -- (g);
  \draw [dline] (g) -- (x2);
	\end{tikzpicture}
	\raisebox{-0.25em}{$\Rightarrow\ R:$}
	\begin{tikzpicture}[circuit]
	\node (x1) [ipoint] {};
	\node (x2) [ipoint,right=of x1] {};
	\draw [dline] (x1.center) -- (x2.center);
	\end{tikzpicture}
\end{center}
Then, the left hand side of the rewrite does not occur as a subgraph of $G$. So, maybe we could make an exception and not require that $L$ be a subgraph if $G$, but just have \emph{some} mapping on to $G$. If we do this, the box and both dummy points could be mapped on to the box in $G$. However, the result of removing the image of $L$ and replacing it with $R$ is a line, not a circle. A graph that previously had no inputs or outputs is rewritten to an graph with one input and one output, which contradicts the interpretation of rewrite rules representing some kind of ``local'' identity on a diagram. We could make an exception here, but one quickly becomes overwhelmed by the number of special cases that need consideration. We can address this problem uniformly by allowing edge-points. These extra ``dummy'' points can be introduced not only at the boundaries of graphs, but \emph{along} edges as well. This allows rewrites to be performed in a localised manner, without compromising the validity of the graph as a whole.

\section{Related Work}
\label{sec:related}

There is a significant strand of work concerning graph transformations~\cite{Ehrig:Book:2006, Baldan-gts} and rewriting with graph-based presentations of computational processes~\cite{Lafont09diagramrewriting,Lafont08,Lafont2003,Lafont1990}. An extension of these formalisms, known as bigraphs, provides another general formalism for graphical rewriting~\cite{milner2006pure}. Bigraphs are more complex in that they use hyper-graphs and introduce a rich hierarchical structure. Another formalism for graphs, called \emph{site-graphs}, is used in systems biology~\cite{danos2004formal}. These give each vertex a set of `sites' to which edges can be be connected.  The distinction between these forms of graphical rewriting and our formalism is that we have an extended notion of ``graph'' that allows for edges to be dangling at one or both ends, or be connected to themselves. We also consider these graphs as having a fixed interface, drawn as a collection of input and output wires and consider only graph rewrite rules that preserve this interface. In this regard, we provide a kind of static checking for well-behaved graph transformation systems, much like types do for functional programs. This property is crucial to the graphical formalisms of many of the systems we wish to model.  Where our constructions and those of traditional graph transformation share significant similarity is in its reliance on adhesive categories~\cite{Lack:2005lr} and the \emph{double-pushout} construction for graph rewriting~\cite{Ehrig1973}. In addition, our construction uses the presentation of typed graphs as a slice over the (adhesive) category of graphs, as presented in~\cite{Ehrig2008}. In this way, our theory can be viewed as a concrete realisation of the theory of adhesive categories and DPO rewriting, as well as a bridge from this work to the (computational) study of monoidal categories.

Maps in many kinds of monoidal categories admit rich graphical languages~\cite{selinger2009survey}. These languages become particularly interesting when one studies algebraic structures within monoidal categories. A developing field in category theory studies these algebras, and how they interact. \cite{Lack:2004p1160} has shown that a certain class of these monoidal algebras, called \emph{PROPs} can be composed in much the same way Beck showed we can compose monads~\cite{Beck1969}. Even richer notions of interacting graphical structures have found applications in the study of non-commuting observables~\cite{Coecke:2008jo} and entanglement~\cite{CoeckeKissinger2010} in quantum mechanics.

In earlier work, we presented a formalism for reasoning about categorical models of quantum information~\cite{2009:DixonDuncan:AMAI}.  In~\cite{2010:DDK:DCM}, we proposed several improvements on this early work and suggested that matching and composition became dual notions. In this paper, we have clarified the formalism in the context of adhesive categories, proved the key properties, and shown how to construct models of monoidal theories.


%



\section{Selective Adhesive Functors and Rewriting}
\label{sec:adhesive-functors}

Adhesive categories provide a useful and quite general setting for performing rewrites on graph-like structures. The distinguishing characteristic of adhesive categories is that pushouts along monomorphisms behave particularly well with respect to pullbacks. The categories we introduce for open-graphs are not exactly adhesive categories, but they live \emph{inside} of adhesive categories and inherit ``enough adhesivity'' to permit graph rewriting. 

In particular, we introduce categories for open-graphs which are subcategories of slices over the category of graphs (\catGraph{}). Since a slice over an adhesive category is adhesive~\cite{Lack:2005lr} and \catGraph{} is an adhesive category, our categories of open-graphs have inclusions into adhesive categories. To make use of ambient adhesive categories, we define a suitably well-behaved inclusion functor, called a selective adhesive functor. This is well-behaved in the sense that essential adhesivity properties for rewriting can be passed back to the subcategory. To define these functors, we first recall the notion of a van Kampen square.


\begin{definition}\label{def:van-kampen}
	A van Kampen square is a pushout
	\begin{center}
		\posquare{A}{B}{C}{D}{}{}{}{}
	\end{center}
	
	Such that for any commutative cube
	\begin{center}
		\begin{tikzpicture}
		  \matrix (m) [cdiag,row sep=1em,column sep=1em]{
		    & A' & & B' \\
		    C' & & D' & \\
		    & A & & B \\
		    C & & D & \\};
		  \path[arrs]
		    (m-1-2) edge (m-1-4) edge (m-2-1)
		            edge (m-3-2)
		    (m-1-4) edge (m-3-4) edge (m-2-3)
		    (m-2-1) edge [-,line width=4pt,draw=white] (m-2-3)
		            edge (m-2-3) edge (m-4-1)
		    (m-3-2) edge (m-3-4)
		            edge (m-4-1)
		    (m-4-1) edge (m-4-3)
		    (m-3-4) edge (m-4-3)
		    (m-2-3) edge [-,line width=4pt,draw=white] (m-4-3)
		            edge (m-4-3);
		\end{tikzpicture}
	\end{center}
	where the back two faces are pullbacks, the following are equivalent:
	\begin{itemize}
		\item the front two faces are pullbacks
		\item the top face is a pushout
	\end{itemize}
\end{definition}

\begin{definition}\label{def:adhesive}
	\cite{Lack:2005lr}. A category $\mathcal A$ is said to be \emph{adhesive} if
	\begin{enumerate}
		\item $\mathcal A$ has pushouts along monomorphisms,
		\item $\mathcal A$ has pullbacks,
		\item and pushouts along monomorphisms in $\mathcal A$ are van
		      Kampen squares.
	\end{enumerate}
\end{definition}

A crucial property of adhesive categories is that they have unique pushout complements over monomorphisms, when they exist.

\begin{definition}\label{def:pushout-complement}
	A \emph{pushout complement} for a pair of arrows $(b : B \rightarrow K, f : K \rightarrow G)$, is another pair of arrows $(c,g)$ such that
	\begin{center}
		\posquare{B}{K}{G'}{G}{b}{g}{c}{f}
	\end{center}
	is a pushout.
\end{definition}

\begin{lemma}\label{lem:unique-pushout-complements}
	\cite{Lack:2005lr}. If a pair of arrows $(b,f)$, where $b$ is mono, has a pushout complement, it is unique up to isomorphism. That is, for any two pushout complements, $(c,g)$ and $(c',g')$, there exists an isomorphism $\phi$ making the following diagram commute:
	\begin{equation}\label{dia:unique-compl}
		\csquareslant{B}{G'}{G''}{G}{c}{g'}{c'}{g}{\phi}
	\end{equation}
\end{lemma}

In order to define subcategories of adhesive categories, where a selected class of pushout squares has unique pushout complements, we define a selective adhesive functor.

\begin{definition}[Selective adhesive functor]\label{def:adhesive-functor}
	Let $\mathcal C$ be a category and $\mathcal A$ be an adhesive category. A functor $S : \mathcal C \rightarrow \mathcal A$ is called a \emph{selective adhesive functor} if it
	\begin{enumerate}
		\item is faithful,
		\item preserves monomorphisms,
		\item creates isomorphisms,
		\item and reflects pushouts.
	\end{enumerate}
\end{definition}

\begin{definition}[$S$-adhesive spans and pushouts]\label{def:adhesive-span}
	Let $S : \mathcal C \rightarrow \mathcal A$ be a selective adhesive functor. A span $A \overset{f}{\longleftarrow} B \overset{g}{\longrightarrow} C$ in $\mathcal{C}$ is called an \emph{$S$-adhesive span} if it has a pushout, and that pushout is preserved by $S$. Such pushouts are called \emph{$S$-adhesive pushouts}.
\end{definition}

Since $S$ reflects \emph{all} pushouts, we could also define $S$-adhesive spans as spans that have a pushout reflected by $S$.

\begin{definition}[$S$-adhesive pushout complement]\label{def:adhesive-complement}
	An \emph{$S$-adhesive pushout complement} for a pair of arrows $(b,f)$ is a pushout complement, where the following diagram is an $S$-adhesive pushout.
	\begin{center}
		\posquare{B}{K}{G'}{G}{b}{g}{c}{f}
	\end{center}
	The map $b$ is called the \emph{boundary} of $K$ and $c$ is called the \emph{coboundary} of $K$ in $G$. 
\end{definition}

Informally, $G'$ should be thought of as $G$ with $K$ cut out from it, where $b$ identifies boundary of $K$, and the coboundary, $c$, identifies the boundary of where $K$ was cut out from $G$. 

When it is convenient, we shall use the notation $G -_{b,f} K := G'$ to denote the pushout complement defined above. In later sections, the boundary map $b$ will be uniquely defined by $K$, so we shall then write simply $G -_f K$.  Since the categories we are concerned with come with a canonical notion of boundary, we typically only require that the boundary of $K$ be mono; unlike \cite{Ehrig2008}, which requires the induced pushout to satisfy an initiality condition.

\begin{lemma}\label{lem:adhesive-complement}
	If a pair of arrows $(b, f)$, where $b$ is mono, have an $S$-adhesive pushout complement, it is unique up to isomorphism.
\end{lemma}
\begin{proof}
	Let $(c,g)$ and $(c',g')$ be $S$-adhesive pushout complements. Then the following diagrams are pushouts in the adhesive category $\mathcal A$.
	\begin{center}
		\posquare{S B}{S K}{S G'}{S G}{S b}{S g}{S c}{S f}
		\qquad
		\posquare{S B}{S K}{S G''}{S G}{S b}{S g'}{S c'}{S f}
	\end{center}
	
	Since $S$ preserves monos, these are both pushout complements of $(S b, S f)$ for $S b$ mono. So this diagram commutes in $\mathcal A$, for $\phi'$ an isomorphism.
	\begin{center}
		\csquareslant{S B}{S G'}{S G''}{S G}{S c}{S g'}{S c'}{S g}{\phi'}
	\end{center}
	
	Since $S$ creates isomorphisms, there exists an iso $\phi : G' \rightarrow G''$ such that $S \phi = \phi'$. Substituting this map in, we have:
	\begin{center}
		\csquareslant{S B}{S G'}{S G''}{S G}{S c}{S g'}{S c'}{S g}{S \phi}
	\end{center}
	
	Diagram (\ref{dia:unique-compl}) commutes by the faithfulness of $S$.
\end{proof}

\begin{definition}[Rewrite rule]\label{def:adhesive-rewrite-rule}
	A rewrite rule $\rwrulew{L}{b_1,b_2}{R}$ is a span of monomorphisms:
	\[ L \overset{b_1}{\longleftarrow} B \overset{b_2}{\longrightarrow} R \]
\end{definition}

For the sake of conciseness, we will often denote a rewrite rule simply as $\rwrule{L}{R}$, leaving the boundary maps implicit. When we do this, each time we write $\rwrule{L}{R}$, it denotes the same rewrite rule, and in particular, it has the same boundary maps.

\begin{definition}[$S$-matching]\label{def:adhesive-matching}
	For a rewrite rule $\rwrulew{L}{b_1,b_2}{R}$, a monomorphism $m : L \rightarrow G$ is called an \emph{$S$-matching} if $B \overset{b_1}{\longrightarrow} L \overset{m}{\longrightarrow} G$ has an $S$-adhesive pushout complement.
\end{definition}

\begin{definition}[$S$-adhesive rewrite]\label{def:adhesive-rewrite}
\begin{changebar}	Let $\rwrulew{L}{b_1,b_2}{R}$ be a rewrite rule and $m : L \rightarrow G$ be an $S$-adhesive matching. Then for $G'$ the $S$-adhesive pushout complement of $B \overset{b_1}{\longrightarrow} L \overset{m}{\longrightarrow} G$, the following diagram is called an \emph{$S$-adhesive rewrite} if the right hand pushout is $S$-adhesive:
\end{changebar}
	\begin{center}
		\begin{tikzpicture}
			\matrix (m) [cdiag] {
	        L & B  & R \\
	        G & G' & H \\
			};
			\path [arrs]
		  	  	(m-1-2) edge node [swap] {$b_1$} (m-1-1)
		  	  	(m-1-2) edge node {$c$} (m-2-2)
				(m-1-2) edge node {$b_2$} (m-1-3)

		  	  	(m-1-1) edge node [swap] {$m$} (m-2-1)
		  	  	(m-1-2) edge node {} (m-2-2)
		  	  	(m-1-3) edge node {} (m-2-3)

		  	  	(m-2-2) edge node {} (m-2-1)
		  	  	(m-2-2) edge node {} (m-2-3);
		\NEbracket{(m-2-1)};
	    \NWbracket{(m-2-3)};
		\end{tikzpicture}
	\end{center}
	
	In such a case, we write $H$ as $G[\rwsubstw{L}{b_1,b_2}{R}]_m$.
\end{definition}

Note that the left hand pushout above is also $S$-adhesive, by the definition of $S$-matching. We often don't care about the particular rewrite rule and matching used to rewrite one graph into another, but merely that there \emph{exists} such a rewrite involving a rule in some fixed set. For this, we introduce rewrite systems and a ``rewrites-to'' relation.

\begin{definition}[Rewrite system]
\label{def:rewrite-system}
	A set of rewrite rules $\mathbb S$ is called a \emph{rewrite system}. We define the relation $G\,\rewritesto_{\mathbb S}\,H$ to mean there exists a rule $\rwrule{L}{R} \in \mathbb S$ and an $S$-adhesive matching $m : L \rightarrow G$ such that $H \cong G[\rwrule{L}{R}]_m$. The reflexive, transitive closure of $\rewritesto_{\mathbb S}$ is denoted $\rewritetrans_{\mathbb S}$, and the reflexive, symmetric, transitive closure as $\rewriteequiv_{\mathbb S}$.
\end{definition}

\begin{theorem}\label{thm:plugging-and-subtraction}
$S$-adhesive pushout complements commute with adhesive pushouts. Consider the following diagram, where $b$ is mono, $(b,m)$ has an $S$-adhesive pushout complement, and $(p,q)$ and $(p',q)$ are both $S$-adhesive spans.
	\begin{center}
		\begin{tikzpicture}
		\matrix (m) [cdiag,row sep=1em] {
		B & G -_{b,m} K &   &    \\
	      &             & P & H  \\
		K & G           &   &    \\
		};
		\path [arrs]
		  (m-1-1) edge node {$c$} (m-1-2)
		  (m-1-1) edge node [swap] {$b$} (m-3-1)
		  (m-1-2) edge node [swap] {$s$} (m-3-2)
		  (m-3-1) edge node [swap] {$m$} (m-3-2)

		  (m-2-3) edge node [swap] {$p'$} (m-1-2)
		  (m-2-3) edge node {$p$}  (m-3-2)
		  (m-2-3) edge node {$q$}  (m-2-4);
		\NWbracket{(m-3-2)}
		\end{tikzpicture}
	\end{center}
	
	Then, for the pushout injections $i : G \hookrightarrow G +_{p,q} H$ and $i' : G -_{b,m} K \hookrightarrow (G -_{b,m} K) +_{p',q} H$, there is an open-graph isomorphism, commuting with the coboundaries $c$ and $c'$ of $K$ in $G$ and $G +_{p,q} H$ respectively.
	\begin{equation}\label{dia:compatible-with-coboundary}
		\csquare{B}{(G +_{p,q} H) -_{b,im} K}
		        {G -_{b,m} K}{(G -_{b,m} K) +_{p',q} H}
		        {c'}{i'}{c}{\cong}
	\end{equation}
\end{theorem}

\begin{proof}
	The proof follows from the associativity of pushouts and the uniqueness of pushout complements. First, note that, in the following diagram, [1] commutes and is a pushout because $s p' = p$.
	\begin{center}
		\begin{tikzpicture}
		\node at (1,0) {\small [1]};
		\matrix (m) [cdiag] {
		  & P           &  H            \\
		B & G -_{b,m} K &               \\
		K & G           &  G +_{p,q} H  \\
		};
		\path [arrs]
		 (m-1-2) edge node {$q$} (m-1-3)
		 (m-1-2) edge node {$p'$} (m-2-2)
		 (m-2-1) edge node {$c$} (m-2-2)
		 (m-2-1) edge node {$b$} (m-3-1)
		 (m-3-1) edge node [swap] {$m$} (m-3-2)
		 (m-2-2) edge node {$s$} (m-3-2)
		 (m-1-3) edge (m-3-3)
		 (m-3-2) edge node [swap] {$i$} (m-3-3);
		\NWbracket{(m-3-2)}
		\NWbracket{(m-3-3)}
		\end{tikzpicture}
	\end{center}
	By associativity of pushouts, the following diagram also commutes, and the marked squares are pushouts:
	\begin{center}
		\begin{tikzpicture}
		\node at (-0.75,-0.9) {\small [2]};
		\matrix (m) [cdiag] {
		  & P           &  H                                 \\
		B & G -_{b,m} K &   (G -_{b,m} K) +_{p',q} H        \\
		K & G           &  G +_{p,q} H                       \\
		};
		\path [arrs]
		 (m-1-2) edge node {$q$} (m-1-3)
		 (m-1-2) edge node {$p'$} (m-2-2)
		 (m-2-1) edge node {$c$} (m-2-2)
		 (m-2-1) edge node {$b$} (m-3-1)
		 (m-3-1) edge node [swap] {$m$} (m-3-2)
		 (m-3-2) edge node [swap] {$i$} (m-3-3)
		 (m-2-2) edge (m-2-3)
		 (m-1-3) edge (m-2-3)
		 (m-2-3) edge (m-3-3);
		\NWbracket{(m-2-3)}
		\NWbracket{(m-3-3)}
		\end{tikzpicture}
	\end{center}
Now compare [2] to the subtraction of $im : K \rightarrow G +_{p,q} H$:
	\begin{center}
		\posquare{B}{(G +_{p,q} H) -_{b,im} K}
		         {K}         {G +_{p,q} H}
		         {}{i m}{b}{}
	\end{center}
The result then follows from uniqueness of pushout complements.
\end{proof}

\begin{theorem}\label{thm:pushout-and-rewrite}
	$S$-adhesive rewrites commute with $S$-adhesive pushouts. Let $m : L \rightarrow G$ be a matching of $\rwsubstw{L}{b_1,b_2}{R}$. The rewrite is computed as the double pushout:
	\begin{center}
		\begin{tikzpicture}
		\matrix (m) [cdiag] {
		L & B           & R                   \\
		G & G -_{b,m} L & G[\rwsubst{L}{R}]_m \\
		};
		\path [arrs]
		 (m-1-2) edge node [swap] {$b_1$} (m-1-1)
		 (m-1-2) edge node {$b_2$} (m-1-3)

		 (m-2-2) edge node {$s$} (m-2-1)
		 (m-2-2) edge node [swap] {$s'$} (m-2-3)

		 (m-1-1) edge node [swap] {$m$} (m-2-1)
		 (m-1-2) edge node {$c$} (m-2-2)
		 (m-1-3) edge node {$m'$} (m-2-3);
		\NEbracket{(m-2-1)}
		\NWbracket{(m-2-3)}
		\end{tikzpicture}
	\end{center}
	Let $(p,q)$, $(p', q)$ and $(\widehat p, q)$ be three adhesive spans, such that:
	\begin{equation}\label{dia:compatible}
		\begin{tikzpicture}
		\matrix (m) [cdiag] {
		G                   &   &   \\
		G -_{b_1,m} L       & P & H \\
		G[\rwsubst{L}{R}]_m &   &   \\
		};
		\path [arrs]
		  (m-2-2) edge [bend right] node [swap] {$p$} (m-1-1)
		  (m-2-2) edge node [swap] {$p'$} (m-2-1)
		  (m-2-2) edge [bend left] node {$\widehat p$} (m-3-1)
		  (m-2-1) edge node {$s$} (m-1-1)
		  (m-2-1) edge node [swap] {$s'$} (m-3-1)
		  (m-2-2) edge node {$q$} (m-2-3);
		\end{tikzpicture}
	\end{equation}
	Then, for the pushout injection $i : G \rightarrow G +_{p,q} H$, if $i m$ is mono, the following is an isomorphism:
	\[ (G[\rwsubst{L}{R}]_m) +_{\widehat p,q} H \cong
	   (G +_{p,q} H)[\rwsubst{L}{R}]_{i m} \]
\end{theorem}
\begin{proof}
	Since pushout complements are unique up to isomorphism, we can choose $(G -_{b_1,m} L)$ to be equal to $((G[\rwsubst{L}{R}]_m) -_{b_2,m'} R)$, for the same coboundary $c$. Then, by two applications of Thm \ref{thm:plugging-and-subtraction}, we can choose $(G -_{b_1,m} L) +_{p',q} H = ((G[\rwsubst{L}{R}]_m) -_{b_2,m'} R) +_{p',q} H$ as the pushout complement of both of the following squares.
	\begin{center}
		\begin{tikzpicture}
		\matrix (m) [cdiag] {
		L           & B                        & R \\
		G +_{p,q} H & (G -_{b_1,m} L) +_{p',q} H &
		     G[\rwsubst{L}{R}] +_{\widehat p, q} H \\
		};
		\path [arrs]
		 (m-1-2) edge (m-1-1)
		 (m-1-2) edge (m-1-3)

		 (m-2-2) edge (m-2-1)
		 (m-2-2) edge (m-2-3)

		 (m-1-1) edge node [swap] {$i m$} (m-2-1)
		 (m-1-2) edge node {$c'$} (m-2-2)
		 (m-1-3) edge (m-2-3);
		\NEbracket{(m-2-1)}
		\NWbracket{(m-2-3)}
		\end{tikzpicture}
	\end{center}
	
	Note that $c'$ becomes the coboundary for both squares because diagram (\ref{dia:compatible-with-coboundary}) commutes. This is then exactly the computation of the rewrite $(G +_{p,q} H)[\rwsubst{L}{R}]_{im}$.
\end{proof}

We shall use these two theorems throughout the paper to show that rewriting is compatible with several notions of composing graphs.

\section{Open-Graphs}
\label{sec:open-graphs}

In this section, we provide a formal definition for the notion of graphs that can contain edges with unconnected-ends, called \emph{open-graphs}. We do this by introducing a special kind of graph with two distinct types of points. It has points that should be considered as ``real'' vertices, and other intermediate points, called edge-points that occur along edges. In this construction, the ``logical'' edges of an open-graph, or wires, can be presented as chains of edge-points, which need not have a vertex at either end. Thus we can define the boundary of an open-graph as the unconnected ends of these wires. This provides the interface by which we connect open-graphs together. We prove several useful properties about the category \catOGraph{} of open-graphs, and show that the inclusion $S : \catOGraph \hookrightarrow \catGraphSlice$ is a selective adhesive functor into the (adhesive) slice category \catGraphSlice.

To fix notation, recall the standard definition for directed graphs as a functor category.

\begin{definition}
  Let $\catGraph$ be the category of graphs. It is defined as the functor category $[\mathbb G, \catSet]$, for $\mathbb G$ defined as:
  \begin{center}
	\cpair{E}{P}{s}{t}
  \end{center}
  $E$ identifies the edges of the graph, and $P$ the points.
  $s$ and $t$ are functions taking an edge to its source and target
  respectively. If $t(e) = p$ then $e$ is called an \emph{in-edge} of
  $p$ and if $s(e) = p$ then $e$ is called an \emph{out-edge} of $p$. 
\end{definition}

Note that our language for graphs differs slightly from the convention, in that we use the term ``point'' rather than ``vertex''. The reason for this will become clear once we introduce a typing on points. The type graph $2_\mathcal{G}$ will be used to distinguish points that should be interpreted as ``logical'' vertices, from the ``dummy''-points that occur along an edge:

\begin{center}
	$2_\mathcal{G} :=$
	\begin{tikzpicture}[mcgraph, node distance=1em and 3em]
    	\node[valg] (p1) {$V$}; 
    	\node[right=of p1,valg] (p2) {$\epsilon$};  
    	\path[->] 
      		(p1) edge [bend right=20] (p2)
      		(p2) edge [bend right=20] (p1);
    	\path[->] 
      		(p2) edge[loop right] node[above] {} ();
    \end{tikzpicture}
\end{center}

A graph, $G$, is said to be typed by $2_\mathcal{G}$ when there is a typing morphism $\tau : G \rightarrow 2_\mathcal{G}$. When a vertex, $p \in G$, is mapped to $V$, i.e. $\tau(p) = V$, we refer to it as a \emph{vertex-point} or simply as a \emph{vertex}. The other points in $G$, those with $\tau(p) = \epsilon$, are called \emph{edge-points}.


\begin{definition}[$\catOGraph$]
 	The category $\catOGraph$ of \emph{open-graphs}, is a subcategory of
  the slice category $\catGraphSlice$. Objects are those of
  $\catGraphSlice$ where each edge-point has at most one in-edge and
  one out-edge. The morphisms of $\catOGraph$ are the same as those in
  $\catGraphSlice$, with the additional restriction that they be full on
  vertices: any edge adjacent to a vertex $f(v)$ must also be in the image of
  $f$.
\end{definition}


This slice construction plays two roles. As well as distinguishing `real' vertices from edge-points, the lack of a self-loop on $V$ ensures that every path between two vertices must have at least one edge-point.

\begin{example}
  A diagrammatic presentation of an open-graph:
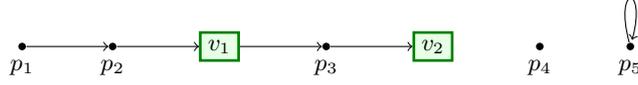
\begin{center} 
\begin{tikzpicture}[mcgraph, node distance=3em and 3em]
  \node[ipoint] (p1) {};  
  \node[ipointlabel,below=of p1] {$p_{1}$};  

  \node[right=of p1,ipoint] (p2) {};  
  \node[ipointlabel,below=of p2] {$p_{2}$};  

  \node[right=of p2,andg] (v1) {$v_1$};  

  \node[right=of v1,ipoint] (p3) {};  
  \node[ipointlabel,below=of p3] {$p_{3}$};  


  \node[right=of p3,andg] (v2) {$v_2$};  

  \node[right=of v2,ipoint] (p4) {};  
  \node[ipointlabel,below=of p4] {$p_{4}$};  

  \node[right=of p4,ipoint] (p5) {};  
  \node[ipointlabel,below=of p5] {$p_{5}$};  

  \path[->,elabel] 
  (p1) edge (p2)
  (p2) edge (v1)
  (v1) edge (p3)
  (p3) edge (v2)
  (p5) edge[loop above]  ()
  ;
\end{tikzpicture}
\end{center}
This diagram abbreviates a graph with its morphism to
$2_{\mathcal{G}}$, which can otherwise be drawn in the more verbose
fashion:
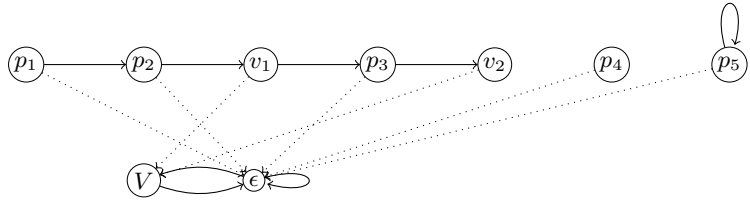
\begin{center}
\begin{tikzpicture}[mcgraph, node distance=3em and 3em]
  \node[valg] (p1) {$p_{1}$};  
  \node[right=of p1,valg] (p2) {$p_{2}$};  
  \node[right=of p2,valg] (v1) {$v_1$};  
  \node[right=of v1,valg] (p3) {$p_{3}$};  
  \node[right=of p3,valg] (v2) {$v_2$};  
  \node[right=of v2,valg] (p4) {$p_{4}$};  
  \node[right=of p4,valg] (p5) {$p_5$};  
  \path[->,elabel] 
  (p1) edge (p2)
  (p2) edge (v1)
  (v1) edge (p3)
  (p3) edge (v2)
  (p5) edge[loop above]  ()
  ;

  \node[valg,below=of p2] (pv) {$V$}; 
  \node[right=of pv,valg] (pe) {$\epsilon$};  
  \path[->] 
  (pv) edge [bend right=20] (pe)
  (pe) edge [bend right=20] (pv);
  \path[->] 
  (pe) edge[loop right] node[above] {} ();

  \path[dotted, ->] 
  (p1) edge (pe)
  (p2) edge (pe)
  (p3) edge (pe)
  (p4) edge (pe)
  (p5) edge (pe)
  (v1) edge (pv)
  (v2) edge (pv)
;

\end{tikzpicture}
\end{center}
\noindent where the dotted arrows indicate the type-morphism for
points, and the edge mapping is trivially inferred.
\end{example}

\begin{definitions}[$\catOGraph$ Notation]
  If an edge-point $p \in P_G$ has no in-edges, it is called an
  \emph{input}. We write the set of inputs of $G$ as
  $\In(G)$. Similarly, an edge-point with no out-edges is called an
  \emph{output}, and the set of outputs is written $\Out(G)$. The
  inputs and outputs define an open-graph's \emph{boundary}. If a
  boundary point has no in-edges and no out-edges, (it is both and
  input and output) it is called an \emph{isolated point}. An open
  graph consisting of only isolated points is called a
  \emph{point-graph}.
\end{definitions}

Note that when there is no ambiguity, we shall use $\In(G)$ and $\Out(G)$ to also refer to the point-graph containing only the inputs or outputs of $G$. As graphs, these have natural inclusions into $G$. We now define the \emph{boundary graph of a open-graph} which plays a particularly important role for composition as well as decomposition of open-graphs.

\begin{definition}[Boundary Graph and Boundary Map] \label{def:boundary-map}
  Given an open-graph $G$, its \emph{boundary graph} is the point-graph
  formed from the coproduct of its inputs and outputs: $B := \In(G) +
  \Out(G)$. The \emph{boundary map} of $G$ is the induced map
  $b : B \rightarrow G$ of the inclusions of $\In(G)$ and $\Out(G)$.
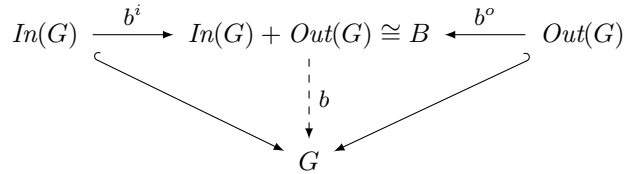
	\begin{center}
		\begin{tikzpicture}[-latex]
			\matrix (m) [cdiag] {
		\In(G) &	\In(G) + \Out(G) \cong B & \Out(G)  \\
      & G & \\
			};
			\path [arrs]
				(m-1-1) edge node {$b^i$} (m-1-2)
				(m-1-3) edge node [swap] {$b^o$} (m-1-2)
				(m-1-2) edge [dashed] node {$b$} (m-2-2)
				(m-1-1) edge [right hook-latex] (m-2-2)
				(m-1-3) edge [left hook-latex] (m-2-2);
		\end{tikzpicture}
	\end{center}
  Note that for a boundary map $b$, we refer to the associated coproduct injections as $b^i$ and $b^o$.
\end{definition}

A boundary map identifies all the inputs and outputs of $G$, and is injective except on that isolated points of $G$, where it is 2-to-1.

\begin{example}
  The following illustrates a graph (below) with its boundary graph (above), where the boundary map is indicated by the dotted arrows.
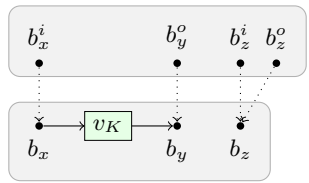
\begin{center}
  \begin{tikzpicture}[mcgraph,node distance=2em and 2em]
  \node[ipoint] (jpx) {};  
  \node[ipointlabel, below=of jpx] (jpxl) {$b_{x}$};
  { [node distance=1.5em and 1.5em]
  
  \node[ivert, right=of jpx] (jv1) {$v_{K}$};  

  \node[ipoint, right=of jv1] (jpy) {};  
  \node[ipointlabel, below=of jpy] (jpyl) {$b_{y}$};  

  \node[node distance=2em and 2em,ipoint, right=of jpy] (jpz) {};  
  \node[ipointlabel, below=of jpz] (jpzl) {$b_{z}$};  

  \path[->] (jpx) edge (jv1)
            (jv1) edge (jpy);
  }
  \begin{pgfonlayer}{background}
  \node[rectangle, fill=black!5, draw=black!30,rounded corners=1ex,
  fit=(jv1) (jpx) (jpxl) (jpy) (jpyl) (jpz) (jpzl)] (J) {};  
  \end{pgfonlayer}

  \node[ipoint, above=of jpx] (bjpx) {};  
  \node[ipointlabel, above=of bjpx] (bjpxl) {$b^i_{x}$};

  \node[ipoint, above=of jpy] (bjpy) {};  
  \node[ipointlabel, above=of bjpy] (bjpyl) {$b^o_{y}$};  

  \node[ipoint, above=of jpz] (bjpzi) {};  
  \node[ipointlabel, above=of bjpzi] (bjpzli) {$b^i_{z}$};  

  { [node distance=1em and 1em]
  \node[ipoint, right=of bjpzi] (bjpzo) {};  
  \node[ipointlabel, above=of bjpzo] (bjpzlo) {$b^o_{z}$};  
  }
  \begin{pgfonlayer}{background}
  \node[rectangle, fill=black!5, draw=black!30,rounded corners=1ex,
  fit=(bjpx) (bjpxl) (bjpy) (bjpyl) (bjpzi) (bjpzli) (bjpzo) (bjpzlo)] (BJ) {};  
  \end{pgfonlayer}

  \path[->, dotted] 
  (bjpx) edge (jpx) 
  (bjpy) edge (jpy) 
  (bjpzi) edge (jpz)
  (bjpzo) edge (jpz)
;
  \end{tikzpicture}
\end{center}
\end{example}

Notice that because each isolated-point is both an input and an output, a boundary graph has two points for each isolated point in its associated graph. It is important to note that a boundary map is mono if and only if its associated graph has no isolated points.

\begin{definition}[Share the same boundary] \label{def:same-boundary}
	Two graphs $G$ and $H$ are said to \emph{share the same boundary},
  $B$, by boundary maps $b_1$ and $b_2$, when $G
  \overset{b_1}{\longleftarrow} B \overset{b_2}{\longrightarrow} H$,
  $\In(L) \cong \In(R)$, $\Out(L) \cong \Out(R)$, and the following
  diagram commutes:
	
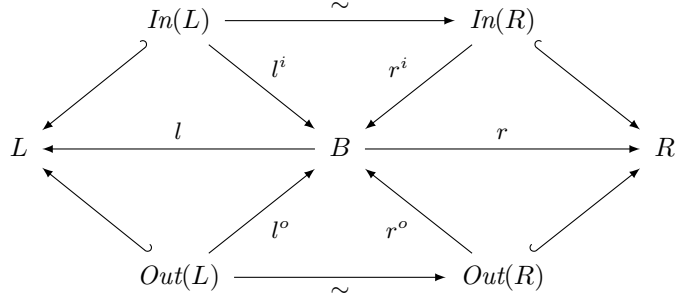
\begin{center}
		\begin{tikzpicture}
			\matrix (m) [cdiag] {
			  &  \In(L) &   & \In(R)  &   \\
			L &         & B &         & R \\
			  & \Out(L) &   & \Out(R) &   \\
			};
			\path [arrs]
				(m-2-3) edge node [swap] {$l$} (m-2-1)
				(m-2-3) edge node {$r$} (m-2-5)
				(m-1-2) edge [left hook-latex] (m-2-1)
				(m-3-2) edge [left hook-latex] (m-2-1)
				(m-1-4) edge [right hook-latex] (m-2-5)
				(m-3-4) edge [right hook-latex] (m-2-5)
				(m-1-2) edge node {$l^i$} (m-2-3)
				(m-3-2) edge node [swap] {$l^o$} (m-2-3)
				(m-1-4) edge node [swap] {$r^i$} (m-2-3)
				(m-3-4) edge node {$r^o$} (m-2-3)
				
				(m-1-2) edge node {$\sim$} (m-1-4)
				(m-3-2) edge node [swap] {$\sim$} (m-3-4);
		\end{tikzpicture}
	\end{center}
\end{definition}

Notice that for two graphs with no isolated points, this condition means that boundaries of the two graphs are in bijection, and furthermore that bijection sends inputs to inputs and outputs to outputs.

We now show some basic properties of $\catOGraph$. In particular, we develop the properties needed to show that the inclusion of $\catOGraph$ into $\catGraphSlice$ is a selective adhesive functor.



\begin{lemma} \label{lem:mono-injections}
	A map in \catOGraph{} is a monomorphism iff it is injective.
\end{lemma}
\begin{proof}
  To prove the left to right direction of the iff, assume $f : G \rightarrow H$ is mono but not injective. Then there must be some point or edge, $x$ in $H$, that has more than one point or edge in its pre-image. For $K$ the smallest graph in $G$ that contains a point or edge in the pre-image of $x$, there exist two distinct embeddings $i_1, i_2 : K \rightarrow G$ in $\catGraphSlice{}$ such that $f i_1 = f i_2$. Either $K$ is a single-vertex, an  isolated point, or a single edge with two endpoints, thus it is an open-graph, but $i_1$ and $i_2$ are not necessarily full on vertices. However, the induced maps $[1_G, i_1] : G + K \rightarrow G$ and $[1_G, i_2] : G + K \rightarrow G$ are still distinct \emph{and} are full on vertices, where $+$ is the coproduct of $\catGraphSlice{}$ and $1_G$ is the identity on $G$. Thus: 
\[
 f \circ [1_G, i_1] = [f, f i_1] =
 [f, f i_2] = f \circ [1_G, i_2]
\]
This implies that $f$ is not mono in \catOGraph{}, and from this contradiction we get that $f$ is injective. 

The reverse direction of the iff follows from injective morphisms in $\catGraph$ being monos.  
\end{proof}

We will now prove a similar result for surjections. However, \catOGraph{} is quite restrictive on the types of maps and graphs that exist, so not all epimorphisms are surjective. However, all \emph{strong} epimorphisms are. To show this, we first recall the notion of strong epimorphism and prove a simple fact about surjections.

\begin{definition} \label{def:strong-epi}
	A \emph{strong epimorphism} in $\mathcal C$ is an epimorphism $e$ that is left-orthogonal to all monomorphisms in $\mathcal C$. That is, for any commutative square of the following form, with $m$ as a monomorphism:
	\begin{center}
		\begin{tikzpicture}[-latex]
		    \matrix(m)[cdiag]{
		    A & B \\
		    C & D  \\};
		    \path [arrs] (m-1-1) edge [->>] node {$e$} (m-1-2)
		                 (m-2-1) edge [right hook-latex] node [swap] {$m$} (m-2-2)
		                 (m-1-1) edge node [swap] {$f$} (m-2-1)
		                 (m-1-2) edge node {$g$} (m-2-2)
						 (m-1-2) edge [dashed] node [swap] {$d$} (m-2-1);
		\end{tikzpicture}
	\end{center}
there exists a unique diagonal map, $d$, making the diagram commute.
\end{definition}

\begin{lemma}\label{lem:surj-full}
	For the following commutative triangle:
	
	\begin{center}
	\begin{tikzpicture}[-latex]
		\matrix (m) [cdiag] {
		A & B \\
		  & C \\
		};
		\path
			(m-1-1) edge node [auto] {$e$} (m-1-2)
			(m-1-1) edge node [auto,swap] {$f$} (m-2-2)
			(m-1-2) edge node [auto] {$f'$} (m-2-2);
	\end{tikzpicture}
	\end{center}
	where $e$ is a surjection, $f$ is full on vertices iff $f'$ is.
\end{lemma}
\begin{proof}
	First, consider the case when $f'$ is full on vertices. All surjections are full on vertices, thus $f' e$ is also full on vertices. Now, suppose $f$ is full on vertices. Note that precomposing with a surjection does not affect the image of a map, so $f[A] = f'e[A] = f'[B]$. So $f'$ is full on vertices.
\end{proof}

\begin{lemma}\label{lem:surj-strong}
	A map in \catOGraph{} is a strong epimorphism iff it is surjective.
\end{lemma}
\begin{proof}
  We first show that surjections in \catOGraph{} are strong epimorphisms. First note that surjections in \catOGraph{} are strong epimorphisms in \catGraphSlice{}; monos in \catOGraph{} are injections and hence monos in \catGraphSlice{}. Therefore, it suffices to show that the diagonal map $d$ from Def~\ref{def:strong-epi} is full on vertices. This follows from Lem~\ref{lem:surj-full}.
	
  To show that all strong epimorphisms are surjections, we begin by
  noting that strong epimorphisms are, in particular, extremal
  epimorphisms. That is, given an epimorphism $e$, such that for any
  factorisation $e = m f$, where $m$ is mono, then $m$ must be an
  isomorphism. Suppose some map $e$ does not have this property, then
  it factors as $e = m f$ for some monomorphism that is not an
  isomorphism. Then $m$ must not be surjective, so $e$ must also not
  be. From this contradiction, strong epimorphisms in \catOGraph{} are
  all surjections.
\end{proof}

\begin{lemma}\label{lem:strong-epi-mono-factorisations}
	\catOGraph{} has unique strong epi-mono factorisations.
\end{lemma}
\begin{proof}
	If a map $f$ factors as $m e$, where $e$ is a strong epimorphisms and $m$ is a monomorphism, this factorisation is automatically unique. This factorisation exists because any map $f$ factors through its image:
	\[ A \overset{f_e}{\twoheadrightarrow} f[A]
	     \overset{f_m}{\hookrightarrow} B \]
	The map $f_e$ of $f$ onto its image is full on vertices because it is surjective, and the embedding $f_m$ of the image of $f$ in $B$ is full on vertices precisely when $f$ is. By Lemmas \ref{lem:mono-injections} and \ref{lem:surj-strong}, all strong epi-mono factorisations are of this form.
\end{proof}

\begin{lemma}\label{lem:graphic-mono-iso}
	The embedding functor $S : \catOGraph \hookrightarrow
  \catGraphSlice$ preserves and reflects monomorphisms and strong epimorphisms, and it creates isomorphisms.
\end{lemma}
\begin{proof}
  Monos and strong epimorphisms follow from Lemmas \ref{lem:mono-injections} and \ref{lem:surj-strong}. Creation of isomorphisms follows from the fact that the definition of open-graph is invariant under isomorphism and all isomorphisms are full on vertices.
\end{proof}

\begin{lemma}\label{lem:bc-spans}
	The embedding functor $S$ reflects colimits.
\end{lemma}

\begin{proof}
	Let the following diagram be a coequaliser in $\catGraphSlice$:
	\begin{center}
		\begin{tikzpicture}
			\matrix (m) [cdiag] {
			  S A & S B & S Q \\
			};
			\path [arrs]
			   (m-1-1.20) edge node {$S f$} (m-1-2.160)
			   (m-1-1.-20) edge node [swap] {$S g$} (m-1-2.-160)
			   (m-1-2) edge node {$S q$} (m-1-3);
		\end{tikzpicture}
	\end{center}
	
	Then, because $S$ is faithful, $q f = q g$. Suppose there is some $q'$ in \catOGraph{} such that $q' f = q' g$, then there exists unique $u$ in \catGraphSlice{} making this diagram commute:
	
	\begin{center}
		\begin{tikzpicture}
			\matrix (m) [cdiag] {
			  S A & S B & S Q  \\
			      &     & S Q' \\
			};
			\path [arrs]
			   (m-1-1.20) edge node {$S f$} (m-1-2.160)
			   (m-1-1.-20) edge node [swap] {$S g$} (m-1-2.-160)
			   (m-1-2) edge node {$S q$} (m-1-3)
			   (m-1-2) edge node [swap] {$S q'$} (m-2-3)
			   (m-1-3) edge [dashed] node {$u$} (m-2-3);
		\end{tikzpicture}
	\end{center}
	
	$S q$ is a regular epimorphism in \catGraphSlice{}, so in particular it is a strong epimorphism. By Lem \ref{lem:graphic-mono-iso}, $q$ is a strong epimorphism in \catOGraph{}. Thus, by Lemma \ref{lem:surj-full} $u$ is full on vertices, or equivalently, $u = S u'$ for some (unique) $u'$ in \catOGraph{}. So, by faithfulness of $S$, the following diagram commutes:
	\begin{center}
		\begin{tikzpicture}
			\matrix (m) [cdiag] {
			  A & B & Q  \\
			    &   & Q' \\
			};
			\path [arrs]
			   (m-1-1.20) edge node {$f$} (m-1-2.160)
			   (m-1-1.-20) edge node [swap] {$g$} (m-1-2.-160)
			   (m-1-2) edge node {$q$} (m-1-3)
			   (m-1-2) edge node [swap] {$q'$} (m-2-3)
			   (m-1-3) edge [dashed] node {$u'$} (m-2-3);
		\end{tikzpicture}
	\end{center}
	
	Let $SC$ be a set-indexed coproduct in \catGraphSlice{}, with injections $S i_j : S G_j \rightarrow S C$. Then, for any arrows $f_j : G_j \rightarrow H$, there exists a unique $u$ making the following diagram commute for all $j$:
	\begin{center}
		\begin{tikzpicture}
			\matrix (m) [cdiag] {
			S G_j & SC \\
			      & S H  \\
			};
			\path [arrs]
			  (m-1-1) edge node {$S i_j$} (m-1-2)
			  (m-1-1) edge node [swap] {$S f_j$} (m-2-2)
			  (m-1-2) edge [dashed] node {$u$} (m-2-2);
		\end{tikzpicture}
	\end{center}
	
	The image of $u$ is the union of the images of all the maps $f_j$, so it is full on vertices. Therefore $C \cong \coprod G_j$ is a coproduct in \catOGraph{}.
\end{proof}

\begin{theorem}\label{thm:ograph-selective-adhesive}
	The embedding functor $S : \catOGraph \hookrightarrow \catGraphSlice$ is a selective adhesive functor.
\end{theorem}

\section{Composition and Decomposition for Open Graphs}
\label{sec:compose-graphs}

In this section, we show how open-graphs can be composed and decomposed using the embedding functor $\catOGraph \rightarrow \catGraphSlice$ which we call $S$. In particular, using the fact that $S$ is a selective adhesive functor, we introduce definitions for subtracting one graph from another (by pushout complements) and for connecting open-graphs along their boundary.

We will first define the spans of graphs that preserve open-graphs under pushout. The crucial features of such pushouts are:

\begin{itemize} 
\item if two edges are identified, then the whole paths of edge-points they are on are also identified: never identify only a middle section of one edge with a middle section of another edge;
\item outputs should only be connected to inputs: never connect the output of an edge to the output of another edge, and likewise with inputs.
\end{itemize}

This idea is formalised by the notion of \emph{boundary-coherent spans}.

\begin{definition}[Boundary Coherent] \label{def:boundary-coherent}
	A span $H_1 \overset{f}{\longleftarrow} G \overset{g}{\longrightarrow} H_2$ is
  called \emph{boundary coherent} when $f$ and $g$ are monos and:
	\begin{enumerate}
		\item for all $p \in \In(G)$ at most one of $f(p)$ and $g(p)$ is an
      input;
		\item for all $p \in \Out(G)$ at most one of $f(p)$ and $g(p)$ is an
      output. 
	\end{enumerate}
  A parallel pair of arrows $f,g : G \rightarrow H$, is called a \emph{boundary coherent pair} when the span $H \overset{f}{\longleftarrow} G \overset{g}{\longrightarrow} H$ is boundary coherent. 
 \end{definition}

\begin{theorem}
Boundary-coherent spans are $S$-adhesive spans.
\end{theorem}
\begin{proof}
Let $A \overset{f}{\longleftarrow} B \overset{g}{\longrightarrow} C$ be a boundary coherent span. Then the following is a pushout in \catGraphSlice{}.

\begin{center}
	\posquare{S A}{S B}{S C}{D}{S f}{p_2}{S g}{p_1}
\end{center}

Since $S$ reflects pushouts, it suffices to show that $D$, $i_1$ and $i_2$ are in the image of $S$. Since $S$ preserves monos, $S f$ and $S g$ are mono. A pushout of monos in \catGraphSlice{} is (up to isomorphism) just a union. So, without loss of generality, we can let $S A = S B \cap S C$ and $D = S B \cup S C$, and the two inclusions of the intersection form a boundary-coherent span. Thus we can rewrite the above pushout as follows.
\begin{center}
	\monoposquare{S B \cap S C}{S B}{S C}{S B \cup S C}{S f}{i_2}{S g}{i_1}
\end{center}
Suppose an edge point $p$ in $S B \cup S C$ has two out-edges. Then one must be in $S B$ and the other in $S C$. Thus neither are in the intersection, so $p$ is an output in $S B \cap S C$. But it is not an output in $S B$ or $S C$, thus contradicting the boundary coherence assumption. A contradiction follows similarly when $p$ has two in-edges, so $S B \cup S C$ is an open-graph. Moreover, $i_1$ and $i_2$ are full on vertices because $S f$ and $S g$ are, so $(S f, S g)$ is an $S$-adhesive span.
\end{proof}

This provides boundary-coherent spans with nice properties for pushouts as reflected by the selective adhesive functor $S$. These pushouts, which we call \emph{mergings}, will play a central role in our construction of rewriting.

\begin{definition}[Merging]
  Given a boundary-coherent span of monos $G_1
  \overset{m_1}{\hookleftarrow} K \overset{m_2}{\hookrightarrow} G_2$,
  we use the notation, $M := G_1 \mergew{m_1,m_2} G_2$, for the
  pushout of the span, which we call the \emph{merging} of
  $G_1$ and $G_2$ on $K$ by $m_1$ and $m_2$: 
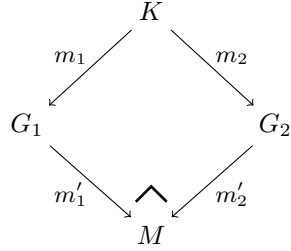
\begin{center}
  \begin{tikzpicture}[text height=1.5ex]
	  \node[] (top) {$K$}; 
	  \node[below left=of top] (left) {$G_1$}; 
	  \node[below right=of top] (right) {$G_2$}; 
	  \node[below=of top] (mid) {}; 
	  \node[below=of mid] (bot) {$M$};
	  \path[->] (top) edge node [above left,elabel] {$m_1$} (left)
	            (top) edge node [above right,elabel] {$m_2$} (right)
	            (left) edge node [below left,elabel] {$m_1'$} (bot)
	            (right) edge node [below right,elabel] {$m_2'$} (bot)
	            ;
 
	  { [node distance=0em and 0em]
	     \node[above=of bot, circle, inner sep=2pt] (pushout) {};
	     \draw[line width=1pt] (pushout.west) -- (pushout.north) -- (pushout.east); 
	  }
	\end{tikzpicture}
\end{center}
\noindent This makes $M$ the smallest graph containing $G_1$ and $G_2$
with a single copy of the shared sub-graph $K$, as identified by $m_1$
and $m_2$. 
\end{definition}

\begin{example}  \label{ex4:merge} 
An illustration of merging graphs: 
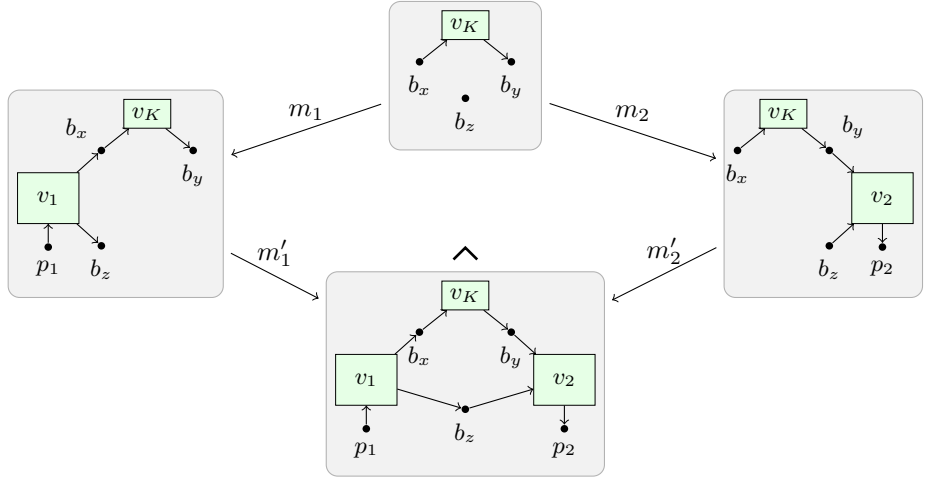
\begin{center}
  \begin{tikzpicture}[mcgraph,node distance=5em and 7em]
  \node[ipoint] (jpx) {};  
  \node[ipointlabel, below=of jpx] (jpxl) {$b_{x}$};
  { [node distance=0.7em and 0.7em]
  
  \node[ivert, above right=of jpx] (jv1) {$v_{K}$};  

  \node[ipoint, below right=of jv1] (jpy) {};  
  \node[ipointlabel, below=of jpy] (jpyl) {$b_{y}$};  

  \node[node distance=2em and 2em,ipoint, below=of jv1] (jpz) {};  
  \node[ipointlabel, below=of jpz] (jpzl) {$b_{z}$};  

  \path[->] (jpx) edge (jv1)
            (jv1) edge (jpy);
  }
  \begin{pgfonlayer}{background}
  \node[rectangle, fill=black!5, draw=black!30,rounded corners=1ex,
  fit=(jv1) (jpx) (jpxl) (jpy) (jpyl) (jpz) (jpzl)] (J) {};  
  \end{pgfonlayer}

  \node[ipoint, left=of J.south west] (lpy) {};
  \node[ipointlabel, below=of lpy] (lpyl) {$b_y$};  
  { [node distance=0.7em and 0.7em]
  \node[ivert, above left=of lpy] (lvk) {$v_{K}$};  

  \node[ipoint,below left=of lvk] (lpx) {};  
  \node[ipointlabel,above left=of lpx] (lpxl) {$b_x$};  

  \node[ivert, inner sep=0.7em, below left=of lpx] (lv1) {$v_{1}$}; 

  \node[ipoint,below=of lv1] (lp1) {};  
  \node[ipointlabel,below=of lp1] (lp1l) {$p_1$};  

  \node[ipoint,below right=of lv1] (lpz) {};  
  \node[ipointlabel,below=of lpz] (lpzl) {$b_z$};  
  \path[->] 
   (lvk) edge (lpy)
   (lp1) edge (lv1)
   (lpx) edge (lvk)
   (lv1) edge (lpx)
   (lv1) edge (lpz);
  }
  \begin{pgfonlayer}{background}
  \node[rectangle, fill=black!5, draw=black!30,rounded corners=1ex,
  fit=(lvk) (lv1) (lpx) (lpxl) (lpy) (lpyl) (lpz) (lpzl)] (L) {};  
  \end{pgfonlayer}

  \node[ipoint, right=of J.south east] (rpx) {};  
  \node[ipointlabel, below=of rpx] (rpxl) {${b}_{x}$};  
  { [node distance=0.7em and 0.7em]
 \node[ivert, above right=of rpx] (rvk) {$v_{K}$};  

  \node[ipoint,below right=of rvk] (rpy) {};  
  \node[ipointlabel,above right=of rpy] (rpyl) {$b_y$};  

  \node[ivert, inner sep=0.7em, below right=of rpy] (rv2) {$v_2$}; 

  \node[ipoint,below=of rv2] (rp2) {};  
  \node[ipointlabel,below=of rp2] (rp2l) {$p_2$};  

  \node[ipoint, below left=of rv2] (rpz) {}; 
  \node[ipointlabel, below=of rpz] (rpzl) {$b_{z}$};  
  \path[->] (rpz) edge (rv2)
            (rv2) edge (rp2)
            (rpy) edge (rv2)
            (rpx) edge (rvk)
            (rvk) edge (rpy);
  }
  \begin{pgfonlayer}{background}
  \node[rectangle, fill=black!5, draw=black!30,rounded corners=1ex, fit=(rv2) (rpx) (rpzl) (rvk)] (R) {};  
  \end{pgfonlayer}

  \node[below=of J] (mc) {};  
  { [node distance=0.7em and 0.7em]
  \node[ivert, at=(mc)] (mvk) {$v_{K}$};  

  \node[ipoint,below left=of mvk] (mpx) {};  
  \node[ipointlabel,below=of mpx] (mpxl) {$b_x$};  

  \node[ipoint,below right=of mvk] (mpy) {};  
  \node[ipointlabel,below=of mpy] (mpyl) {$b_y$};  

  \node[ivert, inner sep=0.7em, below left=of mpx] (mv1) {$v_1$}; 
  \node[ivert, inner sep=0.7em, below right=of mpy] (mv2) {$v_2$}; 
 
 \node[ipoint,below=of mv1] (mp1) {};  
  \node[ipointlabel,below=of mp1] (mp1l) {$p_1$};
  
 \node[ipoint,below=of mv2] (mp2) {};  
  \node[ipointlabel,below=of mp2] (mp2l) {$p_2$};  

  \node[node distance=3.5em and 0.7em,ipoint,below=of mvk] (mpz) {}; 
  \node[ipointlabel,below=of mpz] (mpzl) {$b_z$};  

  \draw[->] (mpx) -- (mvk);
  \draw[->] (mv1) -- (mpx);
  \draw[->] (mvk) -- (mpy);
  \draw[->] (mpy) -- (mv2);
  \draw[->] (mv1) -- (mpz);
  \draw[->] (mpz) -- (mv2);
  \path[->] 
    (mv2) edge (mp2)
    (mp1) edge (mv1);
        
  }
  \begin{pgfonlayer}{background}
  \node[rectangle, fill=black!5, draw=black!30,rounded corners=1ex,
  fit=(mv1) (mp1l) (mp2l) (mv2) (mvk) (mpzl)] (M) {};  
  \end{pgfonlayer}

  \path[->,shorten <= 3pt,shorten >= 3pt] 
  (J) edge node[above] {$m_1$} (L)
  (J) edge node[above] {$m_2$} (R)
  (L) edge node[above] {$m_1'$} (M)
  (R) edge node[above] {$m_2'$} (M);

{ [node distance=0em and 0em]
    \node[above=of M, circle, inner sep=3pt] (pushout) {};
    \draw[line width=1pt] (pushout.west) -- (pushout.north) -- (pushout.east); 
  }

\end{tikzpicture} 
\end{center} 
The grey boxes are drawn around the graphs involved to distinguish between edges in the graphs and those of the pushout diagram. The image of the maps are indicated by the naming of edge-points and vertices.
\end{example}


A particularly important special case of merging is composition along half-edges, which we call \emph{plugging}.

\begin{definition}[Plugging]
  A graph merging $G_1 \mergew{b_1,b_2} G_2$ is called a
  \emph{plugging} and written $G_1 \plugw{b_1,b_2} G_2$, when the
  graph being merged on is a point graph; i.e. in the span $G_1
  \overset{b_1}{\hookleftarrow} P \overset{b_2}{\hookrightarrow} G_2$,
   $P$ is a point-graph.
\end{definition}

\begin{example}  \label{ex4:plug} 
An illustration of plugging using pushouts.
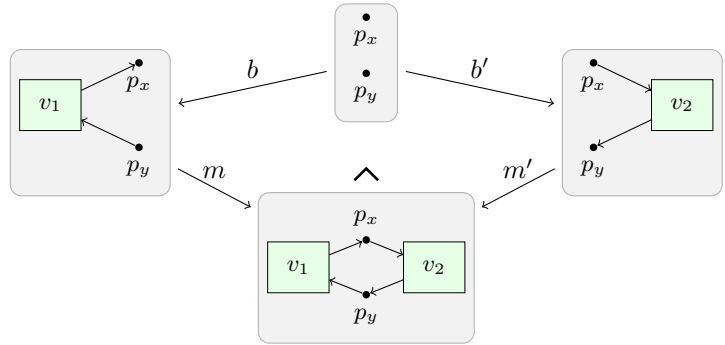
\begin{center}
  \begin{tikzpicture}[mcgraph,node distance=5em and 7em]
  \node[ipoint] (js1) {};  
  \node[ipointlabel, below=of js1] (js1l) {$p_{x}$};  

  { [node distance=0.5em and 2em]
  \node[ipoint, below=of js1l] (js2) {};  
  \node[ipointlabel, below=of js2] (js2l) {$p_{y}$};  
  }

  \begin{pgfonlayer}{background}
  \node[rectangle, fill=black!5, draw=black!30,rounded corners=1ex, fit=(js1) (js2) (js2l) (js1l)] (J) {};  
  \end{pgfonlayer}

  \node[ipoint, left=of J] (lt1) {};
  \node[ipointlabel, below=of lt1] (lt1l) {$p_x$};  
  { [node distance=0.5em and 2em]
  \node[ivert, inner sep=0.7em, below left=of lt1] (lv1) {$v_1$}; 
  \node[ipoint,below right=of lv1] (ls1) {};  
  \node[ipointlabel,below=of ls1] (ls1l) {$p_y$};  
  \path[->] (lv1) edge (lt1)
            (ls1) edge (lv1);
  }
  \begin{pgfonlayer}{background}
  \node[rectangle, fill=black!5, draw=black!30,rounded corners=1ex, fit=(lv1) (ls1) (lt1) (ls1l) (lt1l)] (L) {};  
  \end{pgfonlayer}

  \node[ipoint, right=of J] (rs1) {};  
  \node[ipointlabel, below=of rs1] (rs1l) {${p}_{x}$};  
  { [node distance=0.5em and 2em]
  \node[ivert, inner sep=0.7em, below right=of rs1] (rv1) {$v_2$}; 
  \node[ipoint, below left=of rv1] (rt1) {$$};  
  \node[ipointlabel, below=of rt1] (rt1l) {$p_{y}$};  
  \path[->] (rv1) edge (rt1)
            (rs1) edge (rv1);
  }
  \begin{pgfonlayer}{background}
  \node[rectangle, fill=black!5, draw=black!30,rounded corners=1ex, fit=(rv1) (rs1) (rt1) (rs1l) (rt1l)] (R) {};  
  \end{pgfonlayer}

  \node[below=of J] (kc) {};  
  { [node distance=0.5em and 1em]
  \node[ivert, inner sep=0.7em, left=of kc] (kv1) {$v_1$}; 
  \node[ivert, inner sep=0.7em, right=of kc] (kv2) {$v_2$}; 
  
  \node[ipoint,above=of kc] (bx) {}; 
  \node[ipoint,below=of kc] (by) {}; 
  \node[ipointlabel,above=of bx] (bxl) {$p_x$};  
  \node[ipointlabel,below=of by] (byl) {$p_y$};  

  \draw[->] (kv1) -- (bx);
  \draw[->] (bx) -- (kv2);
  \draw[->] (kv2) -- (by);
  \draw[->] (by) -- (kv1);
  }

  \begin{pgfonlayer}{background}
  \node[rectangle, fill=black!5, draw=black!30,rounded corners=1ex, fit=(kv1) (kv2) (bxl) (byl)] (M) {};  
  \end{pgfonlayer}

  \path[->,shorten <= 3pt,shorten >= 3pt] 
  (J) edge node[above] {$b$} (L)
  (J) edge node[above] {$b'$} (R)
  (L) edge node[above] {$m$} (M)
  (R) edge node[above] {$m'$} (M);

{ [node distance=0em and 0em]
    \node[above=of M, circle, inner sep=3pt] (pushout) {};
    \draw[line width=1pt] (pushout.west) -- (pushout.north) -- (pushout.east); 
  }
  \end{tikzpicture}
\end{center}
\end{example}



Note that the special case of plugging formed by pushouts on the empty open-graph is the disjoint union of open-graphs, written simply as $G + H$; visually this corresponds to placing graphs side by side.

We now introduce a dual notion to merging, called \emph{subtraction} which is formed by $S$-adhesive pushout complements. Intuitively subtraction removes part of an open-graph identified by a monomorphism. We first give a concrete definition for subtraction and then we show that this definition does indeed produce $S$-adhesive pushout complements.



\begin{definition}[Subtraction] \label{def:subtraction} 
  We define the \emph{subtraction} of $G$ from $M$, at a mono $m : G
  \injmap M$, written $M -_m G$, as the graph $H$ defined by:
	\begin{align*}
		P_{H} & = (P_M \backslash m[P_A]) + P_B \\
		E_{H} & = (P_M \backslash m[P_A]) \\
	\end{align*}
  \noindent This removes all of $G$ from $M$, but re-introduces the
  vertices from the boundary graph $B$ of $G$. Let $b : B \rightarrow G$ be the boundary map of
  $G$.  The source and target maps are then defined as follows for each $e \in E_{H}$, which includes
  edges that formerly went to boundary points in $G$:
	\begin{align*}
		s_{H}(e) & =  \begin{cases}
						  b^o(p) &
						    \textit{if}\ p \in \Out(G)\ \textit{and}\ m(p) = s_M(e) \\
						  s_M(e) & \textit{otherwise}
					   \end{cases} \\
		t_{H}(e) & =  \begin{cases}
		  				  b^i(p) &
		    			  	\textit{if}\ p \in \In(G)\ \textit{and}\ m(p) = t_M(e) \\
						  t_M(e) & \textit{otherwise}
					   \end{cases} \\
	\end{align*}
  The process of removing the image of $m$ leaves some edges without a source or a target. In that case, the edge is re-connected to a point in $P_B$. We call the induced embedding, $c : B \hookrightarrow H$, the \emph{coboundary} of $b$ with respect to $m$.  When the embedding of $G$ into $M$ is implicit, we simply write $M-G$.
\end{definition}

For this definition to be valid, we need to show that $H$ is an
open-graph; specifically, that the maps $s_H$ and $t_H$ are total and
well-defined. 
\begin{proof}
  The source map $s_H$ is total because the source of an edge, $e \in
  E_H$, is in the image of $m$ iff it is an output of $G$. Because $m$ is
  mono, $s_H$ is well-defined. Similarly for $t_H$.
\end{proof}


\begin{theorem}\label{lem:subtracts-are-s-po-complements}
  Subtractions by graphs without isolated points are $S$-adhesive pushout complements: given $H := M -_m G$, the following diagram is an $S$-adhesive pushout:
 	\begin{equation}\label{eqn:po-square-exists}
 		\posquare{B}{G}{H}{M}{b}{f}{c}{m}
 	\end{equation}
  where $b$ is the boundary map of $G$, and $c$ is the coboundary of $m$.
\end{theorem}
\begin{proof}
	First, we show $b, c$ is boundary coherent. By the definition of subtraction, for $p \in B$, if
  $b(p)$ is an input, then $c(p)$ is an output. Similarly, if $b(p)$ is an output,
  $c(p)$ is an input. So, $b,c$ satisfies the boundary coherence condition.


	The pushout of $b$ and $c$ is the result of identifying the boundary of
  $G$ with its coboundary in $H$. By case analysis, the resulting graph $M''$ is
  isomorphic to $M$, and for the induced embedding $m''$ of $G$ into
  $M''$, the following diagram commutes:
	\begin{center}
		\ctri{G}{M}{M''}{m}{m''}{\cong}
	\end{center}
	So, diagram (\ref{eqn:po-square-exists}) is also a pushout square for some $f$.
\end{proof}

Since $G$ contains no isolated points, the maps $b$ and $c$ from the theorem are mono. As $B$ is a point graph, this implies $(b,c)$ defines a plugging, and $H := M -_m G$ is the uniquely determined open-graph such that $G \plugw{b,c} H \cong M$.

\section{Rewriting with Open-Graphs}
\label{sec:rewriting}

We now introduce rewrite rules for open-graphs, how they can be applied, under what conditions they can commute with merging and plugging, and how rewrites can themselves be composed. Finally we present a rewrite system called edge-homeomorphism that lets us ignore intermediate edge-points. 

\begin{definition}[Rewrite] \label{def:rewrite-rule}
\begin{changebar}
A span $L \overset{b_1}{\longleftarrow} B \overset{b_2}{\longrightarrow} R$, in which $L$ and $R$ share the same boundary, $B$, by monos $b_1$ and $b_2$, is called a \emph{rewrite rule} and is written $\rwrulew{L}{b_1,b_2}{R}$. The rewrite rule is said to \emph{rewrite} $G$ to $G'$ at a mono $m : L \hookrightarrow G$, called the \emph{matching}, when $G'$ is defined according to the following double pushout: 

\begin{center}
	\begin{tikzpicture}
		\matrix (m) [cdiag] {
        L & B & R \\
        G & G-_{m}L & G'\\
		};
		\path [arrs]
	  	  	(m-1-2) edge [left hook-latex] node [swap] {$b_1$} (m-1-1)
	  	  	(m-1-2) edge [right hook-latex] node {$b_2$} (m-1-3)

	  	  	(m-1-1) edge[right hook-latex] node [swap] {$m$} (m-2-1)
	  	  	(m-1-2) edge[right hook-latex] node {} (m-2-2)
	  	  	(m-1-3) edge[right hook-latex] node {} (m-2-3)

	  	  	(m-2-2) edge [left hook-latex] node {} (m-2-1)
	  	  	(m-2-2) edge [right hook-latex] node {} (m-2-3);
	\NEbracket{(m-2-1)};
    \NWbracket{(m-2-3)};
	\end{tikzpicture} 
\end{center} 
where the left pushout serves to compute the subtraction $G-_{m}L$, and the right pushout the rewritten graph $G'$, which we shall also write as $G[\rwsubstw{L}{b_1,b_2}{R}]_m$. 
\end{changebar}
\end{definition}

Notice that because we require a rule to be a span of monos, there can be no isolated points in $L$ or $R$, as the boundary map is 2-1 on isolated points.




\begin{example}[Circles]  \label{ex:circle-rewrite}
We now return to the challenging example introduced at the end
of \S\ref{sec:motiv}. We will rewrite the graph
\begin{tikzpicture}[circuit, node distance=0.5em and 0.7em]
\node (x1) [] {};
\node (g) [andg, right=of x1, inner sep=0.4em] {};
\node (x2) [right=of g] {};
\node (x3) [below=of g] {};
\draw [dline] (g) -- (x2.center) -- (x3.center) -- (x1.center) -- (g);
\end{tikzpicture} 
by 
\begin{tikzpicture}[circuit]
\node (x1) [halfe] {};
\node (g) [andg, right=of x1, inner sep=0.4em] {};
\node (x2) [halfe,right=of g] {};
\draw [dline] (x1) -- (g) -- (x2);
\end{tikzpicture}
$\rwarrow$
\begin{tikzpicture}[circuit]
\node (x1) [halfe] {};
\node (x2) [halfe,right=of x1] {};
\draw [dline] (x1) -- (x2);
\end{tikzpicture}
to get \begin{tikzpicture}[circuit]
		\node [style=none] (0) at (0, 0.25) {};
		\node [style=none] (1) at (0, -0.25) {};
		\draw[in=180, out=180, looseness=2.00, shorten >=-2.5 pt, dline] (1.center) to (0.center);
		\draw[in=0, out=0, looseness=2.00, shorten <=0 pt, shorten >=-0.1 pt, line] (0.center) to (1.center);
\end{tikzpicture}}.
The pushout construction for this rewrite is as follows:

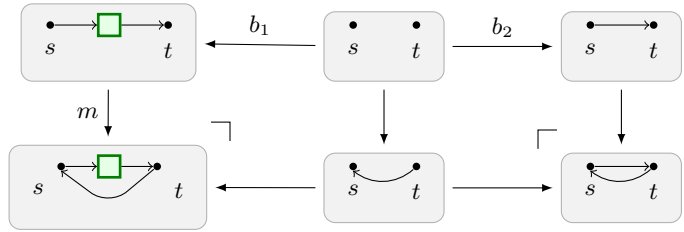
\begin{center}
\begin{tikzpicture}[mcgraph,node distance=5em and 7em]
	\matrix (m) [node distance=2em and 2em, row sep=3em, column sep=5em,inner sep=0.5em] {
{  
  \node (lc) [] {};
  { [node distance=0.5em and 2em]
  \node (lv1) [andg, at=(lc), inner sep=0.4em] {};
  \node (lt1) [ipoint, right=of lv1.center] {};
  \node (ls1) [ipoint, left=of lv1.center] {};
  \node[ipointlabel, below=of ls1] (ls1l) {$s$};  
  \node[ipointlabel, below=of lt1] (lt1l) {$t$};  
  \path[->] (lv1) edge (lt1)
            (ls1) edge (lv1);
  }
}
& 
{ 
  \node[ipoint] (js1) {};  
  \node[ipointlabel, below=of js1] (js1l) {$s$};  
  \node[ipoint, right=of js1] (js2) {};  
  \node[ipointlabel, below=of js2] (js2l) {$t$};  
}
& 
{  
  \node (rs1) [ipoint] {};
  { [node distance=0.5em and 2em]
  \node (rt1) [ipoint, right=of rs1] {};
  \node[ipointlabel, below=of rs1] (rs1l) {$s$};  
  \node[ipointlabel, below=of rt1] (rt1l) {$t$};  
  \path[->] (rs1) edge (rt1);
  }
}
\\
{  
  \node[node distance=3em and 4em] (k1c) {};  
  { [node distance=0.5em and 1.6em]
  \node[andg, inner sep=0.4em,at=(k1c)] (k1x) {}; 
  \node[ipoint, left=of k1x.center] (k1s) {}; 
  \node[ipoint, right=of k1x.center] (k1t) {};   
  \node[ipointlabel,below left=of k1s] (k1sl) {$s$};  
  \node[ipointlabel,below right=of k1t] (k1tl) {$t$};  
  \draw[->] (k1s) -- (k1x);
  \draw[->] (k1x) -- (k1t);
  \node[below=of k1x] (k1y) {}; 
  \draw[rounded corners=2ex,->] (k1t) -- (k1y.center) -- (k1s.south);
  }
}
&
{ 
  \node[ipoint] (m1s) {}; 
  \node[ipoint, right=of m1s] (m1t) {};   
  \node[ipointlabel,below=of m1s] (m1sl) {$s$};  
  \node[ipointlabel,below=of m1t] (m1tl) {$t$};  
  \path[rounded corners=2ex,->]
    (m1t) edge[bend left=40] (m1s.south);
}
 &  
{  
  {\node[ipoint] (k2s) {}; 
  \node[ipoint, right=of k2s] (k2t) {}; 
  \node[ipointlabel,below=of k2s] (k2sl) {$s$};  
  \node[ipointlabel,below=of k2t] (k2tl) {$t$};  
  \path[->] (k2s) edge (k2t);
  \path[->] (k2t) edge[bend left=40] (k2s.south);
  }
}
\\
};

  \begin{pgfonlayer}{background}
  \node[rectangle, fill=black!5, draw=black!30,rounded corners=1ex, fit=(lv1) (ls1) (lt1) (ls1l) (lt1l)] (L) {};  
  \end{pgfonlayer}

  \begin{pgfonlayer}{background}
  \node[rectangle, fill=black!5, draw=black!30,rounded corners=1ex, fit=(rs1) (rt1) (rs1l) (rt1l)] (R) {};  
  \end{pgfonlayer}

  \begin{pgfonlayer}{background}
  \node[rectangle, fill=black!5, draw=black!30,rounded corners=1ex, fit=(k1x) (k1sl) (k1tl) (k1y)] (K1) {};  
  \end{pgfonlayer}

  \begin{pgfonlayer}{background}
  \node[rectangle, fill=black!5, draw=black!30,rounded corners=1ex,
  fit=(m1s) (m1t) (m1sl) (m1tl)] (M) {};  
  \end{pgfonlayer}

  \begin{pgfonlayer}{background}
  \node[rectangle, fill=black!5, draw=black!30,rounded corners=1ex, fit=(js1) (js2) (js2l) (js1l)] (J) {};  
  \end{pgfonlayer}

 \begin{pgfonlayer}{background}
  \node[rectangle, fill=black!5, draw=black!30,rounded corners=1ex, fit=(k2s) (k2t) (k2sl) (k2tl)] (K2) {};
  \end{pgfonlayer}

{ \path[shorten >= 3pt,shorten <= 3pt,arrs]
		  	  	(J) edge node [swap] {$b_1$} (L)
		  	  	(J) edge node {$b_2$} (R)

		  	  	(L) edge node [swap] {$m$} (K1)
		  	  	(J) edge node {} (M)
		  	  	(R) edge node {} (K2)

		  	  	(M) edge node {} (K1)
		  	  	(M) edge node {} (K2);
}
		\NEbracket{($(K1.north east) - (0.2cm,0.2cm)$)};
    \NWbracket{($(K2.north west) - (-0.2cm,0.2cm)$)};
  \end{tikzpicture}
\end{center}
\end{example}

Notice that the above rewrite contains additional intermediate edge-points. Informally, these are intended to be treated as part of the edge. In \S\ref{sec:homeomorphism}, we formalise this idea by introducing rewrite rules that insert and remove these intermediate edge-points.

\subsection{Compatibility}

It may initially be surprising to realise that certain pushouts can prohibit certain rewrites. For example consider the following:

\begin{example}
Let $G :=$
\begin{tikzpicture}[circuit,baseline=-0.25em]
  \node[ipoint] (p1) {}; 
\end{tikzpicture},
$H :=$
\begin{tikzpicture}[circuit,baseline=-0.25em]
\node (x1) [ipoint] {};
\node (g) [andg, right=of x1] {$v$};
\node (x2) [ipoint,right=of g] {};
\draw [dline] (x1) -- (g) -- (x2);
\end{tikzpicture}.
For $K :=$
\begin{tikzpicture}[circuit,baseline=-0.25em]
  \node[ipoint] (p1) {}; 
  \node[ipoint, right=of p1] (p2) {};
\end{tikzpicture},
we can find maps $f : K \rightarrow G, g : K \rightarrow H$ such that the $S$-adhesive pushout
$G +_{f,g} H :=$
\begin{tikzpicture}[circuit,baseline=-0.25em] 
\node (g) [andg] {$v$}; 
\node (i1) [ipoint, left=of g] {};
  \path [dline] (i1) edge[bend left=20] (g);
  \path [dline] (g) edge[bend left=20] (i1);
\end{tikzpicture}. While the left-hand side of the rewrite 
\begin{tikzpicture}[circuit,baseline=-0.25em]
\node (x1) [ipoint] {};
\node (g) [andg, right=of x1] {$v$};
\node (x2) [ipoint,right=of g] {};
\draw [dline] (x1) -- (g) -- (x2);
\end{tikzpicture}
$\rwarrow$
\begin{tikzpicture}[circuit,baseline=-0.25em]
\node (x1) [ipoint] {};
\node (x2) [ipoint,right=of x1] {};
\draw [dline] (x1) -- (x2);
\end{tikzpicture} matches $H$, it does not match $G +_{f,g} H$.
\end{example}

We will be primarily concerned with pluggings, and so we now provide a precise definition of what it means for a plugging and a rewrite to be compatible.

\begin{definition}[Compatible] \label{def:merge-rewrite-compatible}
  A plugging $G \plugw{p,q} H$ and a rewrite $G[\rwsubst{L}{R}]_m$
  are said to be \emph{compatible} when there exists a map $\widehat p$ and a matching $\widehat m$, such that
 $(\widehat p, q)$ is a plugging, and:
  \[ G[\rwsubst{L}{R}]_m
       \plugw{\widehat p,q} H \cong (G \plugw{p,q} H)[\rwsubst{L}{R}]_{\widehat m} \]
\end{definition}

We can actually show that \emph{all} pluggings and rewrites are compatible. Before we prove this important theorem, we first show that the boundary of an open-graph in invariant under rewriting.

\begin{theorem}\label{thm:rewrite-preserves-boundary} 
  Rewriting preserves the boundary of an open-graph. 
	Specifically, let the top two squares of the following diagram define the rewrite $G[\rwsubst{L}{R}]_m$:
	
	\begin{center}
	\begin{tikzpicture}
		\matrix (m) [cdiag] {
			L & B       & R         \\
	        G & G -_m L & G[\rwsubst{L}{R}]_m \\
	          & B'      &           \\
		};
		\path [arrs]
	       (m-1-2) edge node [swap] {$b_1$} (m-1-1)
	       (m-1-2) edge node {$b_2$} (m-1-3)

	       (m-2-2) edge node [swap,pos=0.4] {$s$} (m-2-1)
	       (m-2-2) edge node {$s'$} (m-2-3)

	       (m-1-1) edge node [swap] {$m$} (m-2-1)
	       (m-1-2) edge node {$c$} (m-2-2)
	       (m-1-3) edge node {$m'$} (m-2-3)

	       (m-3-2) edge node {$b_1'$} (m-2-1)
	       (m-3-2) edge node [swap] {$k$} (m-2-2)
	       (m-3-2) edge node [swap] {$b_2'$} (m-2-3);
	\NEbracket{(m-2-1)}
	\NWbracket{(m-2-3)}
	\end{tikzpicture}
	\end{center}
	
	Then there exists a map $k$ and a span of boundary maps $b_1'$, $b_2'$ making the bottom two triangles commute.
\end{theorem}
\begin{proof}
	Let $B'$ be the boundary of $G$, and $b_1'$ be its inclusion into $G$. We can show that the boundary of $G$ is in the image of $s$, by the definition of subtraction. If some point $x$ $B'$ is not in the image of $m$, then it is still in $G -_m L$. If it is in the image of $m$, then it must be in the boundary of $L$ in $G$. Since a copy of this boundary is in $G -_m L$, $x$ must be in the image of $s$. Thus, for all $x$ $B'$, there exists $x'$ in $G -_m L$ such that $s(x') = x$. Since $s$ is mono, $x'$ is unique, so let $k$ be defined as the map sending $x$ to $x'$, and let $b_2' = s' k$.
	
	It suffices to show that $s' k$ is a boundary map. If $x$ is an input of $G$, then $k(x)$ is either still an input or becomes an isolated point. In the latter case, it must come from an input of $L$, and hence an input of $R$. Thus $s'(k(x))$ is an input in the combined graph. This follows similarly for outputs. It can also be shown that $s' k$ covers the boundary of $G[\rwsubst{L}{R}]_m$, so it is a boundary map.
\end{proof}



We can now use this theorem and Thm \ref{thm:pushout-and-rewrite} to show not only that pluggings and rewrites are compatible, but explicitly define the maps $\widehat m$ and $\widehat p$ used in Def \ref{def:merge-rewrite-compatible}.

\begin{theorem} \label{thm:plugging-rewrite-compat}
	Rewriting and plugging are compatible. Suppose $(p : K \rightarrow G, q : K \rightarrow H)$ is a plugging and $i$ the embedding of $G$ into $G \plugw{p,q} H$. Let $m : L \rightarrow G$ be a matching of a rewrite rule $\rwrule{L}{R}$. Then there exists $\widehat p$ such that $(\widehat p, q)$ is a plugging, $\widehat m := i m$ is a matching, and
\[ 
G[\rwsubst{L}{R}]_m \plugw{\widehat p,q} H \cong (G \plugw{p,q} H)[\rwsubst{L}{R}]_{im}.
\] 
\end{theorem}
\begin{proof}
	Since $p$ is mono, $i$ is mono, so is $i m$. Since $p,q$ is a plugging, the map $p$ factors through the boundary map $b_1' : B' \rightarrow G$. Let $r$ be a map such that $p = b_1' r$. For $k$ and $b_2'$ defined as in Thm \ref{thm:rewrite-preserves-boundary}, let $p' = k r$ and $\widehat p = b_2' r$. Then by the above theorem, the following diagram commutes:
	\begin{center}
		\begin{tikzpicture}
		\matrix (m) [cdiag] {
		G                   &   &   \\
		G -_m L             & P & H \\
		G[\rwsubst{L}{R}]_m &   &   \\
		};
		\path [arrs]
		  (m-2-2) edge [bend right] node [swap] {$p$} (m-1-1)
		  (m-2-2) edge node [swap] {$p'$} (m-2-1)
		  (m-2-2) edge [bend left] node {$\widehat p$} (m-3-1)
		  (m-2-1) edge node {$s$} (m-1-1)
		  (m-2-1) edge node [swap] {$s'$} (m-3-1)
		  (m-2-2) edge node {$q$} (m-2-3);
		\end{tikzpicture}
	\end{center}
	
  If $p(x)$ is an input, then $p'(x)$ and $\widehat p(x)$ are both inputs, and similarly for outputs. Therefore $(p,q)$, $(p',q)$, and $(\widehat p, q)$ are all boundary-coherent spans, and hence $S$-adhesive spans. The result then follows from Thm \ref{thm:pushout-and-rewrite}. 
\end{proof}

\subsection{Composition of Rewrites}
\label{sub:rewrite-composition}

%

\begin{definition}[Extension]
  Given an open-graph $G$ and a rewrite rule $r :=
  \rwrule{L}{R}$, when $r$ rewrites
  $G$ to $G[\rwsubst{L}{R}]_m$, then the rewrite rule
  $\rwrulew{G}{b_1,b_2}{G[\rwsubst{L}{R}]_m}$ is called the
  \emph{extension} of $r$ by $m$, and written $\exten{r}{m}$. 
\end{definition}

Notice that this is a well defined rewrite rule because the boundary span $b_1, b_2$ is uniquely defined by Thm~\ref{thm:rewrite-preserves-boundary}.  

\begin{example}
Returning to Example~\ref{ex:circle-rewrite}, the extension of this rewrite is
the span:
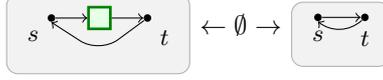
\begin{center}
\begin{tikzpicture}[circuit, node distance=0.7em and 0.7em]
{  
  \node[node distance=3em and 4em] (k1c) {};  
  { [node distance=0.5em and 1.6em]
  \node[andg, inner sep=0.4em,at=(k1c)] (k1x) {}; 
  \node[ipoint, left=of k1x.center] (k1s) {}; 
  \node[ipoint, right=of k1x.center] (k1t) {};   
  \node[ipointlabel,below left=of k1s] (k1sl) {$s$};  
  \node[ipointlabel,below right=of k1t] (k1tl) {$t$};  
  \draw[->] (k1s) -- (k1x);
  \draw[->] (k1x) -- (k1t);
  \node[below=of k1x] (k1y) {}; 
  \draw[rounded corners=2ex,->] (k1t) -- (k1y.center) -- (k1s.south);
  }
}
\end{tikzpicture} $\leftarrow \emptyset \rightarrow $
\begin{tikzpicture}[circuit, node distance=1.4em and 1.4em]
  {\node[ipoint] (k2s) {}; 
  \node[ipoint, right=of k2s] (k2t) {}; 
  \node[ipointlabel,below=of k2s] (k2sl) {$s$};  
  \node[ipointlabel,below=of k2t] (k2tl) {$t$};  
  \path[->] (k2s) edge (k2t);
  \path[->] (k2t) edge[bend left=40] (k2s.south);
  }
\end{tikzpicture}
\end{center}
\noindent where the shared boundary is the empty open-graph, denoted by $\emptyset$.
\end{example}

Extension provides a construction of the rewrite relation $\rewritesto_{\mathbb S}$ for open-graphs. That is, $G \rewritesto_{\mathbb S} H$ precisely when there exists a rule in $\mathbb S$ that can be extended to $\rwrule{G}{H}$.



We now show how rules can be directly combined using the underlying operations on open-graphs. First, note that any rewrite rule $\rwrulew{L}{b_1,b_2}{R}$ has an opposite rewrite $\rwrulew{R}{b_2,b_1}{L}$ given by flipping the span around. Also, for two rewrite rules $\rwrulew{L}{b_1,b_2}{R}$ and $\rwrulew{R}{b_2,b_3}{R'}$, we can, by abuse of notation, assume they are both spans over the same boundary graph, and write $L \rwarrow R \rwarrow R'$ for the rule $\rwrulew{L}{b_1,b_3}{R'}$.

\begin{definition}[Sequential Composition]
  Given rewrite rules $r_1 := \rwrule{L_1}{R_1}$ and $r_2 :=
  \rwrule{L_2}{R_2}$ and a
  merged graph $M := R_1 \mergew{k_1,k_2} L_2$, then the \emph{sequential
    composition} of $r_1$ and $r_2$ at $k_1,k_2$ is the rewrite rule defined by:
  \[
  (\seqcompww{r_1}{k_1,k_2}{r_2}) := \rwrulew{(M[\rwsubst{R_1}{L_1}]_{m_1})}{b_1,b_2}{(M[\rwsubst{L_2}{R_2}]_{m_2})}
  \]
where $m_1$ is the embedding of $R_1$ into $M$, $m_2$ is the
embedding of $L_2$ into $M$, and $(b_1,b_2)$ is the following boundary span induced by two applications of Thm~\ref{thm:rewrite-preserves-boundary}:
\begin{center}
	\begin{tikzpicture}
	\matrix (m) [cdiag] {
	M[\rwsubst{R_1}{L_1}]_{m_1} & \cdot &
	M &
	\cdot & M[\rwsubst{L_2}{R_2}]_{m_2} \\
	& & B & & \\
	};
	\path [arrs]
	  (m-1-2) edge (m-1-1)
	  (m-1-2) edge (m-1-3)
	  (m-1-4) edge (m-1-3)
	  (m-1-4) edge (m-1-5)

	  (m-2-3) edge node {$b_1$} (m-1-1)
	  (m-2-3) edge (m-1-2)
	  (m-2-3) edge (m-1-3)
	  (m-2-3) edge (m-1-4)
	  (m-2-3) edge node [swap] {$b_2$} (m-1-5);
	\end{tikzpicture}
\end{center}
\end{definition}

Sequential composition, unlike extension, is a direct operation on two rewrites to produce a new rewrite. This provides an algorithm for deriving new graphical equations, as we did in \S\ref{sec:motiv}.  Sequential composition is correct in the sense that it does nothing more than $\rewritetrans_{\mathbb S}$.

\begin{theorem}[Soundness] \label{prop:soundness} if
  $(\seqcompww{r_1}{k_1,k_2}{r_2}) := \rwrule{G}{G'}$ is
  a rewrite; then there exists a graph $H$, and monos $m_1$ and $m_2$
  such that $G \stackrel{\exten{r_1}{m_1}}{\longrwarrow} M
  \stackrel{\exten{r_2}{m_2}}{\longrwarrow} G'$.
\end{theorem}
\begin{proof}
  Let $r_1 := \rwrule{L_1}{R_1}$ and $r_2 :=
  \rwrule{L_2}{R_2}$. Let $M$ be exactly $R_1 \mergew{k_1,k_2}
  L_2$. The embedding of $L_2$ into $M$ defines $m_2$. Thus what we have left to prove is that there is an $m_1$ such
  that
  $M[\rwsubst{R_1}{L_1}]_{m_1'}[\rwsubst{L_1}{R_1}]_{m_1}
  \cong M$, where $m_1'$ is the embedding of $R_1$ into $M$. This follows directly from expanding the equation into
  subtractions and mergings, and then recalling that subtractions are
  pushout complements.
\end{proof}

Sequential composition of rewrites is also complete in the sense that 
many rewrites under various extensions can also be represented as the
sequential composition of the rewrites under a single extension.

\begin{theorem}[Completeness]
  if $M_1 \stackrel{\exten{r_1}{m_1}}{\longrwarrow} M_2
  \stackrel{\exten{r_2}{m_2}}{\longrwarrow} M_3$ then there exists
  an $m'$ and $k_1,k_2$ such that $\exten{(\seqcompww{r_1}{k_1,k_2}{r_2})}{m'} : M_1
  \rwarrow M_3$
\end{theorem}
\begin{proof}
  Let $r_1 := \rwrule{L_1}{R_1}$ and $r_2 :=
  \rwrule{L_2}{R_2}$. There is a matching of both $R_1$ and
  $L_2$ in $M_2$. The overlap of these matchings forms a graph $K$ which defines the
  boundary coherent pair $k_1,k_2$, of $K$ into $R_1$ and $L_2$ respectively. We then have $M_2 \cong (R_1 \plugw{p_1,p_2} L_2') \plugw{q_1,q_2} M_2'$, where $L_2' := L_2
  - K$ and $M_2' := (M_2 - R_1) - L_2'$. Thus $M_1 \cong (L_1 \plugw{p_1,p_2} L_2')
  \plugw{q_1,q_2} M_2'$, and $m'$ is simply the embedding of $L_1 \plugw{p_1,p_2} L_2'$
  into $M_1$.
\end{proof}

\subsection{Edge-Homeomorphism} \label{sec:homeomorphism}

Although we have defined everything discreetly so far, open-graphs admit a topological interpretation. Edges can be thought of as copies of the unit interval $[0,1] \subset \mathbb R$, considered as an oriented manifold. Vertices are distinguished points, to which we ascribe semantic meaning, and edges represent ``gluing'' intervals end-to-end, or gluing a vertex on to one edge of an interval. We now briefly elaborate on this idea before introducing a rewrite rule to act in a way analogously to homeomorphism.

\begin{definition}\label{def:wire}
	For an open-graph $G$, a \emph{wire} $W$ in $G$ is a set of connected edge-points, which contains at least one edge, and may also include vertices at its start and end. If a vertex is connected to either end of $W$ in $G$, it is called an \emph{endpoint} of $W$.
\end{definition}

As graphs, wires can be chains or circles. For any wire $W$, we can define an (oriented) manifold $M(W)$ as a quotient over the disjoint union of real unit intervals $\coprod [0,1]_e$, indexed by the edges $e$ in $W$. Whenever there are two edges $e_1$ and $e_2$ in $W$ where $t(e_1) = s(e_2)$, we identify $1 \in [0,1]_{e_1}$ with $0 \in [0,1]_{e_2}$. The unit intervals $[0,1]_e$ then form a collection of charts for $M(W)$ and give an orientation, so $M(W)$ forms an oriented manifold.

\begin{definition}\label{def:homeo-with-manifolds}
	Two graphs $G$ and $G'$ are called \emph{edge-homeomorphic} if $G'$ can be obtained from $G$ by replacing any wire $W$ with a new wire $W'$ where there exists a homeomorphism of oriented manifolds $M(W) \cong M(W')$.
\end{definition}

This topological intuition is encoded discretely in open-graphs as a rewrite system called \emph{edge-homeomorphism}.  

\begin{definition}[Edge-Homeomorphism]
  \label{def:homeo-rewrite-system}
  The following rewrite system is called \emph{edge-homeomorphism} and denoted by $\mathbb H$:
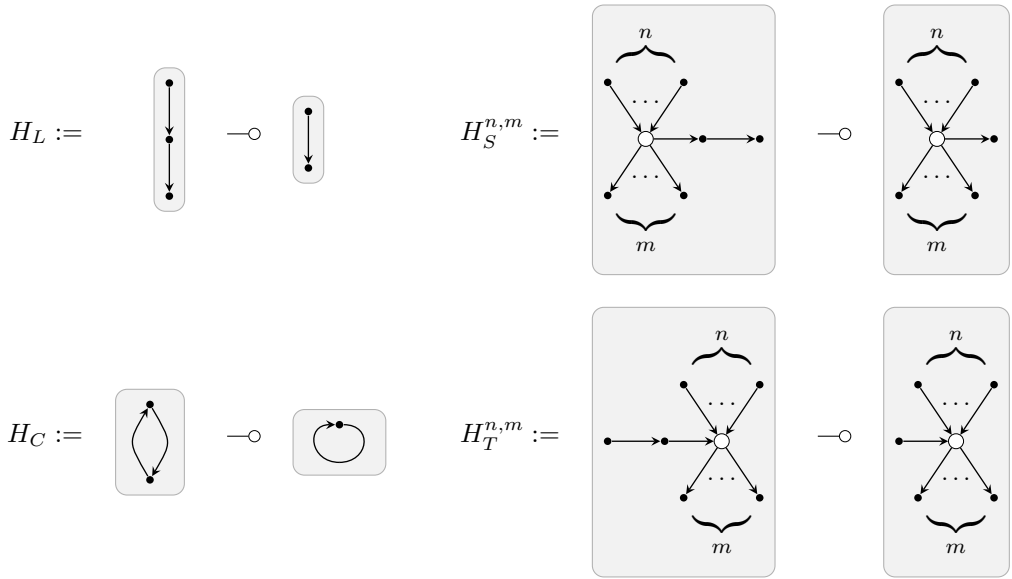
\begin{center}
\begin{tabular}{rrclcrrcl}
$H_L :=$ & \begin{tikzpicture}[circuit]
	\begin{pgfonlayer}{nodelayer}
		\node [style=ipoint] (0) at (-2, 1.5) {};
		\node [style=ipoint] (1) at (-2, 0.75) {};
		\node [style=ipoint] (2) at (-2, 0) {};
	\end{pgfonlayer}
	\begin{pgfonlayer}{edgelayer}
		\draw[dline] (1) to (2);
		\draw[dline] (0) to (1);
	\end{pgfonlayer}
\end{tikzpicture}} & $\rwarrow$ & \begin{tikzpicture}[circuit]
	\begin{pgfonlayer}{nodelayer}
		\node [style=ipoint] (0) at (-2, 1.5) {};
		\node [style=ipoint] (1) at (-2, 0.75) {};
	\end{pgfonlayer}
	\begin{pgfonlayer}{edgelayer}
		\draw[dline] (0) to (1);
	\end{pgfonlayer}
\end{tikzpicture}} 
& & $H^{n,m}_S :=$ & \begin{tikzpicture}[circuit]
	\begin{pgfonlayer}{nodelayer}
		\node [style=brace] (0) at (-2, 2.25) {$\overbrace{\quad}^n$};
		\node [style=ipoint] (1) at (-2.5, 1.75) {};
		\node [style=ipoint] (2) at (-1.5, 1.75) {};
		\node [style=ellipses] (3) at (-2, 1.5) {$\hdots$};
		\node [style=valg] (4) at (-2, 1) {$\ $};
		\node [style=ipoint] (5) at (-1.25, 1) {};
		\node [style=ipoint] (6) at (-0.5, 1) {};
		\node [style=ellipses] (7) at (-2, 0.5) {$\hdots$};
		\node [style=ipoint] (8) at (-2.5, 0.25) {};
		\node [style=ipoint] (9) at (-1.5, 0.25) {};
		\node [style=brace] (10) at (-2, -0.25) {$\underbrace{\quad}_m$};
	\end{pgfonlayer}
	\begin{pgfonlayer}{edgelayer}
		\draw[dline] (5) to (6);
		\draw[dline] (1) to (4);
		\draw[dline] (4) to (5);
		\draw[dline] (2) to (4);
		\draw[dline] (4) to (8);
		\draw[dline] (4) to (9);
	\end{pgfonlayer}
\end{tikzpicture}} & $\rwarrow$ & \begin{tikzpicture}[circuit]
	\begin{pgfonlayer}{nodelayer}
		\node [style=brace] (0) at (-2, 2.25) {$\overbrace{\quad}^n$};
		\node [style=ipoint] (1) at (-2.5, 1.75) {};
		\node [style=ipoint] (2) at (-1.5, 1.75) {};
		\node [style=ellipses] (3) at (-2, 1.5) {$\hdots$};
		\node [style=valg] (4) at (-2, 1) {$\ $};
		\node [style=ipoint] (5) at (-1.25, 1) {};
		\node [style=ellipses] (6) at (-2, 0.5) {$\hdots$};
		\node [style=ipoint] (7) at (-2.5, 0.25) {};
		\node [style=ipoint] (8) at (-1.5, 0.25) {};
		\node [style=brace] (9) at (-2, -0.25) {$\underbrace{\quad}_m$};
	\end{pgfonlayer}
	\begin{pgfonlayer}{edgelayer}
		\draw[dline] (4) to (8);
		\draw[dline] (1) to (4);
		\draw[dline] (2) to (4);
		\draw[dline] (4) to (5);
		\draw[dline] (4) to (7);
	\end{pgfonlayer}
\end{tikzpicture}} 
\\
& & & & & & & & \\
$H_C :=$ & \begin{tikzpicture}[circuit]
		\node [style=ipoint] (0) at (-2.75, 2) {};
		\node [style=ipoint] (1) at (-2.75, 1) {};
		\draw[looseness=1.50, dline, bend left] (1) to (0);
		\draw[looseness=1.50, dline, bend left] (0) to (1);
\end{tikzpicture}}  & $\rwarrow$ & \begin{tikzpicture}[circuit]
		\node [style=ipoint] (0) at (0, 0.25) {};
		\node [style=none] (1) at (0, -0.25) {};
		\draw[in=180, out=180, looseness=2.00, dline] (1.center) to (0);
		\draw[in=0, out=0, looseness=2.00, shorten <=0 pt, shorten >=-0.1 pt, line] (0) to (1.center);
\end{tikzpicture}} & & 
$H^{n,m}_T :=$ & \begin{tikzpicture}[circuit]
	\begin{pgfonlayer}{nodelayer}
		\node [style=brace] (0) at (-2, 2.25) {$\overbrace{\quad}^n$};
		\node [style=ipoint] (1) at (-2.5, 1.75) {};
		\node [style=ipoint] (2) at (-1.5, 1.75) {};
		\node [style=ellipses] (3) at (-2, 1.5) {$\hdots$};
		\node [style=ipoint] (4) at (-3.5, 1) {};
		\node [style=ipoint] (5) at (-2.75, 1) {};
		\node [style=valg] (6) at (-2, 1) {$\ $};
		\node [style=ellipses] (7) at (-2, 0.5) {$\hdots$};
		\node [style=ipoint] (8) at (-2.5, 0.25) {};
		\node [style=ipoint] (9) at (-1.5, 0.25) {};
		\node [style=brace] (10) at (-2, -0.25) {$\underbrace{\quad}_m$};
	\end{pgfonlayer}
	\begin{pgfonlayer}{edgelayer}
		\draw[dline] (2) to (6);
		\draw[dline] (5) to (6);
		\draw[dline] (6) to (8);
		\draw[dline] (1) to (6);
		\draw[dline] (4) to (5);
		\draw[dline] (6) to (9);
	\end{pgfonlayer}
\end{tikzpicture}} & $\rwarrow$ & \begin{tikzpicture}[circuit]
	\begin{pgfonlayer}{nodelayer}
		\node [style=brace] (0) at (-2, 2.25) {$\overbrace{\quad}^n$};
		\node [style=ipoint] (1) at (-2.5, 1.75) {};
		\node [style=ipoint] (2) at (-1.5, 1.75) {};
		\node [style=ellipses] (3) at (-2, 1.5) {$\hdots$};
		\node [style=ipoint] (4) at (-2.75, 1) {};
		\node [style=valg] (5) at (-2, 1) {$\ $};
		\node [style=ellipses] (6) at (-2, 0.5) {$\hdots$};
		\node [style=ipoint] (7) at (-2.5, 0.25) {};
		\node [style=ipoint] (8) at (-1.5, 0.25) {};
		\node [style=brace] (9) at (-2, -0.25) {$\underbrace{\quad}_m$};
	\end{pgfonlayer}
	\begin{pgfonlayer}{edgelayer}
		\draw[dline] (5) to (8);
		\draw[dline] (1) to (5);
		\draw[dline] (4) to (5);
		\draw[dline] (2) to (5);
		\draw[dline] (5) to (7);
	\end{pgfonlayer}
\end{tikzpicture}} 
\end{tabular}
\end{center}
\end{definition}

Applying edge-homeomorphism rewrites to a graph, from left to right,
is called \emph{contracting}. Applying them from right to left
is called \emph{expanding}. Edge-homeomorphism allows
arbitrarily many edge-points to be inserted and removed from paths of
connected edge-points. If $G$ rewrites to $H$ using zero or more edge
homeomorphism rewrites, we say $H$ is an \emph{edge contraction} of
$G$.

\begin{lemma}
	The rewrite system $\mathbb H$ is confluent and terminating.
\end{lemma}
\begin{proof}
  Termination comes from observing that each
  contraction of a morphism decreases the number of
  edge-points. Confluence comes from observing that any two
  contractions result in isomorphic graphs independently of the order
  they are applied (there are no critical pairs).
\end{proof}

Considering graphs modulo edge-homeomorphism corresponds to ignoring the intermediate edge-points. Returning to Example~\ref{ex:circle-rewrite}, the resulting circle with two edge-points can now be contracted to a circle with a single edge-point.

\section{Typed Open-Graphs}
\label{sec:graph-lang}

We now generalise our definition of open-graph by showing how it can be parametrised by a `graphical signature' to form a notion of typed open-graphs. The graphical signature defines the types and arities of vertices, as well as the types of edges which can be used. This generalised construction makes use of more sophisticated type-graphs which can themselves be embedded into the basic case of open-graphs. This lets us build a selective adhesive functor through which rewriting properties are inherited in typed open-graphs.

\begin{definition}
For a fixed set $O$, let $O^*$ be the set of finite
lists of $O$. For another set $A$, a function $T :
A \rightarrow O^* \times O^*$ is called a
\emph{graphical signature}. $T$ should be thought of as a function
assigning input and output types to each element in $ A$. 
\end{definition}

\begin{example}
For instance, a function $T$ defined as
\[
T :: \begin{cases}
 \textrm{f} & \mapsto \left(\,
    \textsf{[ A, B, C ]}, \textsf{[ F, G ]}\,\right) \\
 \textrm{g} & \mapsto \left(\, \textsf{[ E ]}, \textsf{[ B ]}\, \right) \\
\end{cases}
\]
can be visualised as a set of ``boxes'':

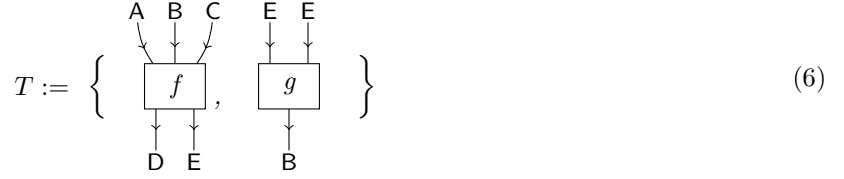
\begin{equation}\label{fig:generators}
\beginpgfgraphicnamed{generators}
\begin{tikzpicture}
	\begin{pgfonlayer}{nodelayer}
		\node [style=elabel, font=\small \sf] (0) at (-1.25, 1) {A};
		\node [style=elabel, font=\small \sf] (1) at (-0.75, 1) {B};
		\node [style=elabel, font=\small \sf] (2) at (-0.25, 1) {C};
		\node [style=elabel, font=\small \sf] (3) at (0.5, 1) {E};
		\node [style=elabel, font=\small \sf] (4) at (1, 1) {E};
		\node [style=none] (5) at (-1, 0.25) {};
		\node [style=none] (6) at (-0.75, 0.25) {};
		\node [style=none] (7) at (-0.5, 0.25) {};
		\node [style=none] (8) at (0.5, 0.25) {};
		\node [style=none] (9) at (1, 0.25) {};
		\node [style=none] (10) at (-2.5, 0) {$T := $};
		\node [style=none] (11) at (-1.75, 0) {$\left\{\vphantom{\left.\right\}\int^{\int^{\int^{\int}}}}\right.$};
		\node [style=square box, minimum width=8 mm] (12) at (-0.75, 0) {f};
		\node [style=square box, minimum width=8 mm] (13) at (0.75, 0) {g};
		\node [style=none] (14) at (1.75, 0) {$\left.\vphantom{\left\{\right.\int^{\int^{\int^{\int}}}}\right\}$};
		\node [style=none] (15) at (-1, -0.25) {};
		\node [style=none] (16) at (-0.5, -0.25) {};
		\node [style=none, anchor=west] (17) at (-0.25, -0.25) {,};
		\node [style=none] (18) at (0.75, -0.25) {};
		\node [style=elabel, font=\small \sf] (19) at (-1, -1) {D};
		\node [style=elabel, font=\small \sf] (20) at (-0.5, -1) {E};
		\node [style=elabel, font=\small \sf] (21) at (0.75, -1) {B};
	\end{pgfonlayer}
	\begin{pgfonlayer}{edgelayer}
		\draw[bend left=15, diredge] (2) to (7.center);
		\draw[diredge] (18.center) to (21);
		\draw[bend right=15, diredge] (0) to (5.center);
		\draw[diredge] (15.center) to (19);
		\draw[diredge] (4) to (9.center);
		\draw[diredge] (16.center) to (20);
		\draw[diredge] (3) to (8.center);
		\draw[diredge] (1) to (6.center);
	\end{pgfonlayer}
\end{tikzpicture}}
\endpgfgraphicnamed
\end{equation}
\end{example}

\begin{remark}
	Graphical signatures are essentially what Selinger calls a \emph{monoidal signature} \cite{selinger2009survey} and Joyal and Street call a \emph{tensor scheme} \cite{Joyal:1991p1143}. We shall see in \S\ref{sec:monoidal-theories} the relationship between these maps and the construction of free monoidal categories.
\end{remark}

For a graphical signature $T$, we can form a \emph{typegraph} \typegraph{} as follows. It has as vertices $O + A$, where every $o \in O$ has a self-loop. For $a \in A$, $T(a)$ is a pair of words $D, C$, defining the domain and codomain of $a$; in particular defining the types of the inputs and outputs of $a$ respectively. For each $d$ in $D$, \typegraph{} has an edge from $d$ to $a$. For each $c$ in $C$, \typegraph{} has an edge from $a$ to $c$. Note that the in-edges and out-edges of each vertex in \typegraph{} have a natural total order given by their word order.

\begin{example}
	$T$ defined as in (\ref{fig:generators}) defines the typegraph \typegraph{}:
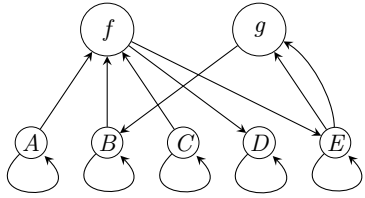
	\begin{center}
\beginpgfgraphicnamed{typegraph}
\begin{tikzpicture}[dline, mcgraph]
	\begin{pgfonlayer}{nodelayer}
		\node [style=valg, minimum width=7 mm] (0) at (-1, 0.75) {f};
		\node [style=valg, minimum width=7 mm] (1) at (1, 0.75) {g};
		\node [style=valg] (2) at (-2, -0.75) {A};
		\node [style=valg] (3) at (-1, -0.75) {B};
		\node [style=valg] (4) at (0, -0.75) {C};
		\node [style=valg] (5) at (1, -0.75) {D};
		\node [style=valg] (6) at (2, -0.75) {E};
	\end{pgfonlayer}
	\begin{pgfonlayer}{edgelayer}
		\draw[out=225, in=315, loop] (6) to ();
		\draw (6) to (1);
		\draw[out=225, in=315, loop] (4) to ();
		\draw (0) to (5);
		\draw[out=225, in=315, loop] (5) to ();
		\draw[looseness=0.75, bend right] (6) to (1);
		\draw (3) to (0);
		\draw (0) to (6);
		\draw[out=225, in=315, loop] (2) to ();
		\draw (2) to (0);
		\draw (1) to (3);
		\draw (4) to (0);
		\draw[out=225, in=315, loop] (3) to ();
	\end{pgfonlayer}
\end{tikzpicture}}
\endpgfgraphicnamed
	\end{center}
\end{example}

\begin{definitions}[Typegraph Notation]
	For a \typegraph-graph $(G, \tau)$, points $p \in \tau^{-1}(O)$ are called \emph{edge-points}. All other points are called \emph{vertices}.
\end{definitions}

For such graphs, we want to have a property even stronger than fullness on vertices. Whereas in the previous section, adjacent edges of $f(v)$ only had to be \emph{covered} by $f$, here they must be in 1-to-1 correspondence with the adjacent edges of $v$. For some vertex $v$ in $G$, we call the set of adjacent edges $N(v)$ its \emph{edge neighbourhood}, and define local isomorphism as follows:

\begin{definition}[Local Isomorphism] \label{def:local-iso}
	A map $f : G \rightarrow H$ is called a \emph{local isomorphism}, for every vertex $v \in G$, the edge function of $f$ restricts to an bijection $f^v : N(v) \overset{\sim}{\rightarrow} N(f(v))$.
\end{definition}

In particular, we can regard the type map $\tau : G \rightarrow \typegraph$ as an arrow from $\tau$ to $1_{\typegraph}$ in the slice category \catTGSlice{}, and ask that it be a local isomorphism.

%

Let $(\catTGSlice)_{\cong}$ be the subcategory of \catTGSlice{} whose objects are pairs $(G, \tau : G \rightarrow \typegraph)$ where $\tau$ is a local isomorphism, and whose arrows are local isomorphisms. We can show this subcategory is in fact full.

\begin{lemma} \label{lem:local-iso-colimit}
	$(\catTGSlice)_{\cong}$ is a full subcategory of \catTGSlice{}.
\end{lemma}

\begin{proof}
	Let $(G, \tau_G)$, $(H, \tau_H)$ be \typegraph-graphs, where $\tau_G$ and $\tau_H$ are both local isomorphisms. For any $f : (G,\tau_G) \rightarrow (H,\tau_H)$ in \catTGSlice, the following diagram commutes:
	\begin{center}
		\ctri{G}{\typegraph}{H}{\tau_G}{f}{\tau_H}
	\end{center}
	
	Thus, for any $v$ in $G$ we get this triangle in \catSet:
	\begin{center}
		\ctri{N(v)}{N(\tau_G(v))}{N(f(v))}{\tau_G^v}{f^v}{\tau_H^{f(v)}}
	\end{center}
	
	Since $\tau_G^v$ and $\tau_H^{f(v)}$ are both bijections, $f^v$ is a bijection, so $(\catTGSlice)_{\cong}$ is a full subcategory.	
	%
	%
\end{proof}

Note that for any \typegraph{}, there is a graph homomorphism $\kappa
: \typegraph \rightarrow 2_{\mathcal G}$ sending every point in
$O$ to $\epsilon$ and every point in $A$ to
$V$. Post-composing each object in $\catTGSlice$ with $\kappa$ yields the forgetful functor:
\[ U_{\kappa} : \catTGSlice \rightarrow \catGraphSlice. \]
In particular, this sends an object $\tau : G \rightarrow \typegraph$ in \catTGSlice{} to an object $\kappa \circ \tau$ in \catGraphSlice{}.

\begin{definition}[Open \typegraph-graph] \label{def:open-t-graph}
	A \typegraph-graph $G$ is called an open \typegraph-graph if $U_\kappa(G) \in \catGraphSlice$ is an open-graph. The category $\catOGraphTG$ the full subcategory of $(\catTGSlice)_{\cong}$ whose objects are open-graphs.
\end{definition}

Note that local isomorphisms are, in particular, full on vertices, so the forgetful functor $U_\kappa$ restricts to another functor
\[ U : \catOGraphTG \rightarrow \catOGraph. \]

%

\begin{lemma}\label{lem:ographtg-monos}
	Monos in \catOGraphTG{} are injective maps.
\end{lemma}

\begin{proof}
	Suppose $m : G \rightarrow H$ in \catOGraphTG{} is not injective. If $m$ takes two distinct edges $e_1$ and $e_2$ to a single edge, then suppose the source of $e_1$ (and hence of $e_2$) is an edge-point. Then, since $G$ is an open-graph, it must take two edge-points to a single edge-point in $H$. Otherwise, suppose it is a vertex, then by local isomorphism, $m$ must take two distinct vertices on to a single vertex. Thus is suffices to only consider points.
	
	If $m$ takes two distinct vertices $v_1$, $v_2$ in $G$ to a single vertex in $H$, then let $K$ be the subgraph of $G$ consisting of just $v_1$ and its neighbourhood. If $m$ takes two distinct edge-points to a single edge-point in $H$, then let $K$ be a graph consisting of a single edge-point. In either case, there are at least two distinct maps $f,g : K \rightarrow G$ such that $m f = m g$.
\TODO{say what this contradicts?}
\end{proof}

\begin{theorem}\label{thm:ographtg-embedding-functor}
	The embedding functor $S' : \catOGraphTG \hookrightarrow \catTGSlice$ is a selective adhesive functor.
\end{theorem}

\begin{proof}
	From Lem \ref{lem:ographtg-monos}, $S'$ preserves monos. Creation of isomorphisms follows from the fact that all isomorphisms are local isomorphisms and the property of being an open-graph is invariant under isomorphism. Faithfulness and reflection of pushouts follows from being a full subcategory embedding.
\end{proof}

\begin{definition}[Boundary-coherence in $\catOGraphTG$]
	A span $A \overset{f}{\leftarrow} B \overset{g}{\rightarrow} C$ in $\catOGraphTG$ is called \emph{boundary-coherent} if its image under $U$ is boundary-coherent in \catOGraph{}.
\end{definition}

\begin{theorem}\label{thm:colimits}
	Boundary-coherent spans in $\catOGraphTG$ are $S'$ adhesive.
\end{theorem}

\begin{proof}
	We prove this property by using the two embeddings and two forgetful functors.
	\begin{center}
		\csquare
		{\catOGraphTG}{\catTGSlice}
		{\catOGraph}{\catGraphSlice}
		{S'}{S}{U}{U_\kappa}
	\end{center}
	
	Let $f,g$ be a boundary-coherent span in $\catOGraphTG$, and let the following square be its pushout in $(\catTGSlice)_{\cong}$.
	\begin{center}
		\csquare{S'(A)}{S'(B)}{S'(C)}{D}{S'(f)}{p_1}{S'(g)}{p_2}
	\end{center}
	
	Since $\catOGraphTG$ is a full subcategory of $\catTGSlice$, it suffices to show that $D$ is in $\catOGraphTG$. By definition, $U(f), U(g)$ is boundary-coherent and hence $S$-adhesive in \catOGraph{}, so its pushout $D'$ exists and $S$ preserves it. $S(D')$ is a pushout of
	\[ (SU(f), SU(g)) = (U_\kappa S' (f),U_\kappa S'(g)) \]
	$U_\kappa(D)$ is the pushout of the RHS, so by uniqueness of pushouts, $S(D') \cong U_\kappa(D)$. $D'$ is an open-graph in \catOGraph{}, so $D$ is an open-graph in $\catOGraphTG$.
\end{proof}

Boundary maps are defined as in \catOGraph{}. The construction of subtraction carries over verbatim, and is preserved by $U$. The uniqueness of pushout complements follows from adhesiveness of \catTGSlice{}.

\section{Monoidal Theories}
\label{sec:monoidal-theories}

Plugging gives us a tool for composing graphs. We can take this a step further and discuss composing graphs in a \emph{categorical} sense, using cospan categories over \catOGraph{} or \catOGraphTG{}. For our purposes, we shall focus on the latter.

For a graphical signature $T : A \rightarrow O^* \times O^*$, we construct the category $\DCsp(\catOGraphTG)$ of \emph{directed cospans} as follows. Its objects are words in $O^*$. Equivalently, they are point graphs in \catOGraphTG, where the points are given a total order. An arrow $G : X \rightarrow Y$ is a cospan
\[ Y \overset{c}{\rightarrow} G \overset{d}{\leftarrow} X \]
where $G$ doesn't contain any isolated points, $d$ is the inclusion of $\In(G) \cong X$, and $c$ is the inclusion of $\Out(G) \cong Y$.

$\DCsp(\catOGraphTG)$ forms a symmetric monoidal category. Composition of maps $G : A \rightarrow B$ and $H : B \rightarrow C$ is by pushout, which is boundary-coherent by construction. For a point graph $A$, the identity of $A$ in $\DCsp(\catOGraphTG)$ is the cospan given by the identity of $A$ in \catOGraphTG.
\[ A \overset{1}{\longrightarrow} A \overset{1}{\longleftarrow} A \]

The monoidal product is given by coproducts in \catOGraphTG. For cospans $G : A \rightarrow B$, $H : C \rightarrow D$, $G \otimes H$ is the cospan
\[ B + D \overset{o}{\longrightarrow} G + H
         \overset{i}{\longleftarrow} A + C, \]
where $i$ and $o$ are the induced maps of coproducts.

Symmetries $\sigma_{A,B} : A \otimes B \rightarrow B \otimes A$ are built using the induced swap map $\sigma := [i_2, i_1]$, for $i_1$ and $i_2$ the coproduct injections of $A + B$.
\[ B + A \overset{1}{\longrightarrow} B + A
         \overset{\sigma}{\longleftarrow}  A + B \]

\begin{remark}
	$\DCsp(\catOGraphTG)$ is actually a monoidal 2-category, where composition and the monoidal product are only associative up to isomorphism. It has as objects point-graphs, as 1-cells cospans, and as 2-cells \typegraph-graph morphisms. For our purposes, we will work with the ``strictified'' category, where composition and $\otimes$ are both taken to be strictly associative. By a minor abuse of notion, for cospans $\mathcal G$, $\mathcal H$, $\mathcal G \cong \mathcal H$ should be read as a $2$-cell isomorphism in the (non-strict) 2-category.
\end{remark}

\subsection{Rewrite Categories}
\label{sub:rewrite-categories}

\begin{lemma}\label{thm:rewriting-on-cospans}
	Let the following cospan be an arrow in $\DCsp(\catOGraphTG)$.
	\[ \mathcal G := Y \overset{c}{\longrightarrow} G \overset{d}{\longleftarrow} X \]
	
	Let $m$ be a matching of a rewrite $\rwrule{L}{R}$ on $G$. Then for the induced rewrite
	\begin{equation}\label{dia:span-rewrite-compat}
		G \overset{b_1}{\longleftarrow} B
		\overset{b_2}{\longrightarrow} G[\rwsubst{L}{R}]_m
	\end{equation}
	there exists unique $\widehat d$, $\widehat c$ such that
	\[ Y \overset{\widehat c}{\longrightarrow} G[\rwsubst{L}{R}]_m \overset{\widehat d}{\longleftarrow} X \]
	is an arrow in $\DCsp(\catOGraphTG)$ and the following diagram commutes for some maps $d'$ and $c'$.
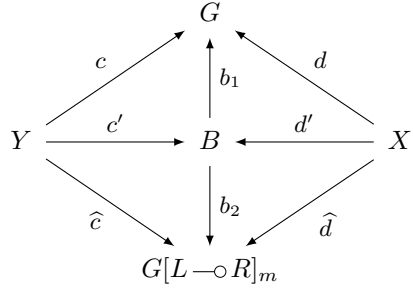
	\begin{center}
		\begin{tikzpicture}
			\matrix (m) [cdiag] {
				  &        G        &   \\
				Y &        B        & X \\
				  & G[\rwsubst{L}{R}]_m &   \\
			};
			\path [arrs]
				(m-2-1) edge node {$c$} (m-1-2)
				(m-2-1) edge node {$c'$} (m-2-2)
				(m-2-1) edge node [swap] {$\widehat c$} (m-3-2)
				
				(m-2-2) edge node [swap] {$b_1$} (m-1-2)
				(m-2-2) edge node {$b_2$} (m-3-2)
				
				(m-2-3) edge node [swap] {$d$} (m-1-2)
				(m-2-3) edge node [swap] {$d'$} (m-2-2)
				(m-2-3) edge node {$\widehat d$} (m-3-2);
		\end{tikzpicture}
	\end{center}
\end{lemma}

\begin{proof}
	Since diagram (\ref{dia:span-rewrite-compat}) is a span of boundary maps, it restricts to a smaller span
	\[ G \overset{b_1'}{\longleftarrow} \In(G)\cong \In(G[\rwsubst{L}{R}]_m) \overset{b_2'}{\longrightarrow} G[\rwsubst{L}{R}]_m  \]
	where $b_1'$ and $b_2'$ are monos. Since the image of $d$ is contained in the image of $b_1'$, it factors uniquely through $b_1'$ as $d = b_1' \circ d'$. Furthermore, $\hat d := b_2' \circ d'$ is the unique map making the above diagram commute.
	The construction follows for $c$ similarly.
\end{proof}

\begin{definition}[Rewriting on Cospans] \label{def:rewrite-cospan}
	For a cospan
	\[ \mathcal G :=
	 Y \overset{c}{\longrightarrow} G \overset{d}{\longleftarrow} X
	\]
	in $\DCsp(\catOGraphTG)$, and a matching $m$ of a rewrite $\rwsubst{L}{R}$ on $G$, we write $\mathcal G[\rwsubst{L}{R}]_m$ for the cospan over $G[\rwsubst{L}{R}]_m$ defined by Lem \ref{thm:rewriting-on-cospans}.
\end{definition}

\begin{theorem}\label{thm:cospan-rewrite-compatiblity}
	Let $\mathcal G : A \rightarrow B$, $\mathcal H : B \rightarrow C$ be cospans in $\DCsp(\catOGraphTG)$, and $m$ be a matching of a rule $\rwrule{L}{R}$ on $G$. Then there exists a matching $m'$ on $\mathcal H \circ \mathcal G$ such that
	\[ \mathcal H \circ (\mathcal G[\rwsubst{L}{R}]_m) \cong
	(\mathcal H \circ \mathcal G)[\rwsubst{L}{R}]_{m'} \]
	Similarly, for any cospan matching $n$ on $\mathcal H$, there exists $n'$ such that
	\[ (\mathcal H[\rwsubst{L}{R}]_n) \circ \mathcal G \cong
	(\mathcal H \circ \mathcal G)[\rwsubst{L}{R}]_{n'} \]
\end{theorem}

\begin{proof}
	The result follows from Thm \ref{thm:plugging-rewrite-compat} and Lem \ref{thm:rewriting-on-cospans}. In both cases, $m'$ and $n'$ are formed by composing the original mapping with the inclusion of the matched graph into the ($S'$-adhesive) pushout.
\end{proof}

Let $\mathbb S$ be a set of rewrite rules. We write $\mathcal G \rewritesto \mathcal H$ if there exists a rule $\rwsubst{L}{R}$ in $\mathbb S$ and a cospan matching $m$ such that $\mathcal G[\rwsubst{L}{R}]_m \cong \mathcal H$. Let $\rewriteequiv$ be the closure of $\rewritesto$ as an equivalence relation.

Let $\rewriteCat{\mathbb S}$ be the category whose objects are the same as those of $\DCsp(\catOGraphTG)$ and whose arrows are equivalence classes of cospans under the relation $\rewriteequiv$. This category is well-defined because of Thm \ref{thm:cospan-rewrite-compatiblity}, and inherits its symmetric monoidal structure from $\DCsp(\catOGraphTG)$.

Let $\mathbb H$ be the typed version of the edge homeomorphism rewrite system from Def \ref{def:homeo-rewrite-system}. This system consists of a line contraction rule $H_L(o)$ and a circle contraction rule $H_C(o)$ for each $o \in O$. It also has an input contraction rule $H_T^k(a)$ for each $a \in A$ and each input $k \in 1..N$ defined by $T(a)$, and similarly an output contraction rule $H_S^k(a)$. Note that when $A$ and $O$ are finite, this rewrite system is finite, unlike in the untyped case, where it is countably infinite.

A particularly important example of a rewrite category is then $\rewriteCat{\mathbb H}$. Arrows in this category correspond exactly to diagrammatic representations of morphisms in a symmetric monoidal category. Since categories of this form exhibit only the identities of various kinds of monoidal categories, they define \emph{free} categories over a graphical signature $T$.

\subsection{Free Monoidal Categories}

A monoidal precategory $\mathcal P$ consists of a class of objects $\textrm{ob} \mathcal P$.

\begin{definition}[Monoidal Precategory]
	Fix a class $O$ and form the free monoid $\textrm{ob} \mathcal M := O^*$ of words in $O$. A \emph{monoidal precategory} is a class of objects $\textrm{ob} \mathcal M$ and for every pair $v,w \in O^*$ a set $\hom(v,w)$ of arrows. A monoidal prefunctor $F : \mathcal M \rightarrow \mathcal N$ consists of a monoid homomorphism $\textrm{ob} \mathcal M \rightarrow \textrm{ob} \mathcal N$ and for every hom-set a function $\hom(v, w) \rightarrow \hom(Fv, Fw)$. The category of monoidal precategories and monoidal prefunctors is called $\catMonPreCat$.
\end{definition}

Note that monoidal precategories do not necessarily have composition
or identities, and the ``monoidal product'' is only defined for
objects. Monoidal categories and graphical signatures are both cases
of monoidal precategories. In the case of a graphical signature $T : A \rightarrow O^* \times O^*$, the class of objects is $O$ and for any pair of words $v,w \in O^*$, the hom-set is formed from the inverse image of $T$:
\[\hom(v,w) := T^{-1}(v,w) \subseteq A. \]

\begin{definitions}
	Let $\textbf{TSMC}(T) := \rewriteCat{\mathbb H}$. Let $\textbf{SMC}(T)$ be the subcategory of $\textbf{TSMC}(T)$ where every graph in the middle of a cospan is directed acyclic.
\end{definitions}

$\textbf{SMC}(T)$ has the property that no graphs contain ``feedback loops''. Note that $T$, as a monoidal precategory, embeds canonically into $\textbf{SMC}(T)$, and hence into $\textbf{TSMC}(T)$.

\begin{theorem}\label{thm:free-smc}
	$\textbf{SMC}(T)$ is the free symmetric monoidal category of $T$. That is, for any symmetric monoidal category $\mathcal V$, any monoidal prefunctor $F : T \rightarrow \mathcal V$ extends uniquely to a symmetric monoidal functor from $\textbf{SMC}(T)$. For the embedding of $T \hookrightarrow \textbf{SMC}(T)$, there exists a unique monoidal functor $\hat F$ making the following diagram commute.

	\begin{center}
		\begin{tikzpicture}[-latex]
			\matrix (m) [cdiag] {
			       T          & \mathcal V \\
			  \textbf{SMC}(T) &            \\
			};
			\path [arrs] (m-1-1) edge node {$F$} (m-1-2)
			      (m-2-1) edge [dashed] node [swap] {$\hat F$} (m-1-2)
			      (m-1-1) edge [right hook-latex] (m-2-1);
		\end{tikzpicture}
	\end{center}
\end{theorem}

We can prove the above theorem using the geometric characterisation of symmetric monoidal categories given in \cite{Joyal:1991p1143}. The details of this proof are given in Appendix \ref{app:proof-of-free-smc}.

\begin{definition}[Trace Operator] \label{def:trace-functional} For
  objects $A$, $B$ and $C$ of $\textbf{TSMC}(T)$, a
  \emph{trace operator} is defined to be a function $tr_{A,C}^B(-) : \hom_{\textbf{TSMC}(T)}(A \otimes B, C
  \otimes B) \rightarrow \hom_{\textbf{TSMC}(T)}(A, C)$.  Intuitively, this introduces edges that connect from $B$ in the codomain to $B$ in the domain.

  First, define the graph $L_B$ as a \typegraph-graph with points $B + B$ and exactly
  one edge connecting each $b \in B$ to its copy; $L_B$ is simply a
  collection of edges. Then form a cospan
	\[ B \overset{o}{\rightarrow} L_B \overset{i}{\leftarrow} B \]
	selecting the inputs and outputs of $L_B$.
	
	Let $[\mathcal G] : A \otimes B \rightarrow C \otimes B$ be an arrow
  in $\textbf{TSMC}(T)$, represented by a cospan $\mathcal G$. We can
  write its arrows as induced arrows of the coproducts, for $c_1 : C
  \rightarrow G$, $c_2 : B \rightarrow G$, etc.
	\[ C + B \overset{[c_1, c_2]}{\longrightarrow} G
	         \overset{[d_1, d_2]}{\longleftarrow} A + B \]
	
	Perform the follow (boundary-coherent) pushout of $G$ and $L_B$, for $[d_2,c_2]$ the induced map from the coproduct $B + B$.
	\begin{center}
		\posquare{B + B}{G}{L_B}{G'}{[d_2,c_2]}{p_2}{[o,i]}{p_1}
	\end{center}
	
	$p_1$ is mono because $[o,i]$ is, so let the following be a new cospan $\mathcal G'$.
	\[ C \overset{p_1 c_1}{\longrightarrow} G
	         \overset{p_1 d_1}{\longleftarrow} A \]
	Define $tr_{A,C}^B([\mathcal G]) := [\mathcal G']$.
\end{definition}

\begin{example} An illustration of applying a trace operator. 
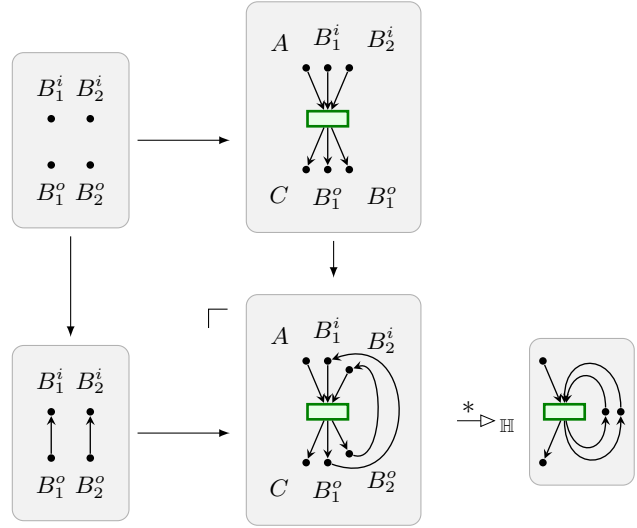
\begin{center} 
  \begin{tikzpicture}[mcgraph,node distance=0.2cm and 1cm]
{[node distance=0.5cm and 0.4cm]
	\matrix (m) [row sep=3em, column sep=5em,inner sep=0.5em] {
    { 
      \node (ki1) [ipoint] {};
      \node (ki1l) [ipointlabel, above=of ki1] {$B_1^i$};
      \node (ki2) [right=of ki1, ipoint] {};
      \node (ki2l) [ipointlabel, above=of ki2] {$B_2^i$};
      \node (ko1) [ipoint, below=of ki1] {};
      \node (ko1l) [ipointlabel, below=of ko1] {$B_1^o$};
      \node (ko2) [ipoint, below=of ki2] {};
      \node (ko2l) [ipointlabel, below=of ko2] {$B_2^o$};
    }
    & 
    {\node (Lp) {};
      \node (lg1) [andg, at=(Lp)] {$\quad$};
      \node (li1) [ipoint,above=of lg1.north west] {};
      \node (li1l) [ipointlabel, above left=of li1] {$A$};
      \node (li2) [above=of lg1.north, ipoint] {};
      \node (li2l) [ipointlabel, above=of li2] {$B_1^i$};
      \node (li3) [above=of lg1.north east, ipoint] {};
      \node (li3l) [ipointlabel, above right=of li3] {$B_2^i$};
      \node (lo1) [ipoint, below=of lg1.south west] {};
      \node (lo1l) [ipointlabel, below left=of lo1] {$C$};
      \node (lo2) [ipoint, below=of lg1.south] {};
      \node (lo2l) [ipointlabel, below=of lo2] {$B_1^o$};
      \node (lo3) [ipoint, below=of lg1.south east] {};
      \node (lo3l) [ipointlabel, below right=of lo3] {$B_1^o$};
      \path[dline] 
      (li1) edge (lg1)
      (li2) edge (lg1)
      (li3) edge (lg1)
      (lg1) edge (lo1)
      (lg1) edge (lo2)
      (lg1) edge (lo3);
    }
    \\
    { \node (ri1) [ipoint] {};
      \node (ri1l) [ipointlabel, above=of ri1] {$B_1^i$};
      \node (ri2) [right=of ri1, ipoint] {};
      \node (ri2l) [ipointlabel, above=of ri2] {$B_2^i$};
      \node (ro1) [ipoint, below=of ri1] {};
      \node (ro1l) [ipointlabel, below=of ro1] {$B_1^o$};
      \node (ro2) [ipoint, below=of ri2] {};
      \node (ro2l) [ipointlabel, below=of ro2] {$B_2^o$};
      \path[dline] 
      (ro1) edge (ri1)
      (ro2) edge (ri2);
    }
    & 
    {{ \node (mg1) [andg] {$\quad$};
      \node (mi1) [above=of mg1.north west, ipoint] {};
      \node (mi1l) [ipointlabel, above left=of mi1] {$A$};
      \node (mi2) [above=of mg1.north, ipoint] {};
      \node (mi2l) [ipointlabel, above=of mi2] {$B_1^i$};
      \node (mi3) [above=of mg1.east, ipoint] {};
      \node (mi3l) [ipointlabel, above right=of mi3] {$B_2^i$};
      \node (mo1) [ipoint, below=of mg1.south west] {};
      \node (mo1l) [ipointlabel, below left=of mo1] {$C$};
      \node (mo2) [ipoint, below=of mg1.south] {};
      \node (mo2l) [ipointlabel, below=of mo2] {$B_1^o$};
      \node (mo3) [ipoint, below=of mg1.east] {};
      \node (mo3l) [ipointlabel, below right=of mo3] {$B_2^o$};
      \node (x) [right=of mg1] {};

      \path[dline] 
      (mi1) edge (mg1)
      (mi2) edge (mg1)
      (mi3) edge (mg1)
      (mg1) edge (mo1)
      (mg1) edge (mo2)
      (mg1) edge (mo3);

      \draw[dline] (mo2) to [in=20,out=-20,looseness=2.3] (mi2);
      \draw[dline] (mo3) to [in=20,out=-20,looseness=1] (mi3);


    } 

    { \node (y) [node distance=0.5cm and 1cm,right=of x] {};
      \node (m2g1) [andg, right=of y] {$\quad$};
      \node (m2i1) [above=of m2g1.north west, ipoint] {};
      \node (m2i2) [above=of m2g1.north, halfe] {};
      \node (m2i3) [above=of m2g1.south east, halfe] {};
      \node (m2o1) [ipoint, below=of m2g1.south west] {};
      \node (m2o2) [halfe, below=of m2g1.south] {};
      \node (m2o3) [halfe, below=of m2g1.north east] {};
      \node (m2x) [ipoint,right=of m2g1] {};
      \node (m2xa) [node distance=0.5cm and 0.1cm, ipoint,right=of m2g1,xshift=1mm] {};

	\draw[dline] (m2g1) to [out=-90,in=-90,looseness=2.5] (m2x);
    \draw[dline] (m2x) to [in=90,out=90,looseness=2.5] (m2g1);
    \draw[dline] (m2g1) to [out=-80,in=-90,looseness=2.5] (m2xa);
    \draw[dline] (m2xa) to [in=80,out=90,looseness=2.5] (m2g1);

      
      \path[dline] 
      (m2i1) edge (m2g1)
      (m2g1) edge (m2o1);

    }

    }  \\
  };
}
  \begin{pgfonlayer}{background}
        \node[rectangle, fill=black!5, draw=black!30,rounded corners=1ex,
        fit=(ki1) (ki2) (ko1) (ko2) (ki1l) (ki2l) (ko1l) (ko2l)] (K) {};  
  \node[rectangle, fill=black!5, draw=black!30,rounded corners=1ex,
  fit=(lg1) (li1) (li2) (li3) (lo1) (lo2) (lo3) (li1l) (li2l) (li3l) (lo1l) (lo2l) (lo3l)] (L) {};  
  \node[rectangle, fill=black!5, draw=black!30,rounded corners=1ex, fit=(ri1) (ri2) (ro1) (ro2) (ri1l) (ri2l) (ro1l) (ro2l)] (R) {};  
  \node[rectangle, fill=black!5, draw=black!30,rounded corners=1ex,
  fit=(mg1) (mi1) (mi2) (mi3) (mo1) (mo2) (mo3) (x) (mi1l) (mi2l) (mi3l) (mo1l) (mo2l) (mo3l)] (M) {};  
  \end{pgfonlayer}
  \begin{pgfonlayer}{background}
  \node[rectangle, fill=black!5, draw=black!30,rounded corners=1ex,
  fit=(m2g1) (m2i1) (m2i2) (m2i3) (m2o1) (m2o2) (m2o3) (m2x)] (M2) {};  
  \end{pgfonlayer}

  \node [at=(y)] () {$\rewritetrans_{\mathbb{H}}\ \ \ $};

  \path[->,shorten <= 3pt,shorten >= 3pt,arrs] 
  (K) edge node[above] {} (R)
  (K) edge node[above] {} ($(K) + (2.2cm,0)$)
  (R) edge node[above] {} ($(R) + (2.2cm,0)$)
  ($(L.south) - (0,0)$) edge node[above] {} ($(L.south) - (0,0.7cm)$);

		\NWbracket{($(M.north east) + (-2.3cm,-0.7cm)$)};

  \end{tikzpicture}
\end{center}
\end{example}

This provides a natural way to work with traced symmetric monoidal
categories, and subsequently compact closed categories. We conjecture that this is, in fact, a free construction of traced symmetric monoidal categories.

\begin{conjecture}\label{thm:free-traced}
	The trace functional as in \ref{def:trace-functional} gives
  $\textbf{TSMC}(T)$ the structure of the free traced symmetric
  monoidal category over $T$.
\end{conjecture}

\subsection{PROPs}
\label{sub:props}

PROPs, or PROduct categories with Permutations, are a convenient way of describing symmetric monoidal algebraic structures \emph{internal} to some monoidal category $\mathcal V$.

\begin{definition}\label{def:prop}
	A PROP is a symmetric monoidal category whose objects are the natural numbers where the tensor product is given by addition.
\end{definition}

Examples of PROPs are the (skeletal) category $\mathbb F$ of finite sets and functions, $\textbf{Csp}(\mathbb F)$ of cospans of finite sets, with composition as pushout, $\textbf{Mat}(\mathbb N)$ whose objects are natural numbers $m$, $n$ and whose arrows are $m \times n$ matrices of natural numbers and $\textbf{Mat}(\mathbb Z)$ the same for integers.

PROPs are interesting because they define categories of algebras.

\begin{definition}\label{def:P-alg}
	For a PROP $\mathbb P$ and some fixed symmetric monoidal category $\mathcal V$, the category $\mathbb P$-Alg of $\mathbb P$-algebras has as objects strict symmetric monoidal functors $\mathbb P \rightarrow \mathcal V$ and has as arrows monoidal natural transformations.
\end{definition}

As their name suggests, algebras of PROPs represent internal algebraic structures. For instance, the algebras of $\mathbb F$, $\textbf{Csp}(\mathbb F)$, $\textbf{Mat}(\mathbb N)$, and $\textbf{Mat}(\mathbb Z)$ in $\mathcal V$ are internal monoids, special Frobenius algebras, bialgebras, and Hopf algebras respectively.

PROPs can be combined with each other in much the same way as monads using \emph{distributive laws} \cite{Lack:2004p1160}, and even more flexible \emph{interaction theories}, like the one used in \cite{CoeckeDuncan2009}.

A rich class of PROPs can be obtained from the rewrite categories defined in \S\ref{sub:rewrite-categories}. Consider a typed graph category made from a ``single-sorted'' graphical signature,
\[T : A \rightarrow \{ \bullet \}^* \times \{ \bullet \}^*.\]
Then, the objects of $\DCsp(\catOGraphTG)$ are point graphs containing $n$ isolated points of type ``$\bullet$'', which we can represent by the natural numbers. Since the monoidal product on objects is the disjoint union, $m \otimes n = m + n$.

Let $\mathbb E$ be some graphical theory, expressed as a rewrite system. Then, for $\mathbb H$ the edge-homeomorphism rewrite system, we can form the combined system $\mathbb E + \mathbb H$, and the rewrite category
\[ \mathcal E := \rewriteCat{(\mathbb E + \mathbb H)}. \]

The algebras of $\mathcal E$ will be structures in $\mathcal V$ that satisfy precisely the identities given graphically by $E$. By expanding the graphical signature $T$, this procedure generalises naturally from PROPs to multi-sorted monoidal theories. Taking $\mathcal V$ to be some concrete category like $\catVect_{\mathbb{C}}$, this formalises the notion of \emph{concrete models} for some graphical theory.

\section{Conclusions and Further Work}
\label{sec:conclusions}

We have presented a theory of open-graphs to support graphical reasoning about computational processes.  These graphs are visualised with an interface made of half-edges that enter or leave the graph. We formalised this by introducing a notion of intermediate points that occur along an edge or ``wire''. This allows a single wire to be cut into arbitrarily many smaller wires, and conversely supports composition by plugging wires together. Methods to support graphical rewriting, using the so called double pushout approach, have also been described, and it has been shown how graphical rewriting rules can themselves be composed. We then formalised the relationship between graphs that are ``semantically'' the same by defining a graph rewrite system called \emph{edge-homeomorphism}, by analogy to homeomorphism in topological spaces. 

Next, we generalised our construction of open-graphs to work with many \emph{types} of vertices and wires. This makes parameterises open-graphs by a  graphical signature with provides the typing rules for how graphs can be composed. This lets us express many kinds of processes, notably those with distinguished inputs and outputs. Building on graphical signatures, we then showed that cospans over typed open-graphs, modulo edge-homeomorphism, form free symmetric monoidal categories over a set of generators. By taking richer rewrite systems, we can obtain a large and interesting class of monoidal theory categories, including PROPs. Therefore, we have a fully general method of reasoning about models of graphical theories. 

The constructions presented here have deliberately been kept finitary and decidable for the case of finite open-graphs. This is with an eye to implementation of graphical reasoning software which would form a conceptual bridge to let us enjoy the intuitive power of graphical languages, while benefiting from rigorous, computer-assisted manipulation. In particular, our theory provides a platform for bringing techniques from rewriting, such as critical pair analysis and Knuth-Bendix completion \cite{knuth-bendix}, to process-centric graphical languages and monoidal categories. An implementation of this work is already largely completed\footnote{\url{http://dream.inf.ed.ac.uk/projects/quantomatic}.}, although a proof that this does indeed implement the theory presented here is future work. Another area of further work is to extend this formalism to support pattern-graphs, as introduced in~\cite{2009:DixonDuncan:AMAI}. More generally, we would like to be able to reason with graphs and rules that contain repeated or recursive structure.


Our construction of PROPs and free symmetric monoidal categories is also only the beginning. The graphical notation for traced symmetric monoidal categories introduced in e.g.~\cite{selinger2009survey} gives us strong reason to believe that Conjecture~\ref{thm:free-traced} is correct. For a suitable notion of traced monoidal theories, generalising the definition of PROPs to the traced setting, we believe that the construction in \S\ref{sub:props} actually forms the \emph{free} traced monoidal theory satisfying the equations reflected by a rewrite system.

On a more fundamental level, the construction of edge-points and edge homeomorphism suggests a deep and telling connection to not only topological graphs, but their more exotic cousins, topological \emph{directed graphs}. This has heretofore only been explored in an ad hoc manner, but we believe it can be made fully formal using the notions of directed topological spaces, as presented by \cite{grandis2009directed} or \cite{Krishnan2009}. We feel that, in the context of such a presentation, the technical content of this paper will arise naturally as a discreet reflection of the deeper, topological theory.

\bibliographystyle{apalike}
\bibliography{bibfile}

\begin{thebibliography}{}

\bibitem[Appelgate et~al., 1969]{Beck1969}
Appelgate, H., Barr, M., Beck, J., Lawvere, F., Linton, F., Manes, E., Tierney,
  M., and Ulmer, F. (1969).
\newblock Distributive laws.
\newblock In {\em Seminar on Triples and Categorical Homology Theory},
  volume~80 of {\em Lecture Notes in Mathematics}, pages 119--140. Springer
  Berlin / Heidelberg.
\newblock 10.1007/BFb0083084.

\bibitem[Baldan et~al., 2008]{Baldan-gts}
Baldan, P., Corradini, A., and K\"{o}nig, B. (2008).
\newblock Unfolding graph transformation systems: Theory and applications to
  verification.
\newblock In {\em Concurrency, Graphs and Models: Essays Dedicated to Ugo
  Montanari on the Occasion of His 65th Birthday}, pages 16--36.
  Springer-Verlag, Berlin, Heidelberg.

\bibitem[Bob~Coecke, 2010]{CoeckeKissinger2010}
Bob~Coecke, A.~K. (2010).
\newblock The compositional structure of multipartite quantum entanglement.
\newblock arXiv:1002.2540v2 [quant-ph].

\bibitem[Coecke and Duncan, 2008]{Coecke:2008jo}
Coecke, B. and Duncan, R. (2008).
\newblock Interacting quantum observables.
\newblock In {\em ICALP 2008}. LNCS.

\bibitem[Coecke and Duncan, 2009]{CoeckeDuncan2009}
Coecke, B. and Duncan, R. (2009).
\newblock Interacting quantum observables: Categorical algebra and
  diagrammatics.
\newblock arXiv:0906.4725v1 [quant-ph].

\bibitem[Danos and Laneve, 2004]{danos2004formal}
Danos, V. and Laneve, C. (2004).
\newblock {Formal molecular biology}.
\newblock {\em Theoretical Computer Science}, 325(1):69--110.

\bibitem[Dixon and Duncan, 2009]{2009:DixonDuncan:AMAI}
Dixon, L. and Duncan, R. (2009).
\newblock Graphical reasoning in compact closed categories for quantum
  computation.
\newblock {\em AMAI}, 56(1):20.

\bibitem[Dixon et~al., 2010]{2010:DDK:DCM}
Dixon, L., Duncan, R., and Kissinger, A. (2010).
\newblock Open graphs and computational reasoning.
\newblock In {\em Proceedings of {DCM}'10}, volume~26, pages 169--180. EPTCS.

\bibitem[Ehrig et~al., 2006]{Ehrig:Book:2006}
Ehrig, H., Ehrig, K., Prange, U., and Taentzer, G. (2006).
\newblock {\em Fundamentals of Algebraic Graph Transformation (Monographs in
  Theoretical Computer Science. EATCS Series)}.
\newblock Springer.

\bibitem[Ehrig et~al., 1973]{Ehrig1973}
Ehrig, H., Pfender, M., and Schneider, H.~J. (1973).
\newblock Graph-grammars: An algebraic approach.
\newblock In {\em 14th Annual Symposium on Switching and Automata Theory},
  pages 167--180. IEEE.

\bibitem[Girard, 1996]{Girard:proof-nets:96}
Girard, J.-Y. (1996).
\newblock Proof-nets: The parallel syntax for proof-theory.
\newblock In {\em Logic and Algebra}, pages 97--124. Marcel Dekker.

\bibitem[Grandis, 2009]{grandis2009directed}
Grandis, M. (2009).
\newblock {\em {Directed Algebraic Topology: Models of Non-Reversible Worlds}}.
\newblock New Mathematical Monographs. Cambridge University Press.

\bibitem[Joyal and Street, 1991]{Joyal:1991p1143}
Joyal, A. and Street, R. (1991).
\newblock The geometry of tensor calculus {I}.
\newblock {\em Advances in Mathematics}, 88:55--113.

\bibitem[Knuth and Bendix, 1970]{knuth-bendix}
Knuth, D.~E. and Bendix, P.~B. (1970).
\newblock {Simple word problems in universal algebras}.
\newblock In {\em Computational Problems in Abstract Algebra}, pages 263--297.
  Pergamon Press.

\bibitem[Krishnan, 2009]{Krishnan2009}
Krishnan, S. (2009).
\newblock A convenient category of locally preordered spaces.
\newblock {\em Applied Categorical Structures}, 17:445--466.
\newblock 10.1007/s10485-008-9140-9.

\bibitem[Lack, 2004]{Lack:2004p1160}
Lack, S. (2004).
\newblock Composing props.
\newblock {\em Theory and Applications of Categories}, 13(9):147--163.

\bibitem[Lack and Sobocinski, 2005]{Lack:2005lr}
Lack, S. and Sobocinski, P. (2005).
\newblock Adhesive and quasiadhesive categories.
\newblock {\em Theoretical Informatics and Applications}, 39(2):522--546.

\bibitem[Lafont, 1990]{Lafont1990}
Lafont, Y. (1990).
\newblock Interaction nets.
\newblock In {\em POPL '90: Proceedings of the 17th ACM SIGPLAN-SIGACT
  symposium on Principles of programming languages}, pages 95--108, New York,
  NY, USA. ACM.

\bibitem[Lafont, 2003]{Lafont2003}
Lafont, Y. (2003).
\newblock Towards an algebraic theory of boolean circuits.
\newblock {\em Journal of Pure and Applied Algebra}, 184(2-3):257 -- 310.

\bibitem[Lafont, 2010]{Lafont09diagramrewriting}
Lafont, Y. (2010).
\newblock Diagram rewriting and operads.
\newblock In Lecture Notes from the Thematic school : Operads CIRM, Luminy
  (Marseille), 20-25 April 2009.

\bibitem[Lafont and Rannou, 2008]{Lafont08}
Lafont, Y. and Rannou, P. (2008).
\newblock Diagram rewriting for orthogonal matrices: A study of critical peaks.
\newblock In {\em RTA'08}, pages 232--245, Berlin, Heidelberg. Springer-Verlag.

\bibitem[Milner, 2006]{milner2006pure}
Milner, R. (2006).
\newblock {Pure bigraphs: Structure and dynamics}.
\newblock {\em Information and computation}, 204(1):60--122.

\bibitem[Penrose, 1971]{Penrose1971Applications-of}
Penrose, R. (1971).
\newblock Applications of negative dimensional tensors.
\newblock In {\em Combinatorial Mathematics and its Applications}, pages
  221--244. Academic Press.

\bibitem[Prange et~al., 2008]{Ehrig2008}
Prange, U., Ehrig, H., and Lambers, L. (2008).
\newblock Construction and properties of adhesive and weak adhesive high-level
  replacement categories.
\newblock {\em Applications of Categorical Structures}, 16:365--388.

\bibitem[Selinger, 2009]{selinger2009survey}
Selinger, P. (2009).
\newblock {A survey of graphical languages for monoidal categories}.
\newblock {\em New Structures for Physics}, pages 275--337.

\end{thebibliography}

\newpage

\appendix

\section{Proof of Freeness for $\textbf{SMC}(T)$}
\label{app:proof-of-free-smc}

We shall prove Thm \ref{thm:free-smc} using the geometric characterisation of symmetric monoidal categories given in \cite{Joyal:1991p1143}.

\begin{theorem}[\ref{thm:free-smc}]
	$\textbf{SMC}(T)$ is the free symmetric monoidal category of $T$. That is, for any symmetric monoidal category $\mathcal V$, any monoidal prefunctor $F : T \rightarrow \mathcal V$ extends uniquely to a symmetric monoidal functor from $\textbf{SMC}(T)$. For the embedding of $T \hookrightarrow \textbf{SMC}(T)$, there exists a unique monoidal functor $\hat F$ making the following diagram commute.

	\begin{center}
		\begin{tikzpicture}[-latex]
			\matrix (m) [cdiag] {
			       T          & \mathcal V \\
			  \textbf{SMC}(T) &            \\
			};
			\path [arrs] (m-1-1) edge node {$F$} (m-1-2)
			      (m-2-1) edge [dashed] node [swap] {$\hat F$} (m-1-2)
			      (m-1-1) edge [right hook-latex] (m-2-1);
		\end{tikzpicture}
	\end{center}
\end{theorem}

First, we recall several definitions from \cite{Joyal:1991p1143}.

\begin{definition}[Generalised Topological Graph] \label{def:genalised-top-graph}
	A \emph{generalised topological graph} is a pair $(G, G_0)$, where $G$ is a Hausdorff space and $G_0$ is a discreet, closed subset where $G - G_0$ is isomorphic to a sum of open intervals $I_o := (0,1)\subseteq \mathbb R$ and copies of $S_1$. The compactification of an open interval $I_o \subseteq G - G_0$ is called an \emph{edge} $\hat e$. A copy of $S_1 \subseteq G - G_0$ is called a \emph{circle} $\hat c$.
\end{definition}

Note that all edges naturally embed in the compactification $\hat G \supseteq G$ obtained by adding endpoints to open edges.

\begin{definition}[Polarised Graph] \label{def:polarised-graph}
	A \emph{polarised graph} is a tuple $\Gamma := (G, G_0, \omega, \pi)$, where $\omega$ assigns each each $\hat e$ and each circle $\hat c$ in $(G, G_0)$ an orientation. We can therefore define an input $\hat e(0)$ and an output $\hat e(1)$ for each edge. For each vertex $v \in G_0$, $\textrm{in}(v)$ is the set of edges such that $\hat e(1) = v$ and $\textrm{out}(v)$ is the set of edges such that $\hat e(0) = v$. $\pi$ then assigns to each $v$ a total order on $\textrm{in}(v)$ and $\textrm{out}(v)$, called a \emph{polarisation}. Also, a polarised graph that contains no directed cycles is called \emph{progressive}.
\end{definition}

Polarised graphs come with a notion of boundary. We can furthermore put an ordering on this boundary.

\begin{definition}[Boundary of a polarised graph] \label{def:anchored-graph}
	For a polarised graph $\Gamma := (G, G_0, \omega, \pi)$, $\partial \Gamma := \hat G - G$ is a discreet space called the \emph{boundary} of $\Gamma$. Points in $\partial \Gamma$ that are the input of some edge are called \emph{inputs} of $\Gamma$, and outputs of edges in $\partial \Gamma$ are called \emph{outputs} of $\Gamma$. A polarised graph with a pair of total orders $\beta_0$ on its inputs and $\beta_1$ on its outputs is called an \emph{anchored graph}.
\end{definition}

\begin{definition}[Valuation] \label{def:valuation}
	For an anchored graph $\Gamma$ and a monoidal precategory $\mathcal M$, a \emph{valuation} $v$ of $\Gamma$ is a function $v_0$ that assigns an object of $\mathcal M$ to every edge in $\Gamma$ and a function $v_1$ that assigns an arrow to every vertex in such a way that respects the domain on codomain of arrows in $\mathcal M$. A map of anchored graphs with valuations $(\Gamma, v) \rightarrow (\Gamma', v')$ is a collection of maps that respect all of the structure of $\Gamma$ and the valuations.
\end{definition}

Since an anchored graph gives a total order to inputs and outputs, we can associate input and output words to a pair $(\Gamma, v)$. Let $T : A \rightarrow O^* \times O^*$ be a graphical signature. $\mathbb F_S(T)$ is the category whose objects words in $O^*$. For words $v$ and $w$, arrows are isomorphism classes of progressive anchored graphs with valuations into $T$ that have input word $v$ and output word $w$.

It was shown in \cite{Joyal:1991p1143} that $\mathbb F_S(T)$ is the free symmetric monoidal category over $T$. For the proof of theorem \ref{thm:free-smc} it suffices to show that a symmetric monoidal equivalence exists from $\textbf{SMC}(T)$ to $\mathbb F_S(T)$.

We can now prove Thm \ref{thm:free-smc} by defining a geometric realisation functor $\llbracket - \rrbracket_T : \textbf{SMC}(T) \rightarrow \mathbb F_S(T)$ that is identity-on-objects and showing it admits a (weak) inverse.

\begin{proof}
Let $\mathcal G : X \rightarrow Y$ be an arrow in $\textbf{SMC}(T)$. Choose a directed cospan \( Y \overset{c}{\longrightarrow} G \overset{d}{\longleftarrow} X \) of \typegraph-graphs to represent the equivalence class $\mathcal G$.

The category \catGraph{} sits inside the category of simplicial complexes, so there is a geometric realisation functor $\llbracket - \rrbracket : \catGraph \rightarrow \catTop$.

$G$ is an element of the slice category over \typegraph, so it comes with a map $\tau_G : G \rightarrow \typegraph$. The underlying graph of $G$ has an embedding of its boundary and its set of vertices. That is, there exist maps $b : X + Y \rightarrow G$ and $v : V \rightarrow G$ in \catGraph, where $X + Y$ and $V$ are discreet graphs.

For $H := \llbracket G \rrbracket - \llbracket X + Y \rrbracket$ and $H_0 := \llbracket V \rrbracket$, $(H, H_0)$ defines a generalised topological graph. Note that the compactification $\hat H = \llbracket G \rrbracket$. Since each edge (or circle) in $\hat H$ has an underlying directed chain (or cycle) of edge points, we can equip it with an orientation $\omega$. Recall that edges adjacent to a vertex in \typegraph{} have a natural total order given by their word order in $T$. We can use this order to assign a polarisation $\pi$ to the vertices in $H_0$. Thus $(H, H_0, \omega, \pi)$ defines a polarised graph. It is progressive precisely because $G$ is directed-acyclic. The total order on $X$ and $Y$ induce a total order on the inputs and outputs of $G$, and hence total orders $\beta_i, \beta_o$ on the inputs and outputs of the polarised graph. Thus $\Gamma = (H, H_0, \omega, \pi, \beta)$ is a progressive anchored graph. For $\Gamma$, a valuation $v$ into $T$ can clearly be deduced by the typing map $\tau_G : G \rightarrow \typegraph$, so $(\Gamma, v)$ is an arrow in $\mathbb F_S(T)$.

Let $\Gamma'$ be the result of performing this construction on some other $G'$ representing $\mathcal G$. Then $G$ could be rewritten to $G'$ by only merging or subdividing edges. The only step of the construction that makes explicit reference to (internal) edge-points is the application of $\llbracket - \rrbracket : \catGraph \rightarrow \catTop$ to the underlying graphs of $G$ and $G'$. This process forgets edge points, so $\Gamma' \cong \Gamma$. Also, for any $G'$ that yields a progressive anchored graph $\Gamma' \cong \Gamma$, $G'$ is simply another triangularisation of $\Gamma$, so $G'$ rewrites to $G$ using edge-homeomorphism.

This construction respects composition and the symmetric monoidal structure, so $\llbracket \mathcal G \rrbracket_T = \Gamma$ defines a symmetric monoidal functor into $\mathbb F_S(T)$. Furthermore, $\llbracket - \rrbracket_T$ admits a weak inverse by sending a progressive anchored graph $\Gamma$ to the equivalence class $\mathcal G$ represented by \emph{any} $G$ such that the above construction performed on $G$ yields a progressive anchored graph $\Gamma' \cong \Gamma$.
\end{proof}

\end{document}